\documentclass[11pt,a4paper]{amsart}

\usepackage[utf8]{inputenc}
\usepackage[english]{babel}
\usepackage{amsmath}
\usepackage{amssymb}
\usepackage{amsfonts} 
\usepackage{amsthm}
\usepackage{mathrsfs}
\usepackage{graphicx}
\usepackage{enumitem}
\setlist[itemize]{leftmargin=*}
\usepackage{xcolor}
\usepackage{url}
\usepackage{accents}
\usepackage{bm}
\usepackage[all]{xy}
\usepackage{hyperref,doi}
\usepackage{caption}
\usepackage{subcaption}
\usepackage{booktabs}
\usepackage{mathtools}
\usepackage{etoolbox}
\usepackage{overpic}
\usepackage{xspace}
\usepackage{comment}
\usepackage{chngcntr}
\makeatletter
\newcommand{\boxedred}[1]{\textcolor{red}{\fbox{\normalcolor\m@th$#1$}}}
\makeatother
\makeatletter
\newcommand{\eqlabel}[1]{\refstepcounter{equation}\textup{\tagform@{\theequation}}\label{#1}}
\makeatother

\newcommand{\N}{\mathbb{N}}
\newcommand{\Z}{\mathbb{Z}}

\newcommand{\R}{\mathbb{R}}
\newcommand{\C}{\mathbb{C}}

\newcommand{\B}{\mathcal{B}}
\newcommand{\floor}[1]{\left\lfloor #1\right\rfloor}
\newcommand\restr[2]{{
		\left.\kern-\nulldelimiterspace
		#1
		\vphantom{\big|}
		\right|_{#2}
}}

\newcommand{\de}{\partial}

\newcommand{\mz}{\frac{1}{2}}

\newcommand{\ang}[1]{\left\langle#1\right\rangle}
\newcommand{\uno}{\bm{1}}
\newcommand{\supp}[1]{\operatorname{supp}\left(#1\right)}
\newcommand{\nin}{\not\in}

\newcommand{\weakto}{\rightharpoonup}
\newcommand{\weakstarto}{\stackrel{*}{\rightharpoonup}}

\newcommand{\fantasma}[1]{\leavevmode\phantom{#1}}

\newcommand{\cptsub}{\subset\hspace{-1pt}\subset}
\DeclareMathOperator{\sgn}{sgn}

\DeclareMathOperator{\diam}{diam}

\DeclareMathOperator{\dist}{dist}

\newtheoremstyle{plainstyle}
{\glueexpr\parskip*4\relax}                    
{\glueexpr\parskip*2\relax}                    
{}                   
{}                           
{\bfseries}                   
{.}                          
{.2em}                       
{\thmname{#1}\thmnumber{ #2}\thmnote{ (#3)}}  
\theoremstyle{plainstyle}

\newtheorem{definition}{Definition}[section]
\newtheorem{rmk}[definition]{Remark}

\newtheoremstyle{itstyle}
{\glueexpr\parskip*4\relax}                    
{\glueexpr\parskip*2\relax}                    
{\itshape}                   
{}                           
{\bfseries}                   
{.}                          
{.2em}                       
{\thmname{#1}\thmnumber{ #2}\thmnote{ (#3)}}  
\theoremstyle{itstyle}

\newtheorem{thm}[definition]{Theorem}
\newtheorem{lemmaen}[definition]{Lemma}
\newtheorem{corollary}[definition]{Corollary}
\newtheorem{proposition}[definition]{Proposition}
\newtheorem*{thm*}{Theorem}
\newtheorem*{corollary*}{Corollary}

\newcommand{\obar}[1]{\overline{#1}}

\newcommand{\set}[1]{\left\{#1\right\}}

\newcommand{\pa}[1]{\left(#1\right)}

\newcommand{\abs}[1]{\left|#1\right|}
\newcommand{\norm}[1]{\left\|#1\right\|}

\newcommand{\brapa}[1]{\left[#1\right)}
\newcommand{\pabra}[1]{\left(#1\right]}

\usepackage[a4paper,left=3cm,right=3cm,top=3.4cm,bottom=3.4cm]{geometry}
\usepackage{appendix}
\usepackage{scalerel}
\usepackage[usestackEOL]{stackengine}
\usepackage[stretch=30,shrink=30]{microtype}

\newcommand{\vfd}{\mathbf{v}}
\newcommand{\envdim}{q}
\newcommand{\subman}{{\mathcal{M}^m}}
\newcommand{\ball}{\mathbb{B}}
\newcommand{\good}{\mathcal{G}}
\newcommand{\goodrk}{{\mathcal{G}^f}}
\newcommand{\adm}{\mathcal{A}}
\newcommand{\enrat}{C'}
\newcommand{\massbd}{C''}

\renewcommand{\bar}{\obar}

\renewcommand{\tilde}{\widetilde}
\renewcommand{\epsilon}{\varepsilon}

\def\media{
	\,\ThisStyle{\ensurestackMath{%
			\stackinset{c}{.2\LMpt}{c}{.5\LMpt}{\SavedStyle-}{\SavedStyle\phantom{\int}}}%
		\setbox0=\hbox{$\SavedStyle\int\,$}\kern-\wd0}\int
}

\counterwithin{equation}{section}

\usepackage[strings]{underscore}
\usepackage[square,numbers]{natbib}


\begin{document}
	
	\raggedbottom

	\author[A.~Pigati]{Alessandro Pigati}

	\author[T.~Rivi{\`e}re]{Tristan Rivi{\`e}re}

	\thanks{(Pigati and Rivi{\`e}re) {\sc ETH Z\"urich, Department of Mathematics,
		R\"amistrasse 101, 8092 Z\"urich, Switzerland}. E-mail addresses: \href{mailto:alessandro.pigati@math.ethz.ch}{alessandro.pigati@math.ethz.ch} and \href{mailto:tristan.riviere@math.ethz.ch}{tristan.riviere@math.ethz.ch}.}

	\title[Regularity of parametrized stationary varifolds]{The regularity of parametrized integer \\ stationary varifolds in two dimensions}

	\begin{abstract}
		We establish an optimal regularity result 
		for parametrized two-dimensional stationary varifolds. Namely, we show that the parametrization map is a smooth minimal branched immersion and that the multiplicity function is constant.
		
		We provide some applications of this regularity result, especially in the calculus of variations for the area functional.
		%
	\end{abstract}

	\maketitle

	\section{Introduction}

	Geometric measure theory was born with the ambition of giving a general satisfactory solution to Plateau's problem. After the work of Douglas and Rad\'o in the two-dimensional case, the need was felt for an unparametrized approach, in order to overcome the difficulties arising in higher dimensions, seeking a weak notion of submanifold compatible with the calculus of variations. 
	A successful theory was proposed only in 1960 by Federer and Fleming, in \cite{fedfle}. Their \emph{theory of currents} merged the abstract homological framework of general de Rham's currents with the seminal analytic ideas of generalized hypersurface (thought as a boundary of a finite-perimeter set), proposed by Caccioppoli and De Giorgi, and of rectifiable set, introduced by Besicovitch and his school. Integral currents satisfy a suitable weak compactness property and their area, called \emph{mass}, is lower semicontinuous in the natural weak topology. These two crucial properties allow to use the so-called direct method to find a minimizer for the mass. The interior regularity for area minimizing currents in codimension one is due to the combined contributions of De Giorgi, Fleming, Almgren, Simons and Federer (see e.g. \cite[Chapter~7]{simon}), while in arbitrary codimension the optimal interior regularity has been achieved by the successive efforts from Almgren, De Lellis and Spadaro (see \cite{almbig,dls1,dls2,dls3}).
	
	However, this weak notion of submanifold allows for \emph{cancellation of mass} under weak limits. This phenomenon is not desirable in many variational problems, such as the min-max procedures to construct unstable critical points of the area, since in these instances one usually infers the nontriviality of the critical point from the value of its mass (one cannot rule out a trivial critical point by homotopical or homological means, as opposed to Plateau's problem). To overcome this difficulty, in his unpublished notes \cite{almmim} Almgren introduced the idea of \emph{varifold}, which is simply a Radon measure on the Grassmannian bundle of a manifold. Having a measure on this bundle, rather than just on the base space, is fundamental in order to define the correct notion of pushforward under a diffeomorphism, 
	thus obtaining a meaningful definition of stationarity extending the notion of minimal submanifold. As varifolds are just measures, their mass is retained under weak limits.
	
	An important class of varifolds is formed by the integer rectifiable ones, namely those varifolds which can be represented as a countable superposition, with positive integer coefficients, of $k$-rectifiable sets. This is the kind of varifolds which is most commonly used and studied, since it is more concrete than the general definition but still enjoys compactness properties. The fundamental reference on general varifolds is \cite{allard}, where the compactness of integer stationary varifolds (or more generally of integer varifolds with locally bounded first variation) is proved, along with their regularity on a dense open set, rectifiability criteria for general varifolds and other important estimates. We mention that the existence and regularity theory for varifolds arising from min-max problems, started by Almgren, was later taken over by Pitts, Rubinstein and Smith (see \cite{pitts,prexist,prappl,smith}).
	
	In this work we will consider integer stationary varifolds, in arbitrary codimension, admitting a parametrization in the following sense: they are induced by a weakly conformal map $\Phi\in W^{1,2}(\Sigma,\subman)$, where $\Sigma$ is a closed Riemann surface and $\subman\subseteq\R^\envdim$ is either a closed Riemannian manifold or $\R^\envdim$ itself, together with a multiplicity function $N\in L^\infty(\Sigma,\N\setminus\set{0})$ \emph{on the domain}. They are required to satisfy a natural stationarity property which is local in the domain: namely, we assume that, for almost all domains $\omega\subseteq\Sigma$, the varifold induced by the map $\restr{\Phi}{\omega}$ with the multiplicity function $\restr{N}{\omega}$ is stationary in the complement of the compact set $\Phi(\de\omega)$ (see Definition \ref{pardef} and Remark \ref{rectvar} for the precise definitions). Throughout the paper, such objects will be called \emph{parametrized stationary varifolds}.
	
	
	Our main result is the following, which appears in Corollary \ref{globgenreg} and Remark \ref{converse} in the body of the paper.

	\begin{thm*}
		The triple $(\Sigma,\Phi,N)$ is a parametrized stationary varifold, in a compact manifold $\subman\subseteq\R^\envdim$ or in $\R^\envdim$ itself, if and only if $\Phi$ is a smooth conformal harmonic map and $N$ is a.e. constant. In this case, $\Phi$ is a minimal branched immersion.
	\end{thm*}
	
	In fact we also get a local version of the result, which holds for \emph{local parametrized stationary varifolds}: see Definition \ref{locpardef} and Theorem \ref{genreg}.
	Let us stress the fact that the result holds in \emph{arbitrary codimension} and does not assume any \emph{stability} hypothesis on the image varifold.
	
	On the other hand, everywhere regularity for integer stationary varifolds does not hold without additional assumptions, not even in low dimension (see e.g. the picture below): Pitts was able to obtain the optimal regularity for varifolds arising via min-max in codimension one by introducing the stronger concept of \emph{almost minimizing} varifold. So far, the regularity theory of integer stationary varifolds is still very incomplete and well understood only in special cases. Some notable examples are the structure theorem by Allard--Almgren for one-dimensional varifolds (see \cite{allalm}), Allard's almost everywhere regularity result for the case of constant multiplicity (see \cite[Section~8]{allard}) and, for the stable, codimension one case, Schoen--Simon regularity theorem (see \cite{schoen}) under the assumption that the singular set has locally finite $\mathcal{H}^{n-2}$-measure, which was recently reduced to an optimal assumption in a monumental work by Wickramasekera (see \cite{wick}). Also, the standard strata composing the singular set of a stationary varifold are now known to be rectifiable (see \cite{naber}). The almost everywhere regularity in full generality is still an open problem.
	
	In the present situation, the parametrized structure of the varifold and the natural corresponding stationarity condition turn out to be a good replacement for the stable codimension one assumption (as well as other ad-hoc conditions to exclude singularities arising from self-intersecting minimal hypersurfaces) which was considered in much of the previous literature on the regularity theory. As it will be apparent in the paper, regularity stems from a subtle interaction between stationarity and the topological information of being parametrized.

	The possibility to localize the stationarity with respect to the domain is crucially used in many places in order to get our main regularity result. The weak conformality of $\Phi$ also happens to be important, since it ensures that the map $\Phi$ is holomorphic when the codomain is the plane, a fact which we establish in Section \ref{conicalregsec}: the peculiar properties of nonconstant holomorphic functions, namely that they are branched covering maps and that they cannot vanish to infinite order, turn out to be useful several times.

	Our main result relies on the previous treatment of the simpler situation where $N$ is constant (see \cite{rivtarget}, by the second author; see also Theorem \ref{constn} below). We point out that the starting point of the strategy of \cite{rivtarget}, i.e. the observation that Lebesgue points $z$ for $\Phi$ and $d\Phi$ (with $d\Phi(z)\neq 0$) belong to the regular set of $\Phi$, does not apply when $N$ is a priori not constant. This difficulty is due to the fact that we cannot invoke Allard's $\epsilon$-regularity result in this more general situation.

	We will actually show that $N$ is a.e. constant and $\Phi$ satisfies the harmonic map equation $-\Delta\Phi=A(\Phi)(\nabla\Phi,\nabla\Phi)$ in any local conformal chart. These two facts are essentially equivalent to each other, in view of the preliminary work \cite{rivtarget}, and in some situations it will be easier to establish the former, while in other instances (where Lemma \ref{removab} is invoked) the latter is more convenient. The smoothness of weak solutions to the harmonic map equation was first proved in full generality by H\'elein (see \cite[Section~4.1]{helein}) and then obtained again by the second author in \cite{rivharm}, as a consequence of a general result for linear elliptic systems with an antisymmetric potential. In our situation the smoothness of $\Phi$ is a classical fact, since for parametrized stationary varifolds the map $\Phi$ is easily seen to be continuous (see Proposition \ref{contandsupp}).

	We want to give a simple one-dimensional example illustrating why, in spite of the difficulties surrounding integer stationary varifolds, one should expect to get much more information from the notion of parametrized stationary varifold (and eventually the regularity). The picture depicts a portion of an integer 1-rectifiable stationary varifold in $S^2$ (e.g. the union of three geodesic segments joining the north and south poles, with an angle $\frac{2\pi}{3}$ between any two of them), with multiplicity 2 a.e. This varifold has a nontrivial singular set.

	This varifold can be parametrized by a map $\Phi:S^1\to S^2$, as shown in the picture. However, the parametrized varifold $(S^1,\Phi,1)$ fails to satisfy the local stationarity condition, as witnessed by the highlighted part on the right. Of course, this example is just a heuristic motivation for the hope to obtain the full regularity, as no effective structure theorem is known for two-dimensional stationary varifolds.

	\begin{center}
		\includegraphics[width=11cm]{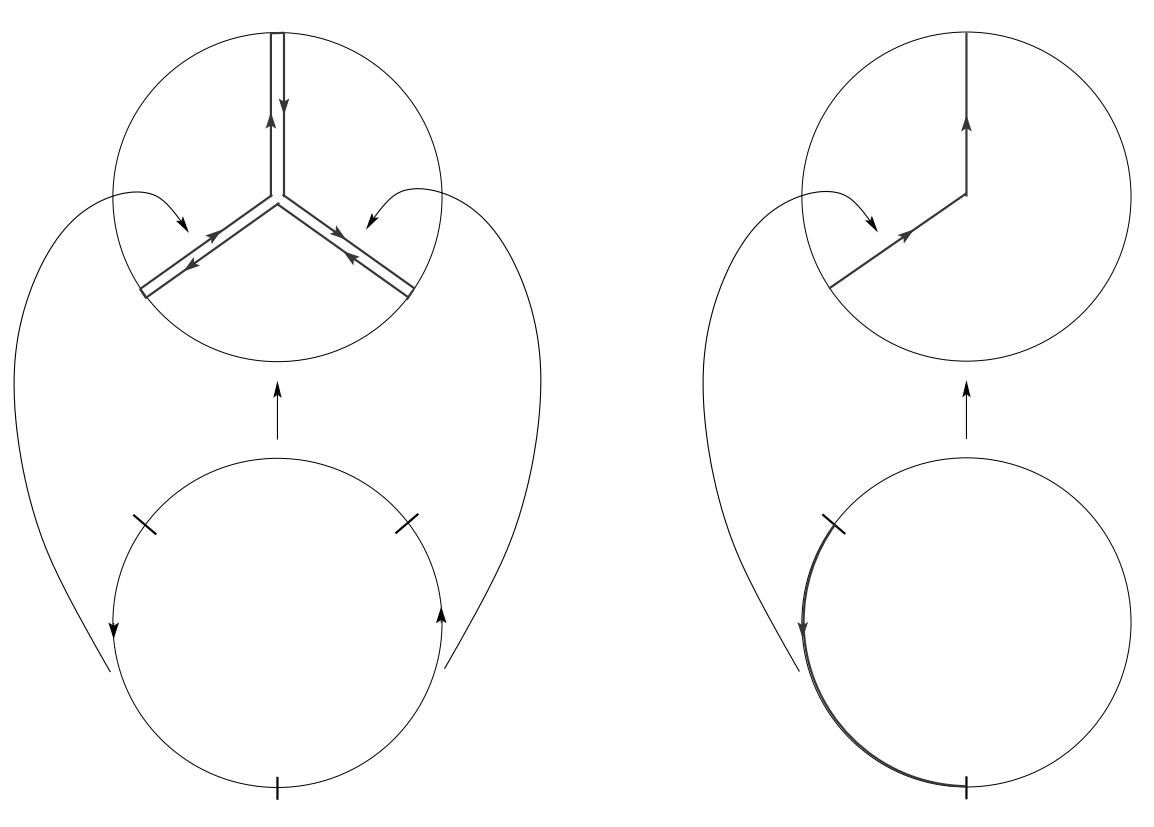}
	\end{center}

	As already said, in view of \cite{rivtarget}, the problem of showing that $N$ is a.e. constant is of the same difficulty. In this sense our result can be seen as a \emph{constancy theorem} in the domain, for the class of parametrized stationary varifolds. The classical constancy theorem asserts that a stationary varifold in a smooth connected manifold of the same dimension has constant density (see \cite[Theorem~41.1]{simon}). Showing that it holds in other settings and under weaker assumptions is a challenging topic in geometric measure theory: we just mention that its truth for stationary varifolds with support contained in a Lipschitz graph (of the same dimension) is a difficult result by Duggan (see \cite[Corollary~4.5]{duggan}).
	
	A natural question is how to produce parametrized stationary varifolds within the calculus of variations for the area functional. Indeed, one motivation for studying this class of varifolds comes from the previous work \cite{rivminmax} by the second author, who developed a new theory for the construction of immersed (possibly branched) minimal surfaces with arbitrary topology, by means of a viscosity method. Namely, one first finds an immersion $\Phi:\Sigma\to M$ which is critical for the perturbed area functional
	\[ A^{\sigma}(\Phi):=\int_\Sigma\,d\text{vol}+\sigma^2\int_\Sigma(1+\abs{A}^2)^p\,d\text{vol},\quad 2<p<\infty\text{ fixed}, \]
	where $\Sigma$ is a Riemann surface (chosen afterwards to make $\Phi$ conformal), the second $A$ is the second fundamental form of the immersion and $d\text{vol}$ is the volume form induced by $\Phi$. The coercivity of this functional ensures that it enjoys a sort of Palais--Smale condition up to diffeomorphisms. The second author shows that, choosing a sequence $\Phi_j$ of critical points of $E^{\sigma_j}$ satisfying a certain entropy condition and making the domains converge in the Deligne--Mumford compactification, then the surfaces $\Phi_j(\Sigma_j)$ converge precisely to a parametrized stationary varifold (up to bubbling and nodal points), with a pointwise multiplicity possibly bigger than 1.

	Bringing the main result of \cite{rivminmax} together with the one of this paper, one gets the following.
	
	\begin{corollary*}
		Let $\mathcal{F}\subseteq\mathcal{P}(\operatorname{Imm}(\Sigma,\subman))$ be a collection of families of $C^1$ immersions and assume that $\mathcal{F}$ is stable under isotopies (through homeomorphisms) of the space $\operatorname{Imm}(\Sigma,\subman)$. Then the min-max value
		\[ \inf_{F\in\mathcal{F}}\max_{\Phi\in F}\abs{\Phi(\Sigma)} \]
		is the area of a branched, possibly disconnected immersed minimal surface, with a locally constant integer multiplicity.
	\end{corollary*}

	Such families $\mathcal{F}$ arise for instance when considering sweep-outs by surfaces realizing some homotopy class: see e.g. \cite{cold} for examples and further discussion.
	The previous corollary should be put in perspective with the existing literature on min-max for surfaces as follows.
	
	\begin{itemize}
		\item For 1-parameter families of spheres, namely for a mountain pass min-max, some results about compactness and absence of energy loss were obtained by Jost: see e.g. \cite{Jos91} (which more generally deals with the min-max of harmonic maps). We also refer to the work by Jost--Struwe in \cite{JS90} for related results in the context of Plateau's problem in the Euclidean space.
		\item For 1-parameter families of spheres, Colding and Minicozzi in \cite{cm} proved the existence of a minimal sphere, using an ingenious harmonic replacement technique.
		\item Again for 1-parameter families, Colding--Minicozzi approach for finding minimal spheres by means of harmonic maps was recently extended by Zhou in \cite{zhou} to arbitrary genus, by adding the conformal class as an additional variable in the variational problem.
		\item In codimension one, the aforementioned Almgren--Pitts theory allows to construct smooth embedded minimal (hyper)surfaces from a min-max in a suitable space of cycles.  This approach was then refined and adapted to other settings (see e.g. \cite{ramic} and \cite{marnev}) and has been responsible for many important applications (see e.g. \cite{marnevsurvey}). For a survey on a simplified approach, which is a combination of the contributions of Almgren, Meeks, Pitts, Simon, Smith and Yau, we refer the reader to the paper \cite{cold} by Colding and De Lellis (which deals for simplicity with 1-parameter sweepouts).
		\item Still in codimension one, another min-max approach based on the link between minimal surfaces and phase transitions was recently proposed by Guaraco in \cite{Gua15}.
		It allows to interpret certain sweepouts in terms of level sets of functions and, in three dimensions, it allows to obtain multiplicity one and the expected Morse index for the resulting minimal surface: this was recently shown by Chodosh and Mantoulidis for generic metrics in \cite{chma}.
	\end{itemize}
	

	During the investigation of this problem, as a potential intermediate step towards the regularity, we asked ourselves whether any conformal solution $\Phi\in W^{1,2}(B_1^2(0),\R^\envdim)$ to the so-called \emph{conductivity equation} $-\operatorname{div}(N\nabla\Phi)=0$ (for some bounded measurable $N$ with values in positive integers) is necessarily harmonic. Actually, we can give a positive answer to this question as a consequence of the main theorem.
	\begin{corollary*}
		Assume $\Phi\in W^{1,2}(B_1^2(0),\R^\envdim)$ is weakly conformal, $N\in L^\infty(B_1^2(0),\N\setminus\set{0})$ and
		\[ -\operatorname{div}(N\nabla\Phi)=0\quad\text{in }\mathcal D'(B_1^2(0),\R^\envdim). \]
		Then $\Delta\Phi=0$ and, if $\Phi$ is nonconstant, $N$ is a.e. constant.
	\end{corollary*}
	We refer to Theorem \ref{condthm} in the body of the paper. However, we are not aware of any purely analytic proof of this fact and leave it as an open problem to find such a proof. We are able to succeed in this task in the planar case $q=2$: see Theorem \ref{planarcond}.

	We end the introduction with a brief summary of the contents of the paper.
	\begin{itemize}
		\item In Section \ref{basicsec} we establish some basic facts about parametrized stationary varifolds: we show the continuity of the parametrization map $\Phi$ (see Proposition \ref{contandsupp}), we define an upper semicontinuous representative $\tilde N$ of the multiplicity function $N$ satisfying $\tilde N=N$ a.e. with respect to the measure $\abs{\nabla\Phi}^2\mathcal{L}^2$ (see Definition \ref{tildendef} and Proposition \ref{nisusc}) and we introduce a local notion of parametrized stationary varifold.
		\item In Section \ref{conicalregsec} we generalize the topological notion of triod (first introduced by Moore in \cite{moore}) and we show that the plane cannot contain uncountably many such disjoint generalized triods (see Lemma \ref{triods}). We use this topological fact to show the regularity of parametrized stationary varifolds contained in a polyhedral cone (see Theorem \ref{conicalreg}). This special case of the problem turns out to be important in order to study a blow-up or a limiting situation obtained in later compactness arguments.
		\item In Section \ref{blowsec} we provide a general result (see Theorems \ref{blow-up} and \ref{blow-up-c}) which allows to blow-up a varifold at a given point or along a sequence of points, with mild assumptions on the decay of the Dirichlet energy of $\Phi$. We show that in the limit one still gets a parametrized stationary varifold and that the parametrization map is the blow-up of $\Phi$, up to a quasiconformal homeomorphism.
		\item In Section \ref{genregsec}, devoted to the regularity in the general case, we initially show a singularity removability lemma (see Lemma \ref{removab}). Then we introduce the set of admissible points where the blow-up can be performed and we prove that the image of its complement has zero Hausdorff dimension (see Lemma \ref{noblowdimzero}). By means of delicate compactness arguments, using the results of Sections \ref{conicalregsec} and \ref{blowsec}, we show that at any such point the speed of decay of the Dirichlet energy is controlled in a uniform way (see Lemma \ref{softlemma} and Corollary \ref{enwellbeh}). We infer that admissible points are relatively open in a set where $\tilde N$ is suitably pinched and we deduce the full regularity result by means of a final blow-up argument (see Theorem \ref{genreg}).
		\item In Section \ref{condsec} we apply our regularity result to 
		positively answer the aforementioned question on the conductivity equation. We also provide an independent, self-contained proof in the planar case $\envdim=2$.
		\item In the appendix we collect some facts about Sobolev functions on the plane. Lemmas \ref{oneinfty} and \ref{oneinftyaux} are possibly new and could have some interest on their own.
	\end{itemize}

	\section{First properties of parametrized stationary varifolds}\label{basicsec}

	Let $\subman$ be either a closed embedded $C^\infty$-smooth submanifold of $\R^\envdim$ or $\R^\envdim$ itself, where $q\ge m\ge 2$ are arbitrary integers. Let $\Sigma$ be a closed connected Riemann surface.

	A map $\Phi\in W^{1,2}(\Sigma,\R^\envdim)$ is \emph{weakly conformal} if, for a.e. $x\in\Sigma$, $d\Phi(x)$ is either zero or a linear conformal map, with respect to the conformal structure of $\Sigma$. For any such map we call $\good\subseteq\Sigma$ the set of Lebesgue points for both $\Phi$ and $d\Phi$ and we let $\goodrk:=\set{x\in\good:d\Phi(0)\neq 0}$ (hence $d\Phi(x)$ is injective and conformal for $x\in\goodrk$).

	\begin{definition}[almost every domain]\label{almostdef}
		We say that a certain property holds \emph{for almost every domain} $\omega\subseteq\Sigma$ if, for any nonnegative $\rho\in C^\infty(\Sigma)$, the property holds for $\omega=\set{\rho>t}$, for a.e. regular value $t>0$ (so in particular it holds for $\Sigma$, as is seen by choosing $\rho\equiv 1$).
		Similarly, given an open set $\Omega\subseteq\C$, a property holds \emph{for almost every domain} $\omega\cptsub\Omega$ if, for any nonnegative $\rho\in C^\infty_c(\Omega)$, the property holds for $\omega=\set{\rho>t}$, for a.e. regular value $t>0$. \hfill\qedsymbol
	\end{definition}

	In the definition below, we will implicitly restrict to the regular values $t>0$ of $\rho$ such that $\restr{\Phi}{\de\set{\rho>t}}$ has a continuous representative (which are a set of full measure, by Sard's theorem and \cite[Theorem~4.21]{evans}) and, with a slight abuse of notation, $\Phi(\de\omega)$ will denote the image by this continuous representative.

	\begin{definition}[parametrized stationary varifold]\label{pardef}
		A triple $(\Sigma,\Phi,N)$ with $\Phi\in W^{1,2}(\Sigma,\R^\envdim)$, $N\in L^\infty(\Sigma,\N\setminus\set{0})$ and $\Phi(\Sigma)\subseteq\subman$ is called a \emph{parametrized integer 2-rectifiable stationary varifold} (in $\subman$) if $\Phi$ is nonconstant, weakly conformal and if, for almost every domain $\omega\subseteq\Sigma$,
		\begin{equation}\label{eq:locstat} \int_\omega N\Big(\ang{d(F(\Phi));d\Phi}_h-F(\Phi)\cdot A(\Phi)(d\Phi,d\Phi)_h\Big)\,d\text{vol}_h=0 \end{equation}
		for all $F\in C^\infty_c(\subman\setminus\Phi(\de\omega),\R^\envdim)$.
		Here $h$ is an arbitrary Riemannian metric compatible with the conformal structure of $\Sigma$ (its choice does not matter, by conformal invariance), $A(X,Y)=-\nabla_X^{\R^\envdim}Y$ denotes the second fundamental form of $\subman$ in $\R^\envdim$ ($A=0$ if $\subman=\R^\envdim$) and $A(\Phi)(d\Phi,d\Phi)_h$ is defined in local coordinates by $\sum_{i,j}h^{ij}A(\Phi)(\de_i\Phi,\de_j\Phi)$. \hfill\qedsymbol
	\end{definition}

	We will usually just say that $(\Sigma,\Phi,N)$ is a \emph{parametrized stationary varifold}.

	\begin{rmk}[equivalent definition]\label{rectvar}
		The definition is clearly independent of the particular representatives of $\Phi$, $d\Phi$, $N$. Calling $\good$ the set of Lebesgue points for both $\Phi$ and $d\Phi$ and applying Lemma \ref{rectif} to a finite atlas of conformal charts, we see that $\Phi(\good)$ is $\mathcal{H}^2$-rectifiable (and $\mathcal{H}^2$-measurable). Moreover, again by Lemma \ref{rectif}, the area formula applies and \eqref{eq:locstat} amounts to say that for almost every $\omega\subseteq\Sigma$ the 2-rectifiable varifold
		\[ \vfd_\omega:=(\Phi(\good\cap\omega),\theta_\omega),\quad\theta_\omega(p):=\sum_{x\in\good\cap\omega\cap\Phi^{-1}(p)}N(x) \]
		is stationary in $\subman\setminus\Phi(\de\omega)$, as is easily seen by writing $F=\pi_{T\subman}F+\pi_{T^\perp\subman}F$.
		In particular, the generalized mean curvature of $\vfd_\omega$ in $\R^\envdim\setminus\Phi(\de\omega)$ is bounded (in $L^\infty$) by $\sqrt{2}\max_\subman\abs{A}$. \hfill\qedsymbol
	\end{rmk}

	\begin{proposition}[continuity of $\bm{\Phi}$]\label{contandsupp}
		The map $\Phi$ has a continuous representative. This representative (still called $\Phi$) satisfies the stationarity property for every open subset $\omega\subseteq\Sigma$, namely
		\[ \int_\omega N\Big(\ang{d(F\circ\Phi);d\Phi}_h-F(\Phi)\cdot A(\Phi)\pa{d\Phi,d\Phi}_h\Big)\,d\text{vol}_h=0 \]
		for all $F\in C^\infty_c(\subman\setminus\Phi(\de\omega),\R^\envdim)$. Moreover, if $\omega$ is connected, $\Phi(\bar\omega)=\supp{\norm{\vfd_\omega}}$ unless $\Phi$ is constant on $\omega$.
	\end{proposition}

	\begin{proof}
		Let $\phi:U\to\Omega$ be a local conformal chart (with $U\subseteq\Sigma$ and $\Omega\subseteq\C$) and let $\Psi:=\Phi\circ\phi^{-1}$, $\tilde\good:=\phi(\good\cap U)$ and $\tilde\goodrk:=\phi(\goodrk\cap U)$. For any $x\in\Omega$ and any $r<\mz\dist(x,\de\Omega)$ we can apply Lemma \ref{slicing} (with $\tau=1$) and obtain a radius $r'\in(r,2r)$ such that \eqref{eq:locstat} applies with $\omega=\phi^{-1}(B_{r'}^2(x))$ and such that
		\begin{equation}\label{eq:goodchoice} \diam\Psi(\de B_{r'}^2(x))\le\sqrt{4\pi}\pa{\int_{B_{2r}^2(x)}\abs{\nabla\Psi}^2\,d\mathcal{L}^2}^{1/2},\quad\mathcal{H}^1(\de B_{r'}^2(x)\setminus\tilde\good)=0. \end{equation}
		Let us assume that $\Psi$ is not (a.e.) constant on $B_{r'}^2(x)$. If $z\in {\bar B}_{r'}^2(x)\cap\tilde\goodrk$ we have $\Psi(z)\in\supp{\norm{\vfd_\omega}}$, as
		\[ \norm{\vfd_\omega}(B_s^\envdim(\Psi(z)))=\mz\int_{\Psi^{-1}(B_s^\envdim(\Psi(z)))\cap B_{r'}^2(x)}(N\circ\phi^{-1})\abs{\nabla\Psi}^2\,d\mathcal{L}^2>0 \]
		for all $s>0$. Hence, since $\supp{\norm{\vfd_\omega}}$ is closed in $\R^\envdim$, by Lemma \ref{essim} the essential image of $\restr{\Psi}{B_{r'}^2(x)}$ is included in $\supp{\norm{\vfd_\omega}}$. The converse inclusion trivially holds, as well, so we conclude that the essential image of $\restr{\Psi}{B_{r'}^2(x)}$ coincides with $\supp{\norm{\vfd_\omega}}$; in particular, the latter includes the compact set $\Gamma:=\Psi(\de B_{r'}^2(x))$ (by the second part of \eqref{eq:goodchoice}).

		Moreover, since in $\R^\envdim\setminus\Gamma$ the varifold $\vfd_\omega$ has generalized mean curvature bounded by $\sqrt{2}\norm{A}_{L^\infty}$, from the monotonicity formula \cite[Theorem~17.6]{simon} and \cite[Remark~17.9(1)]{simon} we deduce
		\[ \mz\int_{B_{2r}^2(x)}(N\circ\phi^{-1})\abs{\nabla\Psi}^2\,d\mathcal{L}^2\ge\norm{\vfd_\omega}(B_s^\envdim(p))\ge e^{-(\sqrt{2}\norm{A}_{L^\infty})s}\cdot\pi s^2 \]
		for all $p\in\supp{\norm{\vfd_\omega}}\setminus\Gamma$ and all $s\le\dist(p,\Gamma)$. If $\subman$ is compact, choosing $s:=\dist(p,\Gamma)\le\diam\subman$ and recalling \eqref{eq:goodchoice}, we conclude that
		\begin{equation}\label{eq:xpdiamest} \begin{split} &\diam\Psi(\tilde\good\cap B_r^2(x))\le\diam\supp{\norm{\vfd_\omega}}\le\diam\Gamma+2\max_{p\in\supp{\norm{\vfd_\omega}}}\dist(p,\Gamma) \\
		&\le 2\pa{\sqrt{\pi}+\pa{e^{(\sqrt{2}\norm{A}_{L^\infty})\diam\subman}\frac{\norm{N}_{L^\infty}}{2\pi}}^{1/2}}\pa{\int_{B_{2r}^2(x)}\abs{\nabla\Psi}^2\,d\mathcal{L}^2}^{1/2}. \end{split} \end{equation}
		If instead $\subman=\R^\envdim$, then we have
		\begin{equation}\label{eq:xpdiamest2} \begin{split} &\diam\Psi(\tilde\good\cap B_r^2(x))\le\diam\supp{\norm{\vfd_\omega}}\le\diam\Gamma+2\sup_{p\in\supp{\norm{\vfd_\omega}}}\dist(p,\Gamma) \\
		&\le 2\pa{\sqrt{\pi}+\pa{\frac{\norm{N}_{L^\infty}}{2\pi}}^{1/2}}\pa{\int_{B_{2r}^2(x)}\abs{\nabla\Psi}^2\,d\mathcal{L}^2}^{1/2}. \end{split} \end{equation}
		This estimate for $\diam\Psi(\tilde\good\cap B_r^2(x))$ is trivially true also when $\Psi$ is a.e. constant on $B_{r'}^2(x)$. The last expressions are infinitesimal as $r\to 0$, locally uniformly in $x$. We infer that $\restr{\Psi}{\tilde\good}$ is locally uniformly continuous on $\Omega$ and thus has a continuous representative. This shows that $\Phi$ has a continuous representative.

		We record here another estimate for $\diam\supp{\norm{\vfd_\omega}}$ independent of $\diam\subman$, which will be useful later. All the points in $\supp{\norm{\vfd_\omega}}$ have distance at most $2D+2$ from $\Gamma$, where $D:=\frac{e^H}{\pi}\norm{\vfd_\omega}(\R^\envdim)$ and $H$ is an upper bound for the generalized mean curvature of $\vfd_\omega$ in $\R^\envdim\setminus\Gamma$: if this were not the case, we would have $\Phi$ nonconstant on $\omega$ and thus $\Gamma\subseteq\Phi(\bar\omega)=\supp{\norm{\vfd_\omega}}$. By connectedness of $\Phi(\bar\omega)$ we could find points $p_j\in\supp{\norm{\vfd_\omega}}$ such that $\dist(p_j,\Gamma)=2j$, for $1\le j\le\floor{D}+1$, and since the balls $B_1^\envdim(p_j)$ are disjoint we would have
		\[ \norm{\vfd_\omega}(\R^\envdim)\ge\sum_j\norm{\vfd_\omega}(B_1^\envdim(p_j))\ge(\floor{D}+1)e^{-H}\pi>De^{-H}\pi, \]
		which is a contradiction. We deduce that
		\begin{equation}\label{eq:suppest} \diam\supp{\norm{\vfd_\omega}}\le\diam\Gamma+\frac{4e^H}{\pi}\norm{\vfd_\omega}(\R^\envdim)+4. \end{equation}

		We now show the statement about the stationarity property. If $F\in C^\infty_c(\subman\setminus\Phi(\de\omega),\R^\envdim)$,
		then we can find a nonnegative $\rho\in C^\infty_c(\omega)$ such that $\rho=1$ on the compact set $\omega\cap\Phi^{-1}(\supp{F})$. For almost every $t\in(0,1)$ the stationarity property \eqref{eq:locstat} holds in $\set{\rho>t}$, so
		\[ \int_{\set{\rho>t}}N\Big(\ang{d(F\circ\Phi);d\Phi}_h-F(\Phi) A(\Phi)\pa{d\Phi,d\Phi}_h\Big)\,d\text{vol}_h=0 \]
		and clearly the left-hand side does not change if we replace $\set{\rho>t}$ with $\omega$.

		Finally, the last statement is obtained with the same argument used in the first part of the proof.
	\end{proof}

	From now on, we will always assume that the map $\Phi$ is continuous, for any parametrized stationary varifold.

	\begin{proposition}[negligibility of $\bm{\Sigma\setminus\goodrk}$ in the target]\label{zerocontrib}
		We have $\mathcal{H}^2(\Phi(\Sigma\setminus\goodrk))=0$.
	\end{proposition}

	\begin{proof}
		As already observed in Remark \ref{rectvar}, the area formula can be applied on subsets of $\good$. In particular, since $d\Phi=0$ on $\good\setminus\goodrk$, we get $\mathcal{H}^2(\Phi(\good\setminus\goodrk))=0$.
		In order to show that $\mathcal{H}^2(\Phi(\Sigma\setminus\good))=0$, we pick any local conformal chart $\phi:U(\subseteq\Sigma)\to\Omega(\subseteq\C)$ and, as in the previous proof, we set $\Psi:=\Phi\circ\phi^{-1}$ and
		$\tilde\good:=\phi(\good\cap U)$.

		Fix an arbitrary $\delta>0$ and an open set $W\subseteq\Omega$ containing $\Omega\setminus\tilde\good$. For any $z\in W$ we can find a radius $r<\mz\dist(z,\de W)\wedge 1$ such that
		$\bar C\pa{\int_{B_{2r}^2(z)}\abs{\nabla\Psi}^2\,d\mathcal{L}^2}^{1/2}<\delta$, where $\bar C$ is the constant appearing in the right-hand side of \eqref{eq:xpdiamest} (or \eqref{eq:xpdiamest2} if $\subman=\R^\envdim$), and
		\[ \int_{B_{2r}^2(z)}\abs{\nabla\Psi}^2\,d\mathcal{L}^2\le 8\int_{B_r^2(z)}\abs{\nabla\Psi}^2\,d\mathcal{L}^2+4r^2: \]
		indeed, if such $r$ did not exist, for $j$ big enough we would have $(2^{-j})^2\le\int_{B_{2^{-j}}^2(z)}\abs{\nabla\Psi}^2\,d\mathcal{L}^2\le\frac{1}{8}\int_{B_{2^{-j+1}}^2(z)}\abs{\nabla\Psi}^2\,d\mathcal{L}^2$, hence $(2^{-j})^2\le\int_{B_{2^{-j}}^2(z)}\abs{\nabla\Psi}^2\,d\mathcal{L}^2=O(2^{-3j})=o((2^{-j})^2)$, which is a contradiction. By Besicovitch covering theorem, we can extract countably many balls $B_{r_i}^2(x_i)$ from this collection with $\uno_W\le\sum_i\uno_{B_{r_i}^2(x_i)}\le\mathfrak{N}\uno_W$, for some universal constant $\mathfrak{N}$. By inequality \eqref{eq:xpdiamest} (or \eqref{eq:xpdiamest2}) we have $\diam\Phi(B_{r_i}^2(x_i))<\delta$ and
		\[ \begin{split} \sum_i\frac{\pi}{4}(\diam\Phi(B_{r_i}^2(x_i)))^2&\le\frac{\pi{\bar C}^2}{4}\sum_i\int_{B_{2r_i}^2(x_i)}\abs{\nabla\Psi}^2\,d\mathcal{L}^2 \\
		&\le
		2\pi{\bar C}^2\sum_i\int_{B_{r_i}^2(x_i)}\abs{\nabla\Psi}^2\,d\mathcal{L}^2+{\bar C}^2\mathcal{L}^2(B_{r_i}^2(x_i)) \\
		&\le 2\pi{\bar C}^2\mathfrak{N}\int_W\abs{\nabla\Psi}^2\,d\mathcal{L}^2+{\bar C}^2\mathfrak{N}\mathcal{L}^2(W). \end{split} \]
		Since $\delta$ was arbitrary, we get $\mathcal{H}^2(\Phi(\Omega\setminus\tilde\good))\le 2\pi{\bar C}^2\mathfrak{N}\int_W\abs{\nabla\Psi}^2\,d\mathcal{L}^2+{\bar C}^2\mathfrak{N}\mathcal{L}^2(W)$. Since $\mathcal{L}^2(\Omega\setminus\tilde\good)=0$ and $W$ was arbitrary as well, we arrive at $\mathcal{H}^2(\Phi(\Omega\setminus\tilde\good))=0$.
	\end{proof}

	\begin{proposition}[structure of fibers]\label{fincpt}
		For any $p\in\subman$ the compact set $\Phi^{-1}(p)$ has finitely many connected components. If $x\in\goodrk$ then $x$ is isolated in $\Phi^{-1}(\Phi(x))$.
	\end{proposition}

	\begin{proof}
		Since the varifold $\vfd_\Sigma$ is stationary, the limit
		\[ M:=\lim_{s\to 0}\frac{\norm{\vfd_\Sigma}(B_s^\envdim(p))}{\pi s^2} \]
		exists. We claim that the number of connected components of $\Phi^{-1}(p)$ is not greater than $M$. If this were not the case, we could split $\Phi^{-1}(p)=\bigsqcup_{j=1}^J K_j$, where the subsets $K_j$ are disjoint and compact and $J>M$ is a finite integer (if the maximum value of $J$ such that this can be done were at most $M$, then one of the subsets would be disconnected, contradicting the maximality). We could then find disjoint open neighborhoods $\omega_j\supseteq K_j$ and we would have
		\[ M=\lim_{s\to 0}\frac{\norm{\vfd_\Sigma}(B_s^\envdim(p))}{\pi s^2}\ge\sum_{j=1}^J\lim_{s\to 0}\frac{\norm{\vfd_{\omega_j}}(B_s^\envdim(p))}{\pi s^2}\ge J \]
		(by \cite[Remark~17.9(1)]{simon}: notice that $\Phi$ must be nonconstant on any connected component of $\omega_j$ intersecting $K_j$, hence by Proposition \ref{contandsupp} $p\in\supp{\norm{\vfd_{\omega_j}}}$). This is a contradiction.

		Assume now $x\in\goodrk$ and call $K_x$ the connected component of $\Phi^{-1}(\Phi(x))$ containing $x$. It suffices to show that $K_x=\set{x}$, since we already know that $\Phi^{-1}(\Phi(x))$ is a finite union of compact connected sets. If $\phi$ is a local conformal chart centered at $x$ and $\Psi:=\Phi\circ\phi^{-1}$, as in the proof of Proposition \ref{nisusc} below we can find a radius $r'>0$ such that $\Phi(x)=\Psi(0)\nin\Psi(\de B_{r'}^2(0))$. Hence $K_x\subseteq\phi^{-1}(B_{r'}^2(0))$ and, since $r'$ is arbitrarily small, we deduce $K_x=\phi^{-1}(\set{0})=\set{x}$.
	\end{proof}

	We now define a more robust representative $\tilde N$ of the multiplicity function $N$, which is canonically defined everywhere and is upper semicontinuous. We point out that (a priori) $\tilde N$ could take values in $\brapa{1,\infty}$ instead of $\N\setminus\set{0}$.

	\begin{definition}[robust representative of $\bm{N}$]\label{tildendef}
		Given $x\in\Sigma$, we call $K_x$ the connected component of $\Phi^{-1}(\Phi(x))$ containing $x$ and we let
		\[ \tilde N(x):=\inf_{\substack{\omega\supseteq K_x, \\ \Phi(x)\nin\Phi(\de\omega)}}\lim_{s\to 0}\frac{\norm{\vfd_\omega}(B_s^\envdim(\Phi(x)))}{\pi s^2}. \]
		The limit exists and is at least $1$, by the stationarity of $\vfd_\omega$ in $\subman\setminus\Phi(\de\omega)$ (which contains $\Phi(x)$) and the fact that $\Phi(x)\in\supp{\norm{\vfd_\omega}}$ (by Lemma \ref{contandsupp}, since $\Phi$ is necessarily nonconstant on the connected component of $\omega$ containing $x$). Notice that $\tilde N=\tilde N(x)$ on $K_x$. Moreover, the infimum is actually a minimum and is achieved whenever $\bar\omega$ is disjoint from the compact set $\Phi^{-1}(\Phi(x))\setminus K_x$. \hfill\qedsymbol
	\end{definition}

	\begin{proposition}[upper semicontinuity of $\tilde{\bm{N}}$]\label{nisusc}
		The function $\tilde N$ is upper semicontinuous and $\tilde N\ge 1$. Moreover, $\tilde N=N$ a.e. on $\goodrk$.
	\end{proposition}

	\begin{proof}
		We already observed that $\tilde N\ge 1$ everywhere. Let $\lambda>1$ and $x\in\Sigma$ such that $\tilde N(x)<\lambda$. Choose any open set $\omega\supseteq K_x$ with $\bar\omega$ disjoint from $\Phi^{-1}(\Phi(x))\setminus K_x$, so that $\lim_{s\to 0}\frac{\norm{\vfd_\omega}(B_s^\envdim(\Phi(x)))}{\pi s^2}<\lambda$. Whenever $z\in\omega$ is close enough to $x$ we have $\Phi(z)\nin\Phi(\de\omega)$, so $K_z\subseteq\omega$ and by definition
		\[ \tilde N(z)\le\lim_{s\to 0}\frac{\norm{\vfd_\omega}(B_s^\envdim(\Phi(z)))}{\pi s^2}. \]
		But as $z\to x$ we have $\Phi(z)\to\Phi(x)$. Hence, eventually $\lim_{s\to 0}\frac{\norm{\vfd_\omega}(B_s^\envdim(\Phi(z)))}{\pi s^2}<\lambda$ (see \cite[Corollary~17.8]{simon}) and so $\tilde N(z)<\lambda$.

		Assume now $x\in\goodrk$ and $\int_{B_r^2(x)}\abs{N-N(x)}\,d\mathcal{L}^2=o(r^2)$, $\int_{B_r^2(x)}\abs{\nabla\Phi-\nabla\Phi(x)}^2\,d\mathcal{L}^2=o(r^2)$. Fix any open set $\omega$ containing $x$. Let $\phi$ be a local conformal chart centered at $x$ and set $\Psi:=\Phi\circ\phi^{-1}$, $\alpha:=\pa{\frac{\abs{\nabla\Psi(0)}}{\sqrt{2}}}^{-1}$. For any $s>0$ small enough we have
		\[ \begin{split} \norm{\vfd_\omega}(B_s^\envdim(\Phi(x)))&\ge\mz\int_{B_{\alpha s}^2(0)\cap\Psi^{-1}(B_s^\envdim(\Psi(0)))}(N\circ\phi^{-1})\abs{\nabla\Psi}^2\,d\mathcal{L}^2 \\
		&\ge\mz N(x)\abs{\nabla\Psi(0)}^2\mathcal{L}^2\Big(B_{\alpha s}^2(0)\cap\Psi^{-1}(B_s^\envdim(\Psi(0)))\Big) \\
		&\fantasma{\ge}-\mz\int_{B_{\alpha s}^2(0)}\abs{(N\circ\phi^{-1})\abs{\nabla\Psi}^2-N(x)\abs{\nabla\Psi(0)}^2}\,d\mathcal{L}^2. \end{split} \]
		By \cite[Theorem~6.1]{evans}, the function $s^{-1}(\Psi(\alpha s\,\cdot)-\Psi(0))$ converges to $\alpha\ang{\nabla\Psi(0),\cdot}$ (which is a linear isometry) in measure on $B_1^2(0)$, hence the first term in the right-hand side is $\pi N(x)s^2+o(s^2)$. Moreover, the function $N\circ\phi^{-1}(\alpha s\,\cdot)$ converges to $N(x)$ in measure and is bounded by $\norm{N}_{L^\infty}$, while $\abs{\nabla\Psi}^2(\alpha s\,\cdot)\to\abs{\nabla\Psi(0)}^2$ in $L^1(B_1^2(0))$. So the last term in the right-hand side is $o(s^2)$.
		This shows that $\tilde N(x)\ge N(x)$.

		Fix now any $0<\epsilon<\alpha^{-1}$. By Lemma \ref{slicingleb}, applied to $y\mapsto\Psi(y)-\Psi(0)-\ang{\nabla\Psi(0),y}$, we can find a radius $r'$ such that
		\[ \abs{\Psi(r'y)-\Psi(0)-\ang{\nabla\Psi(0),r'y}}\le\epsilon r' \]
		for all $y\in S^1$.
		Thus, choosing $\omega:=\phi^{-1}(B_{r'}^2(0))$ and applying the monotonicity formula,
		\[ \tilde N(x)\le e^{(\sqrt{2}\norm{A}_{L^\infty})(\beta-\epsilon)r'}\frac{\norm{\vfd_\omega}(B_{(\beta-\epsilon)r'}^\envdim(\Phi(x)))}{\pi(\beta-\epsilon)^2(r')^2}
		\le(1+O(r'))\frac{\int_{B_{r'}^2(0)}(N\circ\phi^{-1})\abs{\nabla\Psi}^2\,d\mathcal{L}^2}{2\pi(\beta-\epsilon)^2(r')^2}, \]
		where $\beta:=\frac{\abs{\nabla\Psi(0)}}{\sqrt{2}}=\alpha^{-1}$.
		Since $r'$ is arbitrarily small, we get $\tilde N(x)\le\frac{N(x)\beta^2}{(\beta-\epsilon)^2}$. Letting $\epsilon\to 0$ we get the converse inequality $\tilde N(x)\le N(x)$.
	\end{proof}

	It is useful to introduce the following local notion of parametrized stationary varifold.

	\begin{definition}[local counterpart]\label{locpardef}
		Let $\Omega\subseteq\C$ be open. A triple $(\Omega,\Phi,N)$ with $\Phi\in W^{1,2}_{loc}(\Omega,\R^\envdim)$, $N\in L^\infty(\Omega,\N\setminus\set{0})$ and $\Phi(\Omega)\subseteq\subman$ is called a \emph{local parametrized stationary varifold} (in $\subman$) if $\Phi$ is weakly conformal and if, for almost every domain $\omega\cptsub\Omega$,
		\[ \int_\omega N\Big(\ang{\nabla(F(\Phi));\nabla\Phi}-F(\Phi)\cdot A(\Phi)(\nabla\Phi,\nabla\Phi)\Big)\,d\mathcal{L}^2=0 \]
		for all $F\in C^\infty_c(\subman\setminus\Phi(\de\omega),\R^\envdim)$, with $A(\Phi)(\nabla\Phi,\nabla\Phi):=A(\Phi)(\de_1\Phi,\de_1\Phi)+A(\Phi)(\de_2\Phi,\de_2\Phi)$.
		We also require that
		\begin{equation}\label{eq:massass}
		\norm{\vfd_\Omega}(B_s^\envdim(p))=\mz\int_{\Phi^{-1}(B_s^\envdim(p))}N\abs{\nabla\Phi}^2\,d\mathcal{L}^2=O(s^2),
		\end{equation}
		uniformly in $p\in\R^\envdim$. \hfill\qedsymbol
	\end{definition}

	Notice that, in the last definition, the map $\Phi$ is allowed to be constant. The technical assumption \eqref{eq:massass} will be used only in the proof of Lemma \ref{removab}, which in turn is used in Sections \ref{conicalregsec} and \ref{genregsec}.

	\begin{rmk}[localization]
		 If $(\Sigma,\Phi,N)$ is a parametrized stationary varifold and $\phi:U(\subseteq\Sigma)\to\Omega(\subseteq\C)$ is a local conformal chart, then $(\Omega,\Phi\circ\phi^{-1},N\circ\phi^{-1})$ is a local parametrized stationary varifold: assumption \eqref{eq:massass} holds thanks to the monotonicity formula satisfied by $\vfd_\Sigma$. \hfill\qedsymbol
	\end{rmk}

	\begin{rmk}[local statements]\label{alsoloc}
		Proposition \ref{contandsupp} applies to the local case as well (with $\omega\cptsub\Omega$ in the statement), with the same proof: hence, we will tacitly assume that $\Phi$ is continuous for all local parametrized stationary varifolds. The same is true for Proposition \ref{zerocontrib}.
		
		The first part of Proposition \ref{fincpt} holds whenever $\Phi^{-1}(p)$ is compact (with the same proof with a neighborhood $\Phi^{-1}(p)\subseteq\omega\cptsub\Omega$ in place of $\Sigma$), while the second part holds in general (since, in its proof, we can apply the first part to the domain $\phi^{-1}(B_{r'}^2(0))$).
		
		The domain of definition of the function $\tilde N$, i.e. the set $\{\tilde N<\infty\}$, is an open subset of $\Omega$. The same argument of Proposition \ref{fincpt} shows that it consists of all points $x$ such that $K_x$ is compact and disjoint from the closure of $\Phi^{-1}(\Phi(x))\setminus K_x$. Proposition \ref{nisusc} always holds in the local case, with the same proof. \hfill\qedsymbol
	\end{rmk}

	\begin{rmk}[isolating sets with high $\tilde{\bm{N}}$]\label{hightilden}
		A useful fact which will be used in Sections \ref{conicalregsec} and \ref{genregsec} is the following: if $(\Omega,\Phi,N)$ is a local parametrized stationary varifold and $x\in\Omega$ satisfies $\tilde N(x)<\infty$, then for any $0<\epsilon<1$ we can find a neighborhood $K_x\subseteq\omega\cptsub\Omega$ with $\Phi(\omega)\cap\Phi(\de\omega)=\emptyset$ and such that $\omega\cap\Phi^{-1}(\Phi(y))=K_y$ whenever $y\in\omega$ has $\tilde N(y)\ge\tilde N(x)-\epsilon$.

		Indeed, let $K_x\subseteq\omega\cptsub\Omega$ with $\bar{\omega}$ disjoint from $\Phi^{-1}(\Phi(x))\setminus K_x$, so that $\tilde N(x)$ is the density of $\vfd_{\omega}$ at $\Phi(x)$.
		If $y_1,y_2\in\omega$ have the same image and are close enough to $K_x$, then $K_{y_1},K_{y_2}\subseteq\omega$ (since $\Phi(y_1)=\Phi(y_2)\nin\Phi(\de\omega)$) and the density of $\vfd_{\omega}$ at $\Phi(y_1)=\Phi(y_2)$ is less than $\tilde N(x)+1-\epsilon$ (by upper semicontinuity of the density). Hence,
		if $K_{y_1}\neq K_{y_2}$, calling $K_{y_i}\subseteq\omega_i\subseteq\omega$ two disjoint neighborhoods with $\Phi(y_i)\nin\Phi(\de\omega_i)$ we get
		\[ \begin{split} \tilde N(y_1)+1&\le\tilde N(y_1)+\tilde N(y_2)\le\lim_{s\to 0}\frac{\norm{\vfd_{\omega_1}}(B_s^\envdim(\Phi(y_1)))}{\pi s^2}+\lim_{s\to 0}\frac{\norm{\vfd_{\omega_2}}(B_s^\envdim(\Phi(y_2)))}{\pi s^2} \\
		&\le\lim_{s\to 0}\frac{\norm{\vfd_{\omega}}(B_s^\envdim(\Phi(y_1)))}{\pi s^2}<\tilde N(x)+1-\epsilon. \end{split} \]
		The claim is established by shrinking $\omega$ and by replacing it with $\omega\setminus\Phi^{-1}(\Phi(\de\omega))$. \hfill\qedsymbol
	\end{rmk}

	We quote without proof the following theorem, which is the main result of \cite{rivtarget}.

	\begin{thm}[constant multiplicity case]\label{constn}
		Let $(\Sigma,\Phi,N)$ be a parametrized stationary varifold in $\subman$.
		If $N$ is a.e. constant (and hence can be changed to $1$ without affecting the stationarity), then $\Phi\in C^\infty(\Sigma,\subman)$ and $-\Delta_h\Phi=A(\Phi)(d\Phi,d\Phi)_h$. The same holds for local parametrized stationary varifolds.
	\end{thm}

	This statement is proved in \cite{rivtarget} for parametrized stationary varifolds in a compact manifold $\subman$, but the argument carries over to the local case and to the situation $\subman=\R^\envdim$, as well.

	\section{Regularity of parametrized stationary varifolds in a polyhedral cone}\label{conicalregsec}

	This section addresses the regularity problem for local parametrized stationary varifolds in $\R^\envdim$, under the additional constraint that they are contained in a finite union of 2-dimensional planes through the origin.
	We will need a nontrivial result in planar topology, which we state and prove below.

		\subsection{A topological lemma about triods}

		\begin{definition}[generalized triod]
			A \emph{generalized triod} in $S^2$ is a quadruple $T=(K,\gamma_1,\gamma_2,\gamma_3)$ such that:
			\begin{itemize}
				\item $\emptyset\neq K\subseteq S^2$ is compact and connected;
				\item $\gamma_i\in C^\infty([0,1],S^2)$ are injective regular curves (i.e. $\dot\gamma(t)\neq 0$ for all $t\in [0,1]$);
				\item $K$, $\gamma_1(\brapa{0,1})$, $\gamma_2(\brapa{0,1})$, $\gamma_3(\brapa{0,1})$ are pairwise disjoint;
				\item $\gamma_i(1)\in K$.
			\end{itemize}
			We will denote $\supp{T}:=K\sqcup\gamma_1(\brapa{0,1})\sqcup\gamma_2(\brapa{0,1})\sqcup\gamma_3(\brapa{0,1})$. \hfill\qedsymbol
		\end{definition}

		The proof of the following lemma is inspired by the proof of a simpler statement which appears in \cite[Lemma~2.15]{abc}.

		\begin{lemmaen}[uncountably many triods intersect]\label{triods}
			Let $(T_j)_{j\in J}$ be a collection of generalized triods in $S^2$ such that $\supp{T_j}\cap\supp{T_{j'}}=\emptyset$ for any $j\neq j'$. Then $J$ is at most countable.
		\end{lemmaen}

		\begin{proof}
			We equip the set $\mathcal{T}$ of all generalized triods in $S^2$ with the following metric: given $T=(K,\gamma_1,\gamma_2,\gamma_3)$, $T'=(K',\gamma_1',\gamma_2',\gamma_3')\in\mathcal{T}$ we set
			\[ d(T,T'):=d_H(K,K')+\sum_i\max_{t\in [0,1]}d_{S^2}(\gamma_i(t),\gamma_i'(t)), \]
			where $d_{S^2}$ denotes the spherical distance on $S^2$ and $d_H$ is the corresponding Hausdorff distance on the set of all nonempty compact subsets of $S^2$.

			Since the metric space $(\mathcal{T},d)$ is separable, it suffices to show that any triod $T_{j_0}$ is isolated in $\set{T_j\mid j\in J}\subseteq\mathcal{T}$. Let $T_{j_0}=(K,\gamma_1,\gamma_2,\gamma_3)$.

			\emph{Case 1:} $\gamma_1(\brapa{0,1}),\gamma_2(\brapa{0,1}),\gamma_3(\brapa{0,1})$ \emph{do not belong to the same connected component of} $S^2\setminus K$.
			Assume for instance that $\gamma_1(\brapa{0,1})$ and $\gamma_2(\brapa{0,1})$ belong to different connected components: then, letting
			\[ \epsilon:=\min\set{d_{S^2}(\gamma_1(0),K),d_{S^2}(\gamma_2(0),K)}>0, \]
			any different triod $T_j=(K',\gamma_1',\gamma_2',\gamma_3')$ satisfies $d(T_{j_0},T_j)\ge\epsilon$. Indeed, if this were not the case, $\gamma_1'(0)$ would lie in the same component of $S^2\setminus K$ as $\gamma_1(0)$ (since the spherical ball $B_\epsilon^{S^2}(\gamma_1(0))$ is a connected subset of $S^2\setminus K$) and similarly for $\gamma_2'(0)$. But this contradicts the fact that $\supp{T_j}$ is a connected subset of $S^2\setminus K$.

			\emph{Case 2:} $\gamma_1(\brapa{0,1}),\gamma_2(\brapa{0,1}),\gamma_3(\brapa{0,1})$ \emph{belong to the same connected component} $U$ \emph{of} $S^2\setminus K$.
			Since $K$ is connected, there exists a diffeomorphism
			\[ \upsilon:U\to B_1^2(0)\subseteq\C \]
			(indeed, $S^2\setminus U$ is connected and we can apply \cite[Theorems~13.11~and~14.8]{rudin}).
			For $t\in\brapa{0,1}$ let $\alpha_i(t):=\upsilon\circ\gamma_i(t)$.
			Notice that $\lim_{t\to 1}\abs{\alpha_i(t)}=1$.
			Up to applying another diffeomorphism, we can assume that $\abs{\alpha_i(0)}=\mz$ and $\abs{\alpha_i(t)}>\mz$ for $t\in(0,1)$ and $i=1,2,3$ (e.g. by adapting the argument in \cite[Theorem~II.5.2]{kosinski}).

			Let $s_i:=\min\set{t:\abs{\alpha_i(t)}=\frac{3}{4}}>0$. Moreover, for any $\tau\in\pa{\frac{3}{4},1}$ let
			\[ r_i(\tau):=\min\set{t:\abs{\alpha_i(t)}=\tau}>s_i. \]
			By Jordan's closed curve theorem for piecewise smooth curves, the points $\alpha_i(0)$ and $\alpha_i(r_i(\tau))$ are in the same order on the circles $\set{\abs{z}=\mz}$ and $\set{\abs{z}=\tau}$.
			The curves $\alpha_i([0,r_i(\tau)])$ and the circles $\set{\abs{z}=\mz}$, $\set{\abs{z}=\tau}$ bound three disjoint domains $R_1(\tau)$, $R_2(\tau)$, $R_3(\tau)$; we adopt the convention that $R_i(\tau)$ is the region whose closure is disjoint from $\alpha_i([0,r_i(\tau)])$. Let
			\begin{equation}\label{eq:defdelta} \delta:=\inf_{\tau\in\pa{\frac{3}{4},1}}\min_i d_{\R^2}(\alpha_i(0),\obar{R_i(\tau)}) \end{equation}
			and notice that, since
			\[ d_{\R^2}(\alpha_i(0),\obar{R_i(\tau)})=d_{\R^2}(\alpha_i(0),\de R_i(\tau))=d_{\R^2}(\alpha_i(0),\de R_i(\tau)\setminus\de B_\tau^2(0)), \]
			we have $\delta>0$.

			\begin{center}
				\begin{overpic}[width=11cm]{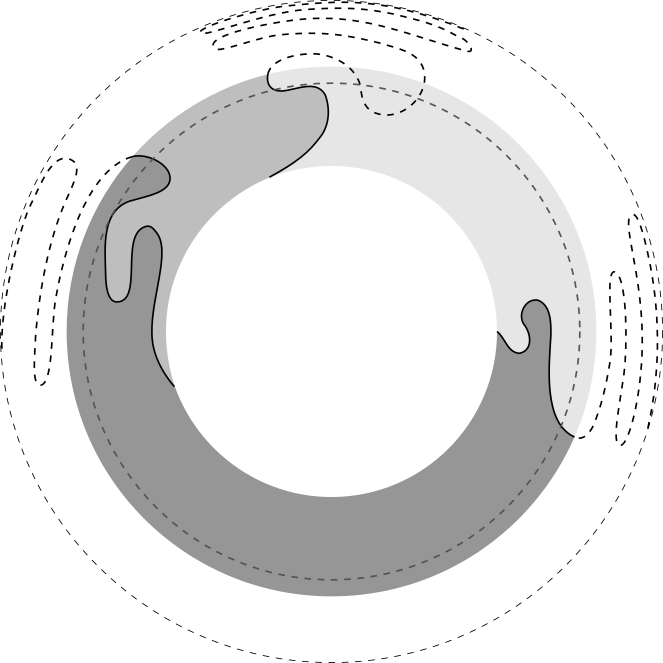}
					\put(65,49){$\alpha_1(0)$}
					\put(40,70){$\alpha_2(0)$}
					\put(28,41){$\alpha_3(0)$}
					\put(74,34){$\alpha_1(s_1)$}
					\put(38,83){$\alpha_2(s_2)$}
					\put(18.5,67){$\alpha_3(s_3)$}
					\put(29,76){$R_1(\tau)$}
					\put(47,17){$R_2(\tau)$}
					\put(70,70){$R_3(\tau)$}
				\end{overpic}
			\end{center}

			Assume now $T_j=(K',\gamma_1',\gamma_2',\gamma_3')$ satisfies $d(T_{j_0},T_j)<\epsilon$. If $\epsilon$ is small enough, we have:
			\begin{itemize}
				\item $\supp{T_j}\subseteq U$ (this is obtained arguing as in the first case), so that we can define $\alpha_i'(t):=\upsilon\circ\gamma_i'(t)$ for $t\in [0,1]$;
				\item $\upsilon(K')\subseteq\set{\frac{3}{4}<\abs{z}<1}$;
				\item $\max_{t\in[0,s_i]}\abs{\alpha_i'(t)-\alpha_i(t)}<\delta'$, for some $\delta'\le\delta$ to be chosen later;
				\item $\alpha_i'([s_i,1])\subseteq\set{\abs{z}>\mz}$.
			\end{itemize}
			Let $t_i':=\max\pa{\set{t:\abs{\alpha_i'(t)}=\mz}\cup\set{0}}<s_i$. We claim that
			\begin{equation}\label{eq:nearness} \abs{\alpha_i'(t_i')-\alpha_i(0)}<\delta. \end{equation}
			If $t_i'=0$ this is trivial, while otherwise $\abs{\alpha_i'(t_i')}=\mz$ and
			\[ \dist(\alpha_i'(t_i'),\alpha_i([0,s_i]))\le\abs{\alpha_i'(t_i')-\alpha_i(t_i')}<\delta', \]
			which yields our claim once $\delta'$ is chosen so small that
			\[ \Big\{z:\abs{z}=\mz,\dist(z,\alpha_i([0,s_i]))<\delta'\Big\}\subseteq B_\delta^2(\alpha_i(0)) \]
			(if such $\delta'$ did not exist, we could find points $\abs{z_k}=\mz$ with $\dist(z_k,\alpha_i([0,s_i]))\to 0$ and $\abs{z_k-\alpha_i(0)}\ge\delta$; up to subsequences we could assume $z_k\to z_\infty$, for some $z_\infty\in\alpha_i([0,s_i])$ with $\abs{z_\infty}=\mz$, hence $z_\infty=\alpha_i(0)$, contradiction).

			Fix now any $\tau$ such that $\max\set{\abs{z}:z\in\upsilon(\supp{T_j})}<\tau<1$.
			The connected set
			\[ \upsilon(K')\sqcup\alpha_1'(\pa{t_1',1})\sqcup\alpha_2'(\pa{t_2',1})\sqcup\alpha_3'(\pa{t_3',1}) \]
			is contained in $\set{\mz<\abs{z}<\tau}$ and is disjoint from $\alpha_1(\brapa{0,1})\sqcup\alpha_2(\brapa{0,1})\sqcup\alpha_3(\brapa{0,1})$, so it is contained in some region $R_{i_0}(\tau)$ and, in particular, $\alpha_{i_0}'([t_{i_0}',1])\subseteq\bar {R_{i_0}(\tau)}$. But, using \eqref{eq:defdelta} and \eqref{eq:nearness}, we infer that $\alpha_{i_0}'(t_{i_0}')\nin\obar{R_{i_0}(\tau)}$. This contradiction shows that such $T_j$ with $d(T_{j_0},T_j)<\epsilon$ cannot exist, completing the treatment of the second case.
		\end{proof}


		\subsection{Planar case}

		We now show the regularity in the special case where the parametrized varifold is contained in a plane.

		\begin{thm}[regularity in codimension 0]\label{planarreg}
			Let $(\Omega,\Phi,N)$ be a local parametrized stationary varifold in $\R^2=\C$ defined on a bounded connected open set $\Omega\subset\C$. Assume that $\Phi^{-1}(p)$ is compact for all $p\in\C$. Then $\Phi$ is holomorphic or antiholomorphic.
		\end{thm}

		\begin{proof}
			We recall that, under these hypotheses, $\Phi^{-1}(p)$ has always finitely many connected components and the upper semicontinuous function $\tilde N\ge 1$ is everywhere finite (see Remark \ref{alsoloc}). It suffices to show that $\Delta\Phi=0$: once this is done, since $\Phi$ is necessarily nonconstant we can pick any $z_0\in\Omega$ such that $\nabla\Phi(z_0)\neq 0$ and, by weak conformality, there is an $r>0$ such that $\restr{\de_z\Phi}{B_r^2(z_0)}=0$ or $\restr{\de_{\bar z}\Phi}{B_r^2(z_0)}=0$; the statement then follows by the analyticity of $\de_z\Phi$ and $\de_{\bar z}\Phi$.

			We further make the following assumptions, which will be dropped in Step 4 below:
			\begin{itemize}
				\item[(i)] $\Phi$ extends continuously to $\bar\Omega$ and $\Phi(\de\Omega)\cap\Phi(\Omega)=\emptyset$;
				\item[(ii)]  $\Phi(\Omega)\subseteq\C$ is open and the varifold $\vfd_\Omega$ equals $\tilde N(x_0)\vfd(\Phi(\Omega))$, for some $x_0$ in $\Omega$, $\vfd(\Phi(\Omega))$ denoting the canonical varifold associated to $\Phi(\Omega)$.
			\end{itemize}
			We show that in this situation the theorem holds, by strong induction on $\tilde N(x_0)$. Notice that $\tilde N(x_0)$ is necessarily an integer, since $\vfd_\Omega$ has integer multiplicity.

			\emph{Step 1.} If $\tilde N(x_0)=1$ then $\tilde N=1$ everywhere: indeed, for every $z\in\Omega$ and every $K_z\subseteq\omega\cptsub\Omega$ we have
			\[ 1\le \tilde N(z)\le\lim_{s\to 0}\frac{\norm{\vfd_\omega}(B_s^2(\Phi(z)))}{\pi s^2}\le\lim_{s\to 0}\frac{\norm{\vfd_\Omega}(B_s^2(\Phi(z)))}{\pi s^2}=1. \]
			By Proposition \ref{nisusc} we can replace $N$ with $\tilde N$ without affecting the stationarity of $(\Omega,\Phi,N)$, hence by Theorem \ref{constn} we have $\Delta\Phi=0$.

			Assume now $\tilde N(x_0)>1$. Fix any $y\in\Omega$ and choose a point $y_i$ in every connected component $K_i$ of $\Phi^{-1}(\Phi(y))$. Choosing disjoint neighborhoods $K_i\subseteq\omega_i\cptsub\Omega$ we have
			\[ \sum_i\tilde N(y_i)=\lim_{s\to 0}\sum_i\frac{\norm{\vfd_{\omega_i}}(B_s^2(\Phi(y)))}{\pi s^2}=\lim_{s\to 0}\frac{\norm{\vfd_\Omega}(B_s^2(\Phi(y)))}{\pi s^2}=\tilde N(x_0), \]
			since $\Phi(y)\nin\bar{\Phi(\Omega\setminus\bigcup_i\omega_i)}$. We deduce that the following dichotomy is true: for any $y\in\Omega$, either $\tilde N(y)=\tilde N(x_0)$ and $\Phi^{-1}(\Phi(y))$ is connected, or $\tilde N(y)\le\tilde N(x_0)-1$ and $\Phi^{-1}(\Phi(y))$ has at least two components.
			We can assume that $\tilde N$ is not identically equal to $\tilde N(x_0)$, since otherwise we are done as in the base case of the induction.

			\emph{Step 2.} We claim that, by inductive hypothesis, $\Phi$ is holomorphic or antiholomorphic on each connected component of the set
			\[ \Omega_0:=\big\{\tilde N\le\tilde N(x_0)-1\big\}=\big\{\tilde N<\tilde N(x_0)\big\}, \]
			which is open by Proposition \ref{nisusc}. For any $y_0\in\Omega_0$ we can take an open set $K_{y_0}\subseteq\omega\cptsub\Omega_0$ with $\bar\omega$ disjoint from $\Phi^{-1}(\Phi(y_0))\setminus K_{y_0}$. Possibly replacing $\omega$ with the connected component of $\omega\setminus\Phi^{-1}(\Phi(\de\omega))$ containing $y_0$, we observe that $\omega$ satisfies the same hypotheses as $\Omega$, as well as (i)--(ii): the only nontrivial task is to check (ii), which we do below.

			By the constancy theorem \cite[Theorem~41.1]{simon} applied to $\vfd_\omega$, which is stationary in $\C\setminus\Phi(\de\omega)$, the varifold $\vfd_\omega$ equals a nontrivial constant multiple of $\vfd(W)$, where $W$ is the connected component of $\C\setminus\Phi(\de\omega)$ containing the connected set $\Phi(\omega)$. Since $\Phi(\omega)$ is relatively closed in $W$, we deduce $W=\Phi(\omega)$. Finally, by definition of $\tilde N(y_0)$ we must have $\vfd_\omega=\tilde N(y_0)\vfd(W)$.
			Thus, the inductive hypothesis applies and we deduce $\Delta\Phi=0$ on $\omega$. Since $y_0$ was arbitrary, we get $\Delta\Phi=0$ on $\Omega_0$ and our claim is established.

			\emph{Step 3.} Notice that $\Phi(\Omega_0)$ is nonempty and open, being $\Phi$ nonconstant on every connected component of $\Omega_0$. We call $D\subset\Omega_0$ the relatively closed, discrete set of points where $\nabla\Phi$ vanishes. The map $\restr{\Phi}{\Omega_0}:\Omega_0\to\Phi(\Omega_0)$ is proper, thanks to the fact that $\Phi(\Omega_0)$ and $\Phi(\bar\Omega\setminus\Omega_0)$ are disjoint, so $\Phi(D)$ is a relatively closed, discrete subset of $\Phi(\Omega_0)$. Hence, $D':=\Phi^{-1}(\Phi(D))$ is still relatively closed in $\Omega_0$ and $\restr{\Phi}{\Omega_0\setminus D'}$, being a proper local diffeomorphism onto $\Phi(\Omega_0)\setminus\Phi(D)$, is a covering map.

			Let $\Omega_{max}:=\big\{\tilde N=\tilde N(x_0)\big\}$, which is closed in $\Omega$. Due to Lemma \ref{removab}, we can assume that $\Phi(\Omega_{max})$ is uncountable.
			Observe that $\Phi(\Omega_{max})$ is relatively closed in the open set $\Phi(\Omega)$, being $\Phi$ a proper map, and $\Phi(\Omega)=\Phi(\Omega_0)\sqcup\Phi(\Omega_{max})$ by the dichotomy of Step 1.
			Take two distinct points $p,q\in\Phi(\Omega_0)$ and choose any ball $p,q\nin\obar B\subseteq\Phi(\Omega)$ such that $\Phi(\Omega_{max})\cap B$ is uncountable.
			We consider a foliation of curves on the connected set $\Phi(\Omega)$ as in the picture (which illustrates the position of $p,q,B$ up to a diffeomorphism of $\Phi(\Omega)$).

			\begin{center}
				\begin{overpic}[width=11cm]{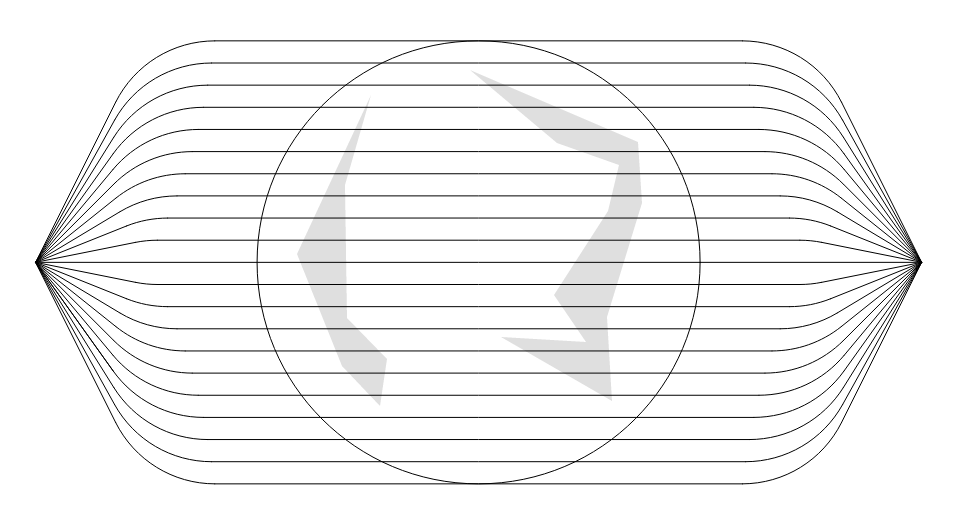}
					\put(0,26){$p$}
					\put(98,26){$q$}
					\put(39,44){$B$}
					\put(28,25.5){\scalebox{0.85}{$\Phi(\Omega_{max})$}}
					\put(55,22.5){\scalebox{0.85}{$\Phi(\Omega_{max})$}}
				\end{overpic}
			\end{center}

			We can assume that uncountably many of these curves intersect $\Phi(\Omega_{max})$: if this does not happen, it means that uncountably many points of $\Phi(\Omega_{max})$ lie on a single horizontal segment in $B$, so it suffices to apply a diffeomorphism which rotates $B$ slightly.
			Uncountably many such curves do not intersect $\Phi(D)=\Phi(D')$, as well. For any such good curve $\gamma:[0,1]\to\Phi(\Omega)$ (with $\gamma(0)=p$, $\gamma(1)=q$) we let
			\[ a:=\min\set{t:\gamma(t)\in\Phi(\Omega_{max})},\quad 1-b:=\max\set{t:\gamma(t)\in\Phi(\Omega_{max})}. \]
			We clearly have $0<a\le 1-b<1$.
			We can lift $\restr{\gamma}{[a/3,2a/3]}$ to two curves $\gamma_1$, $\gamma_2$ in $\Omega_0\setminus D'$ and $\restr{\gamma}{[1-2b/3,1-b/3]}$ to a curve $\gamma_3$ in $\Omega_0\setminus D'$, thanks to the dichotomy observed in Step 1 and the fact that $\restr{\Phi}{\Omega_0\setminus D'}$ is a covering map.

			Finally, the compact set $K:=\Phi^{-1}(\gamma([2a/3,1-2b/3]))$ is connected: assume by contradiction that it splits into two disjoint compact sets $A\sqcup B$. For each point $z\in\gamma([2a/3,1-2b/3])$ the fiber $\Phi^{-1}(z)$ lies either in $A$ or in $B$: this is clear if $z\in\Phi(\Omega_{max})$, since then $\Phi^{-1}(z)$ is connected; if $z\nin\Phi(\Omega_{max})$ we can travel $\gamma$ back or forward until we hit a point $w\in\Phi(\Omega_{max})$ and the corresponding lifted curves will necessarily accumulate against $\Phi^{-1}(w)$ (thanks to the properness of $\Phi$), so they lie either all in $A$ or all in $B$. We infer that $\gamma([2a/3,1-2b/3])=\Phi(A)\sqcup\Phi(B)$, which contradicts the connectedness of $[2a/3,1-2b/3]$.

			Thus any such good curve produces a generalized triod $(K,\gamma_1,\gamma_2,\gamma_3)$ and these triods are disjoint from each other. Since there are uncountably many such triods, this contradicts Lemma \ref{triods}. The inductive proof is complete.

			\emph{Step 4.} We now drop the extra assumptions (i)--(ii). This is done with the same argument of Step 2: for any $y_0\in\Omega$ we can find a neighborhood $\omega\cptsub\Omega$ satisfying (i)--(ii), hence $\Delta\Phi=0$ on $\omega$. We deduce that $\Delta\Phi=0$ on all of $\Omega$.
		\end{proof}

		\begin{corollary}[constant multiplicity in codimension 0]\label{constantn}
			Under the same hypotheses, $N$ is a.e. constant.
		\end{corollary}

		\begin{proof}
			Since $U:=\set{\nabla\Phi\neq 0}\subseteq\Omega$ is connected, it suffices to show the claim locally in $U$. Fix $z_0\in U$. We can find a connected open neighborhood $\omega\cptsub U$ such that $\Phi$ is injective on $\bar\omega$. Arguing as in the proof of Theorem \ref{planarreg}, $\Phi(\omega)$ is open and $\vfd_\omega=\theta\vfd(\Phi(\omega))$ for some $\theta$. By definition of $\tilde N$ and Proposition \ref{nisusc}, $N=\tilde N=\theta$ a.e. on $\omega$.
		\end{proof}

		\subsection{Conical case} We now gradually move to the case where the varifold is contained in a finite union of planes.

		\begin{lemmaen}[regularity in a dihedral angle]\label{notwohalf}
			Let $(\Omega,\Phi,N)$ be a local parametrized stationary varifold in $\R^\envdim$ defined on a bounded connected open set $\Omega\subset\C$. Assume that $\Phi^{-1}(p)$ is compact for all $p\in\R^\envdim$ and that $\Phi$ takes values in the union of two $2$-dimensional closed half-planes $H_a$, $H_b$ with common boundary. Then $\Delta\Phi=0$.
		\end{lemmaen}

		\begin{proof}
			The idea is to straighten the two half-planes into a single plane and then apply Corollary \ref{constantn}. We can assume that $\envdim=3$ and, by Theorem \ref{planarreg}, that the two half-planes are not contained in a single plane.
			Up to translations and rotations, $\Phi(\Omega)\subseteq H_a\cup H_b$, where
			\[ H_a:=\set{\lambda v_1+\mu v_2\mid \lambda\in\R,\mu\in\brapa{0,\infty}},\  H_b:=\set{\lambda v_1+\mu v_3\mid \lambda\in\R,\mu\in\brapa{0,\infty}}, \]
			\[ v_1:=(1,0,0),\quad v_2:=(0,\cos\theta,\sin\theta),\quad v_3:=(0,-\cos\theta,\sin\theta), \]
			for some $0<\theta<\frac{\pi}{2}$. Let
			\[ S:\R^3\to\R^2=\C,\quad S(x,y,z):=\pa{x,\frac{y}{\cos\theta}} \]
			and $\Psi:=S\circ\Phi$. This map is still weakly conformal: indeed, if $x$ is a Lebesgue point for $\nabla\Phi$ and $\nabla\Phi(x)=0$, then the same holds for $\Psi$; if instead $\nabla\Phi(x)$ has full rank, then $d\Phi(x)$ takes values in the linear span of $v_1$, $v_2$, or in the linear span of $v_1$, $v_3$ (since $\liminf_{r\to 0}r^{-1}\norm{\Phi(x+ry)-\Phi(x)-r\ang{\nabla\Phi(x),y}}_{C^0(S^1)}=0$, by Lemma \ref{slicingleb}) and the claim follows from the chain rule.

			We now show that $(\Psi,N)$ is still a local parametrized stationary varifold, i.e. that
			\[ \int_\omega N\ang{\nabla(X\circ\Psi);\nabla\Psi}\,d\mathcal{L}^2=0 \]
			for any $\omega\cptsub\Omega$ and any vector field $X\in C^\infty_c(\R^2\setminus\Psi(\de\omega),\R^2)$. Let
			\[ A:=\begin{pmatrix}1 & 0 \\ 0 & \cos\theta\end{pmatrix}, \quad P:=\begin{pmatrix}1 & 0 & 0 \\ 0 & 1 & 0\end{pmatrix}, \quad Y:=P^t A^{-1}(X\circ S). \]
			Notice that, although $Y$ does not have compact support, it vanishes in a neighborhood of $\Phi(\de\omega)$. Since we know that $\Phi(\obar\omega)$ is compact, we have
			\[ \int_\omega N\ang{\nabla(Y\circ\Phi);\nabla\Phi}\,d\mathcal{L}^2=0. \]
			But, viewing $S$ also as a matrix, $\nabla Y=P^t A^{-1}(\nabla X\circ S)S$ and $P=AS$, thus
			\[ \begin{split} \ang{\nabla(Y\circ\Phi);\nabla\Phi}&=\ang{P^t A^{-1}(\nabla X\circ S\circ\Phi)\nabla(S\circ\Phi);\nabla\Phi} \\
			&=\ang{P^tA^{-1}\nabla(X\circ\Psi);\nabla\Phi} \\
			&=\ang{S^t\nabla(X\circ\Psi);\nabla\Phi} \\
			&=\ang{\nabla(X\circ\Psi);S\nabla\Phi} \\
			&=\ang{\nabla(X\circ\Psi);\nabla\Psi}. \end{split} \]
			This shows the stationarity of $(\Psi,N)$. By Corollary \ref{constantn}, $N$ is a.e. constant. By Theorem \ref{constn}, this implies $\Delta\Phi=0$.
		\end{proof}

		\begin{lemmaen}[regularity in a wedge of several half-planes]\label{nomanyhalf}
			The same conclusion holds if $\Phi$ takes values in the union of finitely many (distinct) closed half-planes $\bigcup_{i=1}^k H_i\subseteq\R^\envdim$ with a common boundary $C=\de H_i$.
		\end{lemmaen}

		\begin{proof}
			It suffices to show that $\Delta\Phi=0$ near any point $x_0\in\Omega$. By Theorem \ref{planarreg} we have $\Delta\Phi=0$ on $\Phi^{-1}(H_i\setminus C)$, for all $i$, so we can assume $\Phi(x_0)\in C$.
			By Remark \ref{hightilden}, shrinking $\Omega$ if necessary, we can further assume that $\Phi^{-1}(\Phi(y))$ is connected whenever $\tilde N(y)\ge\tilde N(x_0)-\mz$ and that $\Phi$ extends continuously to $\bar\Omega$ with $\Phi(\Omega)\cap\Phi(\de\Omega)=\emptyset$.

			By induction on $\lfloor 2\tilde N(x_0)\rfloor$, $\floor{\cdot}$ denoting the integer part, we can also assume that in this situation we have $\Delta\Phi=0$ on the open set $\big\{\tilde N<\tilde N(x_0)-\mz\big\}$ (using e.g. $\lfloor 2\tilde N(x_0)\rfloor=1$ as the base case, which is vacuously true).

			We pick any $0<r<\dist(\Phi(x_0),\Phi(\de\Omega))$ and let $\omega:=\Phi^{-1}(B_r^\envdim(\Phi(x_0)))\cptsub\Omega$, as well as $\omega_i:=\omega\cap\Phi^{-1}(H_i\setminus C)$. If $\omega_i$ is nonempty for at most two values of $i$, then we are done by Lemma \ref{notwohalf}, applied to the connected components of $\omega$. Thus, we suppose e.g. that $\omega_i\neq\emptyset$ for $i=1,2,3$.
			Applying the constancy theorem to the varifold $\vfd_{\omega_i}$, which is a nontrivial stationary varifold in $B_r^\envdim(\Phi(x_0))\cap(H_i\setminus C)$, and using the fact that $\Phi(\omega_i)$ is relatively closed in this set, we infer that
			\[ \Phi(\omega_i)=B_r^\envdim(\Phi(x_0))\cap(H_i\setminus C). \]
			for $i=1,2,3$.
			Let $C':=C\cap\Phi\Big(\big\{x\in\omega:\tilde N(x)\ge\tilde N(x_0)-\mz\big\}\Big)$, which is closed in $B_r^\envdim(\Phi(x_0))$, being $\restr{\Phi}{\omega}:\omega\to B_r^\envdim(\Phi(x_0))$ a proper map. If $C'$ is countable, then we are done by applying Lemma \ref{removab} to $\omega$.

			Otherwise, we can pick for each $c\in C'$ a segment $\beta_i([0,1])\subseteq B_r^\envdim(\Phi(x_0))\cap H_i$ perpendicular to $C$ with $\beta_i(1)=c$ (for $i=1,2,3$). Apart from (at most) countably many exceptions, these segments do not intersect the singular values of $\restr{\Phi}{\omega_i}$. So, arguing as in the proof of Theorem \ref{planarreg}, the curves $\beta_i([0,1/2])$ can be lifted to smooth regular curves $\gamma_i$ in $\omega_i$ and, setting $K:=\Phi^{-1}(c)\cup\bigcup_{i=1}^3\beta_i([1/2,1])$, $(K,\gamma_1,\gamma_2,\gamma_3)$ is a generalized triod. This gives an uncountable family of disjoint generalized triods, contradicting Lemma \ref{triods}.
		\end{proof}

		\begin{thm}[regularity in a polyhedral cone]\label{conicalreg}
			Let $(\Omega,\Phi,N)$ be a local parametrized stationary varifold in $\R^\envdim$, with $\Omega$ connected. Assume that $\Phi$ takes values in the union of finitely many (distinct) $2$-dimensional planes $\Sigma_1,\dots,\Sigma_k$ passing through the origin. Then $\Phi$ takes values in a single plane $\Sigma_{i_0}$ and, once we identify it with $\C$, it is holomorphic or antiholomorphic (but possibly constant).
		\end{thm}

		\begin{proof}
			Assume without loss of generality that $\Phi$ is nonconstant.
			It suffices to show that $\Delta\Phi=0$: then, by weak conformality, $\Phi(\Sigma)$ cannot be contained in a finite union of lines, i.e. for some $i_0$ we have $\Phi^{-1}\pa{\Sigma_{i_0}\setminus\bigcup_{i\neq i_0}\Sigma_i}\neq\emptyset$; thus on this open set we have $\pi_{\Sigma_{i_0}^\perp}\Phi=0$ and we deduce that this holds on all of $\Omega$ by analyticity. The statement then follows by weak conformality, as in the proof of Theorem \ref{planarreg}.

			Fix now $x_0\in\Omega\setminus\Phi^{-1}(0)$ and let $J:=\set{j:\Phi(x_0)\in\Sigma_j}$. We pick any radius $r<\dist(x_0,\de\Omega)$ such that
			$\Phi(\bar B_r^2(x_0))$ intersects only the planes $\Sigma_j$ with $j\in J$.
			Notice that $\bigcup_{j\in J}\Sigma_j$ is a finite union of half-planes with common boundary $\set{t\Phi(x_0):t\in\R}$. We assume that the weak gradient coincides with the classical one on the biggest open subset of $\Omega$ where $\Phi$ is smooth and we call $\goodrk$ the set of Lebesgue points for $\nabla\Phi$ where $\nabla\Phi\neq 0$.

			For $x\in B_r^2(x_0)\cap\goodrk$ we have $\liminf_{s\to 0}s^{-1}\norm{\Phi(x+sy)-\Phi(x)-s\ang{\nabla\Phi(x),y}}_{C^0(S^1)}=0$, by Lemma \ref{slicingleb}. Hence, we can find an open neighborhood $\omega\subset B_r^2(x_0)$ of $x$ such that $\Phi(x)\nin\Phi(\de\omega)$.
			Possibly replacing $\omega$ with the connected component of $\omega\setminus\Phi^{-1}(\Phi(\de\omega))$ containing $x$, we can even assume $\Phi(\omega)\cap\Phi(\de\omega)=\emptyset$. Hence we can apply Lemma \ref{nomanyhalf} on $\omega$, obtaining $\Delta\Phi=0$ near $x$. In particular, this shows that $\goodrk\setminus\Phi^{-1}(0)$ is open and $\Delta\Phi=0$ on $\goodrk\setminus\Phi^{-1}(0)$.

			Using Fubini's theorem and \cite[Theorem~4.21]{evans}, we can pick an $r'<r$ such that the map $\restr{\Phi}{\de B_{r'}^2(x_0)}$ is absolutely continuous (with weak derivative given by the chain rule) and $\int_{\de B_{r'}^2(x_0)\setminus\goodrk}\abs{\nabla\Phi}\,d\mathcal{H}^1=0$. For any relatively open subset $U\subseteq\de B_{r'}^2(x_0)$ containing $\de B_{r'}^2(x_0)\setminus\goodrk$ we have
			\[ \mathcal{H}^1(\Phi(\de B_{r'}^2(x_0)\cap U))\le\int_U\abs{\nabla\Phi}\,d\mathcal{H}^1, \]
			by definition of $\mathcal{H}^1$. Since $U$ is arbitrary, we deduce
			\[ \mathcal{H}^1(K)=0,\quad K:=\Phi(\de B_{r'}^2(x_0)\setminus\goodrk). \]

			Fix now any $x\in B_{r'}^2(x_0)\setminus\Phi^{-1}(K)$. Assume $\de B_s^2(x_0)\cap\Phi^{-1}(\Phi(x))\neq\emptyset$ for all $\abs{x-x_0}<s<r'$. Then we can find a sequence $s_\ell\uparrow r'$ and points $y_\ell\in\de B_{s_k}^2(x_0)$ such that $\Phi(y_\ell)=\Phi(x)$ and $y_k\to y_\infty$, for some $y_\infty\in\de B_{r'}^2(x_0)$. Necessarily we have $y_\infty\in\goodrk$, as $\Phi(y_\infty)=\Phi(x)\nin K$, but this contradicts the fact that $\Phi$ is injective near $y_\infty$.

			Thus there exists a radius $s<r'$ such that $x\in B_s^2(x_0)$ and $\Phi(x)\nin\Phi(\de B_s^2(x_0))$. Again we can let $\omega$ be the connected component of $B_s^2(x_0)\setminus\Phi^{-1}(\Phi(\de B_s^2(x_0)))$ containing $x$ and we can apply Lemma \ref{nomanyhalf} on $\omega$. This shows that $\Delta\Phi=0$ on $B_{r'}^2(x_0)\setminus\Phi^{-1}(K)$.

			Finally, by Lemma \ref{removab} and $\mathcal{H}^1(K)=0$, we have $\Delta\Phi=0$ on $B_{r'}^2(x_0)$. So $\Delta\Phi=0$ on $\Omega\setminus\Phi^{-1}(0)$ and we deduce that $\Delta\Phi=0$ on all of $\Omega$, again by Lemma \ref{removab}.
		\end{proof}

	\section{Blow-up of a parametrized stationary varifold}\label{blowsec}

	Let $(\Omega,\Phi,N)$ be a local parametrized stationary varifold.
	Let us fix a sequence of points $(x_k)\subseteq\Omega$ and a sequence of radii $(r_k)$ such that
	$0<r_k<\mz\dist(x_k,\de\Omega)$. We let
	\[ \ell_k^2:=\int_{B_{r_k}^2(x_k)}\abs{\nabla\Phi}^2\,d\mathcal{L}^2,\quad\Phi_k:=\ell_k^{-1}(\Phi(x_k+r_k\cdot)-\Phi(x_k)),\quad N_k:=N(x_k+r_k\cdot), \]
	\[ \nu_k:=\mz N_k\abs{\nabla\Phi_k}^2\uno_{B_2^2(0)}\mathcal{L}^2,\quad\mu_k:=(\Phi_k)_*\nu_k. \]
	Notice that the functions $\Phi_k,N_k$ are defined on $B_2^2(0)$, so the definition of the measures $\nu_k$ and $\mu_k$, on $B_2^2(0)$ and $\R^\envdim$ respectively, makes sense. Throughout the section we will assume that there exist two constants $\enrat,\massbd\ge 1$ such that
	\begin{itemize}
		\item $0<\int_{B_{2r_k}^2(x_k)}\abs{\nabla\Phi}^2\,d\mathcal{L}^2\le\enrat\int_{B_{r_k}^2(x_k)}\abs{\nabla\Phi}^2\,d\mathcal{L}^2$; \hfill\eqlabel{eq:enbd}
		\item $\limsup_{k\to\infty}\mu_k(B_s^\envdim(p))\le\massbd\pi s^2$ for all $s>0$ and all $p\in\R^\envdim$; \hfill\eqlabel{eq:massbd}
		\item $\ell_k\to 0$.  \hfill\eqlabel{eq:shrinking}
	\end{itemize}

	We will show the following result.
	\begin{thm}[parametrized blow-up]\label{blow-up}
		Up to subsequences, there exist a map $\Phi_\infty\in W^{1,2}(B_2^2(0),\R^\envdim)$, a function $N_\infty\in L^\infty(B_2^2(0),\N\setminus\set{0})$ and a
		quasiconformal homeomorphism $\varphi_\infty\in W^{1,2}(B_2^2(0),\Omega_\infty)$ (for some bounded open set $\Omega_\infty\subseteq\C$), with $\varphi_\infty(0)=0$, such that
		\[ \Phi_k\to\Phi_\infty\quad\mbox{in }C^0_{loc}(B_2^2(0),\R^\envdim),\qquad\nabla\Phi_k\weakto\nabla\Phi_\infty\quad\text{in }L^2(B_2^2(0),\R^{\envdim\times 2}), \]
		\[ \mz N_k\abs{\nabla\Phi_k}^2\mathcal{L}^2\weakstarto N_\infty\abs{\de_1\Phi_\infty\wedge\de_2\Phi_\infty}\mathcal{L}^2\quad\text{as Radon measures}.
		\]
		Moreover, $\Phi_\infty\circ\varphi_\infty^{-1}$ is weakly conformal and $(\Omega_\infty,\Phi_\infty\circ\varphi_\infty^{-1},N_\infty\circ\varphi_\infty^{-1})$ is a local parametrized stationary varifold in $\R^\envdim$.
	\end{thm}

	We refer the reader to \cite[Chapter~4]{imayoshi} and \cite{lehto} for the theory of quasiconformal homeomorphisms in the plane. Before proving this theorem we shall establish a number of intermediate results. Many arguments are similar to those used in \cite{rivminmax}.

	First of all, since $\int_{B_2^2(0)}\abs{\nabla\Phi_k}^2\,d\mathcal{L}^2\le\enrat$ and $\nu_k(B_2^2(0))\le\mz\enrat\norm{N}_{L^\infty}$, up to subsequences there exists $\Phi_\infty\in W^{1,2}(B_2^2(0),\R^\envdim)$ such that $\Phi_k\weakto\Phi_\infty$ in $W^{1,2}(B_2^2(0))$ and there exists a finite Radon measure $\nu_\infty$ on $B_2^2(0)$ such that $\nu_k\weakstarto\nu_\infty$ in $B_2^2(0)$.
	We can also assume that, for all $j\ge 0$,
	\[ \mu_{k,j}:=(\Phi_k)_*\pa{\uno_{B_{2-2^{-j}}^2(0)}\nu_k}\weakstarto\mu_{\infty,j}\quad\text{as }k\to\infty \]
	in $\R^\envdim$, for some finite measure $\mu_{\infty,j}$.
	Since $\mu_{\infty,j}\le\mu_{\infty,j+1}$ and $\mu_{\infty,j}(\R^\envdim)\le \mz\enrat\norm{N}_{L^\infty}$, the measure $\mu_\infty:=\lim_{j\to\infty}\mu_{\infty,j}$
	is defined and is again finite.

	\begin{lemmaen}[uniform convergence $\bm{\Phi_k\to\Phi_\infty}$]\label{c0convandhull}
		The measure $\nu_\infty$ is absolutely continuous with respect to $\mathcal{L}^2$, i.e. $\nu_\infty=m\mathcal{L}^2$ for some nonnegative $m\in L^1(B_2^2(0))$. Moreover, $\Phi_\infty$ is continuous and $\Phi_k\to\Phi_\infty$ in $C^0_{loc}(B_2^2(0),\R^\envdim)$. Finally, for any open subset $\omega\cptsub B_2^2(0)$,
		\[ \Phi_\infty(\bar\omega)\subseteq\operatorname{conv}(\Phi_\infty(\de\omega)), \]
		where $\operatorname{conv}(\cdot)$ denotes the convex hull.
	\end{lemmaen}

	\begin{proof}
		We introduce the \emph{oscillation set}
		\[ \mathcal{O}:=\bigg\{x\in\supp{\nu_\infty}:\liminf_{r\to 0}\frac{\int_{B_{2r}^2(x)}\abs{\nabla\Phi_\infty}^2\,d\mathcal{L}^2}{\nu_\infty(B_r^2(x))}=0\bigg\}. \]

		\emph{Step 1.} We show that $\nu_\infty$ is absolutely continuous with respect to $\mathcal{L}^2$ on the Borel set $B_2^2(0)\setminus\mathcal{O}$.
		Let $E\subseteq B_2^2(0)\setminus\mathcal{O}$ be a Borel set with $\mathcal{L}^2(E)=0$. It suffices to show that $\nu_\infty(K)=0$ for any compact subset $K\subseteq E\cap\supp{\nu_\infty}$ (since this implies that $\nu_\infty(E\cap\supp{\nu_\infty})=0$ and thus $\nu_\infty(E)=0$, as required). We define the Borel sets
		\[ F_j:=\bigg\{x\in K:\inf_{0<r\le\bar r}\frac{\int_{B_{2r}^2(x)}\abs{\nabla\Phi_\infty}^2\,d\mathcal{L}^2}{\nu_\infty(B_r^2(x))}\ge 2^{-j}\bigg\},\quad\bar r:=\mz\dist(K,\de B_2^2(0)) \]
		and observe that $K=\bigcup_j F_j$. Fix $j$ and an open set $E\subseteq V\subseteq B_2^2(0)$. Letting $r_V:=\mz\dist(K,\R^2\setminus V)\le\bar r$, we choose a maximal (finite) subset $\set{x_i}$ of $F_j$ such that $\abs{x_i-x_{i'}}\ge r_V$ for $i\neq i'$. We have
		\[ F_j\subseteq\bigcup_i B_{r_V}^2(x_i),\quad \sum_i\uno_{B_{2r_V}^2(x_i)}\le\mathfrak{N} \]
		for some universal constant $\mathfrak{N}$.
		We deduce that
		\[ \nu_\infty(F_j)\le\sum_i\nu_\infty(B_{r_V}^2(x_i))\le 2^j\sum_i\int_{B_{2r_V}^2(x_i)}\abs{\nabla\Phi_\infty}^2\,d\mathcal{L}^2
		\le 2^j\mathfrak{N}\int_V\abs{\nabla\Phi_\infty}^2\,d\mathcal{L}^2. \]
		Letting $V$ range along a sequence of open sets $V_\ell\supseteq E$ with $\mathcal{L}^2(V_\ell)\to 0$, we deduce $\nu_\infty(F_j)=0$. Hence,
		\[ \nu_\infty(K)\le\sum_j\nu_\infty(F_j)=0. \]

		\emph{Step 2.} We show that $\mathcal{O}=\emptyset$.
		Fix any $x\in B_2^2(0)$ and any $r<\mz\dist(x,\de B_2^2(0))$. Using Lemma \ref{equibded} and Lemma \ref{slicing} we select a radius $r'\in(r,2r)$ such that $\restr{\Phi_{k_i}}{\de B_{r'}^2(x)}\to\restr{\Phi_\infty}{\de B_{r'}^2(x)}$ in $L^\infty$, for some subsequence $(\Phi_{k_i})$, and
		\begin{equation} \label{eq:diamK} \diam\Phi_\infty(\de B_{r'}^2(x))\le\sqrt{4\pi}\pa{\int_{B_{2r}^2(x)}\abs{\nabla\Phi_\infty}^2\,d\mathcal{L}^2}^{1/2} \end{equation}
		(we are implicitly referring to the continuous representative of $\restr{\Phi_\infty}{\de B_{r'}^2(x)}$). Since $(\Omega,\Phi,N)$ is a local parametrized stationary varifold in $\subman$, the varifolds $\vfd_{k_i}$ issued by $(\Phi_{k_i},N_{k_i})$ from the domain $B_{r'}^2(x)$ have generalized mean curvature bounded by $O(\ell_{k_i})$ (in $L^\infty$) in $\R^\envdim\setminus\Phi_{k_i}(\de B_{r'}^2(x))$. As a consequence of assumption \eqref{eq:shrinking}, up to further subsequences they converge to a varifold $\vfd_\infty$ (in $\R^\envdim$) which is stationary in $\R^\envdim\setminus\Phi_\infty(\de B_{r'}^2(0))$.

		Moreover, by Proposition \ref{contandsupp}, $0=\Phi_k(0)\in\supp{\norm{\vfd_k}}$ unless $\supp{\norm{\vfd_k}}=\emptyset$. Using estimate \eqref{eq:suppest} we infer that the sets $\supp{\norm{\vfd_k}}$ are all included in a unique compact set. It follows that $\vfd_\infty$ has compact support, hence by \cite[Theorem 19.2]{simon} (which applies to general stationary varifolds) $\supp{\vfd_\infty}\subseteq K:=\operatorname{conv}(\Phi_\infty(\de B_{r'}^2(x)))$. It follows that
		\begin{equation}\label{eq:convtoK} \sup_{z\in\bar B_{r'}^2(x)}\dist(\Phi_{k_i}(z),K)\to 0: \end{equation}
		if this were not the case, up to subsequences we could find $z_{k_i}\in B_{r'}^2(x)$ such that $\dist(\Phi_{k_i}(z_{k_i}),K)\ge\epsilon$. Eventually $\Phi_{k_i}(x_{k_i})\in\supp{\norm{\vfd_{k_i}}}$ (by Proposition \ref{contandsupp}, since eventually $\Phi_{k_i}$ must be nonconstant on $B_{r'}^2(x)$), so we can assume that $\Phi_{k_i}(x_{k_i})\to p_\infty$ and
		\[ \norm{\vfd_\infty}(\bar B_{\epsilon/2}^\envdim(p_\infty))\ge\limsup_{i\to\infty}\norm{\vfd_{k_i}}(\bar B_{\epsilon/2}^\envdim(\Phi_{k_i}(x_{k_i})))\ge\pi\frac{\epsilon^2}{4}, \]
		thanks to the monotonicity formula and the fact that $\bar B_{\epsilon/2}^\envdim(\Phi_{k_i}(x_{k_i}))\cap\Phi_{k_i}(\de B_{r'}^2(x))=\emptyset$ eventually. This, however, contradicts the fact that $\bar B_{\epsilon/2}^\envdim(p_\infty)\cap K=\emptyset$. Using \eqref{eq:convtoK} and \eqref{eq:massbd} we deduce that
		\[ \nu_\infty(B_{r'}^2(x))\le\liminf_{i\to\infty}\nu_{k_i}(B_{r'}^2(x))
		\le\liminf_{i\to\infty}\mu_{k_i}(\Phi_{k_i}(B_{r'}^2(x)))\le C''\pi(\diam K)^2. \]
		From \eqref{eq:diamK} and the fact that the convex hull preserves the diameter, we have $(\diam K)^2\le 4\pi\int_{B_{2r}^2(x)}\abs{\nabla\Phi_\infty}^2\,d\mathcal{L}^2$. Hence,
		\[ \liminf_{r\to 0}\frac{\int_{B_{2r}^2(x)}\abs{\nabla\Phi_\infty}^2\,d\mathcal{L}^2}{\nu_\infty(B_r^2(x))}\ge\frac{1}{4\pi^2 C''}. \]
		It follows that $\mathcal{O}=\emptyset$.

		\emph{Step 3.} We show that $\Phi_\infty$ has a continuous representative. Since $\Phi_{k_i}\to\Phi_\infty$ in $L^2(B_2^2(0),\R^\envdim)$, from \eqref{eq:convtoK} we infer that $\Phi_\infty(z)\in K$ for a.e. $z\in B_r^2(x)$. In particular, this must happen whenever $z$ is a Lebesgue point. This, together with the estimate for $\diam K$, proves that $\Phi_\infty$ is locally uniformly continuous on the set of its Lebesgue points, hence it has a continuous representative.

		\emph{Step 4.} Assume now by contradiction that $\Phi_k$ does not converge locally uniformly to (the continuous representative of) $\Phi_\infty$. Then we can find a subsequence $\Phi_{k_i}$ and points $x_{k_i}$, lying in a compact subset of $B_2^2(0)$, such that
		\[ \abs{\Phi_{k_i}(x_{k_i})-\Phi_\infty(x_{k_i})}\ge\epsilon. \]
		We can further assume that $x_{k_i}\to x_\infty\in B_2^2(0)$. Let $r<\mz\dist(x_\infty,\de B_2^2(0))$ be such that $4\pi\int_{B_{2r}^2(x_\infty)}\abs{\nabla\Phi_\infty}^2\,d\mathcal{L}^2\le\frac{\epsilon^2}{4}$. Up to subsequences, we can repeat the argument of the two previous steps and conclude that
		\[ \dist(\Phi_{k_i}(x_{k_i}),K)\to 0,\quad\Phi_\infty(x_\infty)\in K, \]
		where again $K:=\operatorname{conv}(\Phi_\infty(\de B_{r'}^2(x_\infty)))$ for a suitable $r'\in(r,2r)$. In particular we have $\limsup_{i\to\infty}\abs{\Phi_{k_i}(x_{k_i})-\Phi_\infty(x_\infty)}\le\diam K\le\frac{\epsilon}{2}$, which is the desired contradiction.

		\emph{Step 5.} Finally, let us turn to the last part of the statement. We can assume that $\omega$ is connected. We already know that $\Phi_k\to\Phi_\infty$ uniformly on $\de\omega$, so we can repeat the argument of Step 2, with $B_{r'}^2(x)$ replaced by $\omega$, and conclude that $\dist(\Phi_k(z),\operatorname{conv}(\Phi_\infty(\de\omega)))\to 0$ for all $z\in\omega$, from which it follows that $\Phi_\infty(z)\in\operatorname{conv}(\Phi_\infty(\de\omega))$.
	\end{proof}

	\begin{lemmaen}[convergence of localized measures in the target]\label{pushforward}
		We have $\mu_\infty=(\Phi_\infty)_*\nu_\infty$ and, for any $\omega\cptsub B_2^2(0)$ with $\mathcal{L}^2(\de\omega)=0$,
		\[ (\Phi_k)_*(\uno_{\omega}\nu_k)\weakstarto(\Phi_\infty)_*(\uno_{\omega}\nu_\infty) \]
		as Radon measures in $\R^\envdim$.
	\end{lemmaen}

	\begin{proof}
		We first show the second statement. Consider any nonnegative $\rho\in C^0_c(\R^\envdim)$.
		By Lemma \ref{c0convandhull} we have $\nu_\infty(\de\omega)=0$, so approximating $\uno_\omega\rho(\Phi_\infty)$ from above and below by functions in $C^0_c(B_2^2(0))$ we get
		\[ \int_{B_2^2(0)}\rho(\Phi_\infty)\uno_\omega\,d\nu_k\to\int_{B_2^2(0)}\rho(\Phi_\infty)\uno_\omega\,d\nu_\infty. \]
		Moreover, thanks to the local uniform convergence $\Phi_k\to\Phi_\infty$ and $\nu_k(B_2^2(0))\le\mz C'\norm{N}_{L^\infty}$,
		\[ \bigg|\int_{B_2^2(0)}\rho(\Phi_k)\uno_\omega\,d\nu_k-\int_{B_2^2(0)}\rho(\Phi_\infty)\uno_\omega\,d\nu_k\bigg|\le\mz C'\norm{N}_{L^\infty}\norm{\rho(\Phi_k)-\rho(\Phi_\infty)}_{L^\infty(\omega)}\to 0 \]
		as $k\to\infty$ and the claim follows, since
		\[ \int_{\R^\envdim}\rho\,d(\Phi_k)_*(\uno_\omega\nu_k)=\int_{B_2^2(0)}\rho(\Phi_k)\uno_\omega\,d\nu_k\to\int_{B_2^2(0)}\rho(\Phi_\infty)\uno_\omega\,d\nu_\infty=\int_{\R^\envdim}\rho\,d(\Phi_\infty)_*(\uno_\omega\nu_\infty). \]
		Choosing $\omega=B_{2-2^{-j}}^2(0)$ we get
		\[ \int_{\R^\envdim}\rho\,d\mu_{\infty,j}=\lim_{k\to\infty}\int_{\R^\envdim}\rho\,d\mu_{k,j}=\int_{B_{2-2^{-j}}^2(0)}\rho(\Phi_\infty)\,d\nu_\infty. \]
		The first statement now follows by letting $j\to\infty$.
	\end{proof}

	\begin{lemmaen}[behaviour at bad points]\label{smallrank}
		Let $\mathcal{G}'\subseteq B_2^2(0)$ denote the set of points $z$ where $\nabla\Phi_\infty(z)$ has full rank and both $\media_{B_r^2(z)}\abs{m-m(z)}\,d\mathcal{L}^2$ and $\media_{B_r^2(z)}\abs{\nabla\Phi_\infty-\nabla\Phi_\infty(z)}^2\,d\mathcal{L}^2$ are infinitesimal as $r\to 0$.
		Then $m=0$ and $\nabla\Phi_\infty=0$ a.e. on $B_2^2(0)\setminus\mathcal{G}'$.
	\end{lemmaen}

	\begin{proof}
		Let $\good$ be the set of Lebesgue points for $\nabla\Phi_\infty$. It suffices to show that $m=0$ a.e. on $\good\setminus\mathcal{G}'$ and $\abs{\nabla\Phi_\infty}^2\le 2m$ a.e. By Lemma \ref{rectif} and the area formula,
		\[ \mathcal{H}^2(\Phi_\infty(\good\setminus\mathcal{G}'))\le\int_{\good\setminus\mathcal{G}'}\abs{\de_1\Phi_\infty\wedge\de_2\Phi_\infty}\,d\mathcal{L}^2=0. \]
		We can thus cover $\Phi_\infty(\good\setminus\mathcal{G}')$ with countably many balls $B_{s_i}^\envdim(p_i)$ such that $\sum_i s_i^2$ is arbitrarily small.
		By assumption \eqref{eq:massbd}, $\mu_\infty(B_{s_i}^\envdim(p_i))\le\liminf_{k\to\infty}\mu_k(B_{s_i}^\envdim(p_i))\le\massbd\pi s_i^2$. Hence, by Lemma \ref{pushforward},
		\[ \int_{\good\setminus\mathcal{G}'}m\,d\mathcal{L}^2=\nu_\infty(\good\setminus\mathcal{G}')\le\nu_\infty\pa{\Phi_\infty^{-1}\pa{{\textstyle\bigcup_i} B_{s_i}^\envdim(p_i)}}=\mu_\infty\pa{{\textstyle\bigcup_i} B_{s_i}^\envdim(p_i)}\le\massbd\pi\sum_i s_i^2 \]
		is arbitrarily small. We deduce that $\int_{\good\setminus\mathcal{G}'}m\,d\mathcal{L}^2=0$. Moreover, for any open sets $V\cptsub W\subseteq B_2^2(0)$ we have
		\[ \int_V\abs{\nabla\Phi_\infty}^2\,d\mathcal{L}^2\le\liminf_{k\to\infty}\int_V\abs{\nabla\Phi_k}^2\,d\mathcal{L}^2\le 2\limsup_{k\to\infty}\nu_k(\bar V)\le 2\nu_\infty(\bar V)\le 2\int_W m\,d\mathcal{L}^2 \]
		and we infer that $\int_W\abs{\nabla\Phi_\infty}^2\,d\mathcal{L}^2\le 2\int_W m\,d\mathcal{L}^2$. Since $W$ is arbitrary, we deduce that $\abs{\nabla\Phi_\infty}^2\le 2m$ a.e. and the claim follows.
	\end{proof}

	\begin{lemmaen}[structure of the density $\bm{m}$]\label{ninfty}
		There exists $N_\infty\in L^\infty(B_2^2(0),\N\setminus\set{0})$ bounded above by $C''$, i.e. the constant in \eqref{eq:massbd}, and
		\[ m=N_\infty\abs{\de_1\Phi_\infty\wedge\de_2\Phi_\infty}\quad\text{a.e.} \]
		Moreover, $\Phi_\infty$ satisfies $\abs{\nabla\Phi_\infty}^2\le 2N_\infty\abs{\de_1\Phi_\infty\wedge\de_2\Phi_\infty}$ a.e.
	\end{lemmaen}

	\begin{proof}
		We let $N_\infty(z):=\frac{m(z)}{\abs{\de_1\Phi_\infty\wedge\de_2\Phi_\infty}(z)}$ if $z\in\mathcal{G}'$, otherwise we let $N_\infty(z):=1$.
		We remark that, for $z\in\mathcal{G}'$,
		\[ \begin{split} &\abs{\nabla\Phi_\infty}^2(z)=\lim_{r\to 0}\media_{B_r^2(z)}\abs{\nabla\Phi_\infty}^2\,d\mathcal{L}^2
		\le\liminf_{r\to 0}\liminf_{k\to\infty}\media_{B_r^2(z)}\abs{\nabla\Phi_k}^2\,d\mathcal{L}^2 \\
		&\le\lim_{r\to 0}\lim_{k\to\infty}\frac{2\nu_k(B_r^2(z))}{\pi r^2}
		=\lim_{r\to 0}\frac{2\nu_\infty(B_r^2(z))}{\pi r^2}=2m(z)=2N_\infty(z)\abs{\de_1\Phi_\infty\wedge\de_2\Phi_\infty}(z), \end{split} \]
		while $\nabla\Phi_\infty=0$ and $m=0$ a.e. on $B_2^2(0)\setminus\mathcal{G}'$, by Lemma \ref{smallrank}. Hence, it suffices to show that $N_\infty(z)\in\N\cap[1,C'']$ for any fixed $z\in\mathcal{G}'$.

		Up to rotations, we can assume that $\de_1\Phi_\infty$ and $\de_2\Phi_\infty$ span $P:=\R^2\times\set{0}\subseteq\R^\envdim$. Let $A:=\begin{pmatrix}\de_1\Phi_\infty^1(z) & \de_2\Phi_\infty^1(z) \\ \de_1\Phi_\infty^2(z) & \de_2\Phi_\infty^2(z)\end{pmatrix}\in GL_2(\R)$.
		We introduce the notation $\ball_s^2(z):=z+A^{-1}(B_s^2(0))$.

		\emph{Step 1.} We first show that
		\begin{equation}\label{eq:negcpts} \limsup_{r\to 0}\limsup_{k\to\infty}r^{-2}\int_{\ball_r^2(z)}N_k\big|\nabla\Phi_k^j\big|^2\,d\mathcal{L}^2=0, \end{equation}
		for any fixed $j=3,\dots,\envdim$. By Lemma \ref{slicingleb} (applied to $\Phi_\infty(z+A^{-1}y)-\Phi_\infty(z)-(y,0)$), for a fixed $r>0$ we can find $r'\in(r,2r)$ and $r''\in(5r,6r)$ such that
		\begin{equation}\label{eq:sliceapp} \begin{split}
		&\Phi_\infty(z+A^{-1}r'y)=\Phi_\infty(z)+r'(y,0)+o(r), \\
		&\Phi_\infty(z+A^{-1}r''y)=\Phi_\infty(z)+r''(y,0)+o(r) \end{split} \end{equation}
		for $y\in S^1$.
		We can assume that $r$ is so small that the error terms in \eqref{eq:sliceapp} are smaller than $r$. In particular,
		\[ \Phi_\infty(\de\ball_{r'}^2(z))\subseteq B_{3r}^\envdim(\Phi_\infty(z)),\quad \Phi_\infty(\de\ball_{r''}^2(z))\cap B_{4r}^\envdim(\Phi_\infty(z))=\emptyset. \]
		Lemma \ref{c0convandhull} gives $\Phi_\infty(\bar\ball_r^2(z))\subseteq B_{3r}^\envdim(\Phi_\infty(z))$ and thus the same holds for $\Phi_k$ eventually.
		Pick any $\chi\in C^\infty_c(B_{4r}^\envdim(\Phi_\infty(z)))$ such that $\chi\equiv 1$ on $B_{3r}^\envdim(\Phi_\infty(z))$ and $|\nabla\chi|\le\frac{2}{r}$.
		Applying the stationarity of the parametrized varifold $(B_2^2(0),\Phi_k,N_k)$, on the domain $\ball_{r''}^2(z)$ and against the vector field $F(p):=\chi(p)(p_j-\Phi_\infty^j(z)) e_j$, we obtain
		\[ \begin{split} \int_{\ball_{r''}^2(z)}N_k\,\chi(\Phi_k)|\nabla\Phi_k^j|^2\,d\mathcal{L}^2&=\ell_k\int_{\ball_{r''}^2(z)}N_k\,F(\Phi_k)\cdot A(\Phi(x_k+r_k\cdot))(\nabla\Phi_k,\nabla\Phi_k)\,d\mathcal{L}^2 \\
		&\fantasma{=}-\int_{\ball_{r''}^2(z)}N_k(\Phi_k^j-\Phi_\infty^j(z))\ang{\nabla(\chi(\Phi_k)),\nabla\Phi_k^j}\,d\mathcal{L}^2. \end{split} \]
		The left-hand side is bounded below by $\int_{\ball_{r}^2(z)}N_k|\nabla\Phi_k^j|^2\,d\mathcal{L}^2$, while the first term in the right-hand side is bounded by
		$o_k(1)\nu_k(\ball_{6r}^2(z))$, thanks to \eqref{eq:shrinking}. Thus,
		\[ \limsup_{k\to\infty}r^{-2}\int_{\ball_{r}^2(z)}N_k|\nabla\Phi_k^j|^2\,d\mathcal{L}^2\le C\limsup_{k\to\infty}r^{-3}\int_{\ball_{r''}^2(z)}|\Phi_k^j-\Phi_\infty^j(z)||\nabla\Phi_k||\nabla\Phi_k^j|\,d\mathcal{L}^2. \]
		Hence, by Cauchy-Schwarz inequality and the bound $\nu_\infty(\ball_{6r}^2(z))=O(r^2)$, it suffices to show that
		\begin{equation}\label{eq:steponeclaim} \limsup_{k\to\infty}\int_{B_r^2(z)}|\Phi_k^j-\Phi_\infty^j(z)|^2|\nabla\Phi_k^j|^2\,d\mathcal{L}^2=o(r^4). \end{equation}
		By \eqref{eq:sliceapp} and the last part of Lemma \ref{c0convandhull}, we have
		\[ \lim_{k\to\infty}\big\|\Phi_k^j-\Phi_\infty^j(z)\big\|_{L^\infty(\ball_r^2(z))}=\big\|\Phi_\infty^j-\Phi_\infty^j(z)\big\|_{L^\infty(\ball_r^2(z))}=o(r), \]
		while $\limsup_{k\to\infty}\int_{\ball_r^2(z)}\big|\nabla\Phi_k^j\big|^2\,d\mathcal{L}^2\le 2\nu_\infty(\bar\ball_r^2(z))=O(r^2)$ and \eqref{eq:steponeclaim} follows.

		\emph{Step 2.} Call $\pi:\R^\envdim\to\R^2$ the projection onto the first two coordinates and set $p_0:=\pi(\Phi_\infty(z))$.
		Fix a parameter $\tau\in\Big(0,\mz\Big)$, which will tend to zero at the end of the proof. The error terms in the following computations will not be uniform in $\tau$ in general.

		For a fixed $r>0$, using Lemma \ref{slicingleb} we can select $r'\in(r-2\tau r,r-\tau r)$ and $r''\in (r+\tau r,r+2\tau r)$ such that
		\[ \Phi_\infty(z+A^{-1}r'y)=\Phi_\infty(z)+r'(y,0)+o(r),\quad \Phi_\infty(z+A^{-1}r''y)=\Phi_\infty(z)+r''(y,0)+o(r) \]
		for $y\in S^1$.
		We assume that $r$ is so small that the above error terms are less than $\tau r$.
		As soon as $k$ is big enough, by Lemma \ref{c0convandhull} we have
		\begin{equation}\label{eq:inclusions} \begin{split}&\pi\circ\Phi_k(\ball_{r'}^2(z))\subseteq B_r^2(p_0),\ \quad\qquad B_r^2(p_0)\subseteq\pi\circ\Phi_k(\ball_{r''}^2(z))\subseteq B_{r+3\tau r}^2(p_0), \\
		&\pi\circ\Phi_k(\de\ball_{r''}^2(z))\cap B_r^2(p_0)=\emptyset,\quad\Phi_k(\ball_{r''}^2(z))\subseteq B_{r+3\tau r}^\envdim(\Phi_\infty(z)). \end{split} \end{equation}
		In particular, by \eqref{eq:massbd},
		\begin{equation}\label{eq:upperest} \limsup_{k\to\infty}\mz\int_{\ball_{r''}^2(z)}N_k\abs{\nabla\Phi_k}^2\,d\mathcal{L}^2\le C''\pi(r+3\tau r)^2. \end{equation}
		Given any $h\in C^\infty_c(B_r^2(p_0))$ with $\norm{\nabla h}_{L^\infty}\le 1$, we can thus use the stationarity of $(B_2^2(0),\Phi_k,N_k)$ on the domain $\ball_{r''}^2(z)$ with the vector field $(h\circ\pi)e_1$. We obtain
		\[ \begin{split} &\int_{\ball_{r''}^2(z)}N_k\ang{\nabla(h(\pi\circ\Phi_k)),\nabla\Phi_k^1}\,d\mathcal{L}^2 \\
		&=\ell_k\int_{\ball_{r''}^2(z)}N_k h(\pi\circ\Phi_k)A^1(\Phi(x_k+r_k \cdot))(\nabla\Phi_k,\nabla\Phi_k)\,d\mathcal{L}^2. \end{split} \]
		The right-hand side is $o_k(1)\int_{\ball_{r''}^2(z)}N_k\abs{\nabla\Phi_k}^2\,d\mathcal{L}^2$, so
		\[ \lim_{k\to\infty}\int_{\ball_{r''}^2(z)}N_k\ang{\nabla(h(\pi\circ\Phi_k)),\nabla\Phi_k^1}\,d\mathcal{L}^2=0. \]
		On the other hand,
		\[ \begin{split} &\int_{\ball_{r''}^2(z)}N_k\ang{\nabla(h(\pi\circ\Phi_k)),\nabla\Phi_k^1}\,d\mathcal{L}^2
		=\int_{\ball_{r''}^2(z)}N_k\de_1 h(\pi\circ\Phi_k)(\abs{\de_1\Phi_k^1}^2+\abs{\de_2\Phi_k^1}^2)\,d\mathcal{L}^2 \\
		&+\int_{\ball_{r''}^2(z)}N_k\de_2 h(\pi\circ\Phi_k)(\de_1\Phi_k^2\de_1\Phi_k^1+\de_2\Phi_k^2\de_2\Phi_k^1)\,d\mathcal{L}^2. \end{split} \]
		Letting $J(\pi\circ\Phi_k):=\abs{\de_1\Phi_k^1\de_2\Phi_k^2-\de_2\Phi_k^1\de_1\Phi_k^2}$ and using Lemma \ref{cptconfineq}, we deduce
		\[ \int_{\ball_{r''}^2(z)}N_k\ang{\nabla(h\circ\Phi_k),\nabla\Phi_k^1}\,d\mathcal{L}^2
		=\int_{\ball_{r''}^2(z)}N_k\de_1 h(\pi\circ\Phi_k)J(\pi\circ\Phi_k)\,d\mathcal{L}^2+\delta_{r,k}, \]
		where $\limsup_{k\to\infty}|\delta_{r,k}|\le 2\tau C''\pi(r+3\tau r)^2+o(r^2)$, thanks to \eqref{eq:upperest} and \eqref{eq:negcpts}.
		Using the vector field $(h\circ\pi)e_2$ similarly, we arrive at
		\[ \limsup_{k\to\infty}\bigg|\int_{\ball_{r''}^2(z)}N_k\nabla h(\pi\circ\Phi_k)J(\pi\circ\Phi_k)\,d\mathcal{L}^2\bigg|\le 4\tau C''\pi(r+3\tau r)^2+o(r^2). \]
		Using the area formula (see Lemma \ref{rectif}), this translates into
		\[ \limsup_{k\to\infty}\bigg|\int_{B_r^2(p_0)}Q_k(w)\nabla h(w)\,dw\bigg|\le 4\tau C''\pi(r+3\tau r)^2+o(r^2), \]
		where $Q_k(w):=\sum_{y\in\ball_{r''}^2(z)\cap\good:\pi\circ\Phi_k(y)=w}N_k(y)$.
		By \eqref{eq:inclusions} and Lemma \ref{zerocontrib} we have $Q_k\in\N$ and $Q_k\ge 1$ a.e. By Lemma \ref{oneinfty} and $\limsup_{k\to\infty}\norm{Q_k}_{L^1(B_{3r/4}^2(p_0))}=O(r^2)$ we finally deduce that, for every $\epsilon>0$,
		\[ \limsup_{r\to 0}\limsup_{k\to\infty}r^{-2}\big\|Q_k-(Q_k)_{B_{r/2}^2(p_0)}\big\|_{L^{1,\infty}(B_{r/2}^2(p_0))}=O(\tau)+O(\epsilon), \]
		where $(Q_k)_{B_{r/2}^2(p_0)}:=\media_{B_{r/2}^2(p_0)}Q_k\,d\mathcal{L}^2$ and the implied constant in $O(\tau)$ depends on $\epsilon$.
		Since $\tau$ and $\epsilon$ are arbitrary, we get
		\[ \limsup_{r\to 0}\limsup_{k\to\infty}r^{-2}\norm{Q_k-(Q_k)_{B_r^2(p_0)}}_{L^{1,\infty}(B_r^2(p_0))}=0. \]
		In particular,
		\[ \limsup_{r\to 0}\limsup_{k\to\infty}\ \operatorname{dist}((Q_k)_{B_r^2(p_0)},\N\setminus\set{0})=0. \] But, again by \eqref{eq:inclusions},
		\[ \begin{split} \bigg|\nu_k(\ball_r^2(z))-\int_{B_r^2(p_0)}Q_k\,d\mathcal{L}^2\bigg|&\le \bigg|\mz\int_{\ball_r^2(z)}N_k\abs{\nabla\Phi_k}^2\,d\mathcal{L}^2-\int_{\ball_r^2(z)}N_k J(\pi\circ\Phi_k)\,d\mathcal{L}^2\bigg| \\
		&\fantasma{\le}+\int_{\ball_{r''}^2(z)\setminus\ball_{r'}^2(z)}N_kJ(\pi\circ\Phi_k)\,d\mathcal{L}^2, \end{split} \]
		so that Lemma \ref{cptconfineq} and \eqref{eq:negcpts} give
		\[ \begin{split} &\limsup_{r\to 0}\limsup_{k\to\infty}r^{-2}\bigg|\nu_k(\ball_r^2(z))-\int_{B_r^2(p)}Q_k\,d\mathcal{L}^2\bigg|\le\limsup_{r\to 0}\limsup_{k\to\infty}r^{-2}\nu_k(\ball_{r''}^2(z)\setminus\ball_{r'}^2(z)) \\
		&=\lim_{r\to 0}r^{-2}\int_{\ball_{r''}^2(z)\setminus \ball_{r'}^2(z)}m\,d\mathcal{L}^2=O(\tau). \end{split} \]
		Hence, $\limsup_{r\to 0}\limsup_{k\to\infty}\operatorname{dist}\Big(\frac{\nu_k(\ball_r^2(z))}{\pi r^2},\N\setminus\set{0}\Big)=0$ as well. As
		\[ N_\infty(z)=|\det(A)|^{-1}m(z)=\lim_{r\to 0}\frac{\nu_\infty(\ball_r^2(z))}{\pi r^2}=\lim_{r\to 0}\lim_{k\to\infty}\frac{\nu_k(\ball_r^2(z))}{\pi r^2}, \]
		we get $N_\infty(z)=|\det(A)|^{-1}m(z)\in\N\setminus\set{0}$ and,
		using \eqref{eq:upperest}, $N_\infty(z)\le C''(1+3\tau)^2$. Letting $\tau\to 0$, we finally deduce $N_\infty(z)\le C''$.
	\end{proof}

	\begin{proof}[Proof of Theorem \ref{blow-up}]
		We let $g_{ij}:=\de_i\Phi_\infty\cdot\de_j\Phi_\infty$ and we define the Beltrami coefficient
		\[ \mu:=\frac{g_{11}-g_{22}+2i g_{12}}{g_{11}+g_{22}+2\sqrt{g_{11}g_{22}-g_{12}^2}}\uno_{B_2^2(0)\cap\mathcal{G}'}\quad\text{on }\C. \]
		In particular, $\mu=0$ on $\C\setminus B_2^2(0)$. Moreover, a.e. on the set $\mathcal{G}'$ we have
		\[ \abs{\mu}^2\le\frac{(g_{11}-g_{22})^2+4g_{12}^2}{(g_{11}+g_{22})^2+4(g_{11}g_{22}-g_{12}^2)}
		=\frac{(g_{11}+g_{22})^2-4(g_{11}g_{22}-g_{12}^2)}{(g_{11}+g_{22})^2+4(g_{11}g_{22}-g_{12}^2)}\le\frac{N_\infty^2-1}{N_\infty^2+1}, \]
		since $(g_{11}+g_{22})^2=\abs{\nabla\Phi_\infty}^4$, which by Lemma \ref{ninfty} is bounded by $4N_\infty^2\abs{\de_1\Phi_\infty\wedge\de_2\Phi_\infty}^2
		=4N_\infty^2(g_{11}g_{22}-g_{12}^2)$. We know that $\norm{N_\infty}_{L^\infty}\le C''$, so by \cite[Theorem~4.24]{imayoshi} there exists a $(C'')^2$-quasiconformal homeomorphism $\varphi_\infty\in W^{1,2}_{loc}(\C,\C)$, with $\varphi_\infty(0)=0$, satisfying a.e.
		\begin{align}\label{eq:defvarphiinfty} \de_{\bar z}\varphi_\infty=\mu\,\de_z\varphi_\infty. \end{align}
		We recall that the inverse map is also a $(C'')^2$-quasiconformal homeomorphism in $W^{1,2}_{loc}(\C,\C)$ (see \cite[Theorem~4.10 and Proposition~4.2]{imayoshi}) and that both $\varphi_\infty$ and $\varphi_\infty^{-1}$ map negligible sets to negligible sets (see \cite[Lemma~4.12]{imayoshi}). Moreover, using the chain rule (which holds by \cite[Lemma~III.6.4]{lehto}), we get that $\varphi_\infty$ has invertible differential a.e. and
		\[ 0=\de_{\bar w}(\varphi_\infty\circ\varphi_\infty^{-1}(w))
		=(\de_z\varphi_\infty)\circ\varphi_\infty^{-1}\de_{\bar w}(\varphi_\infty^{-1})+(\de_{\bar z}\varphi_\infty)\circ\varphi_\infty^{-1}\de_{\bar w}\pa{\bar{\varphi_\infty^{-1}}}. \]
		Being also $\de_z\varphi_\infty\neq 0$ a.e. (by \eqref{eq:defvarphiinfty}), we deduce that
		\[ \de_w\pa{\bar{\varphi_\infty^{-1}}}=\bar{\de_{\bar w}(\varphi_\infty^{-1})}=-(\bar\mu\circ\varphi_\infty^{-1})\de_w(\varphi_\infty^{-1}) \]
		a.e. From now on, $\varphi_\infty$ will denote the homeomorphism restricted to $B_2^2(0)$. Let $\Omega_\infty:=\varphi_\infty(B_2^2(0))$. By the chain rule again, we have $\Phi_\infty\circ\varphi_\infty^{-1}\in W^{1,1}_{loc}(\Omega_\infty)$ and
		\[ \begin{split} &\de_w(\Phi_\infty\circ\varphi_\infty^{-1})\cdot\de_w(\Phi_\infty\circ\varphi_\infty^{-1}) \\
		&\hspace{-3pt}=((\de_z\Phi_\infty\cdot\de_z\Phi_\infty-2\bar\mu\,\de_z\Phi_\infty\cdot\de_{\bar z}\Phi_\infty+\bar\mu^2\de_{\bar z}\Phi_\infty\cdot\de_{\bar z}\Phi_\infty)\circ\varphi_\infty^{-1})(\de_w(\varphi_\infty^{-1}))^2 \end{split} \]
		vanishes a.e., since $\bar\mu=\frac{g_{11}+g_{22}-\sqrt{(g_{11}+g_{22})^2-((g_{11}-g_{22})^2+4g_{12}^2)}}{g_{11}-g_{22}+2i g_{12}}$ on the subset of $\mathcal{G}'$ where $\nabla\Phi_\infty$ is not conformal, while on the complement of this set $\bar\mu=0$. Hence, $\Phi_\infty\circ\varphi_\infty^{-1}$ is weakly conformal.
		By the area formula (see Lemma \ref{rectif}),
		\[ \begin{split} \int_{\Omega_\infty}\abs{\nabla(\Phi_\infty\circ\varphi_\infty^{-1})}^2\,d\mathcal{L}^2&=2\int_{\Omega_\infty}\abs{\de_1(\Phi_\infty\circ\varphi_\infty^{-1})\wedge\de_2(\Phi_\infty\circ\varphi_\infty^{-1})}\,d\mathcal{L}^2 \\
		&=2\int_{\Omega_\infty}\abs{\de_1\Phi_\infty\wedge\de_2\Phi_\infty}\circ\varphi_\infty^{-1}\abs{\de_1\varphi_\infty^{-1}\wedge\de_2\varphi_\infty^{-1}}\,d\mathcal{L}^2 \\
		&=2\int_{B_2^2(0)}\abs{\de_1\Phi_\infty\wedge\de_2\Phi_\infty}\,d\mathcal{L}^2\le\int_{B_2^2(0)}\abs{\nabla\Phi_\infty}^2\,d\mathcal{L}^2, \end{split} \]
		which shows that $\Phi_\infty\circ\varphi_\infty^{-1}\in W^{1,2}(\Omega_\infty,\R^\envdim)$. The same computation shows that
		\begin{equation}\label{eq:pushhomeo} \nu_\infty=N_\infty\abs{\de_1\Phi_\infty\wedge\de_2\Phi_\infty}\mathcal{L}^2=(\varphi_\infty^{-1})_*\pa{\mz N_\infty\circ\varphi_\infty^{-1}\abs{\nabla(\Phi_\infty\circ\varphi_\infty^{-1})}^2\mathcal{L}^2}. \end{equation}
		This, together with $(\Phi_\infty)_*\nu_\infty=\mu_\infty$ (by Lemma \ref{pushforward}) and assumption \eqref{eq:massbd}, shows that $(\Omega_\infty,\Phi_\infty\circ\varphi_\infty^{-1},N_\infty\circ\varphi_\infty^{-1})$ satisfies \eqref{eq:massass}.

		Finally, for any $\omega\cptsub\Omega_\infty$ with smooth boundary, we show that the varifold $\vfd$ associated to $(\omega,\Phi_\infty\circ\varphi_\infty^{-1},N_\infty\circ\varphi_\infty^{-1})$ is stationary in $\R^\envdim\setminus\Phi_\infty\circ\varphi_\infty^{-1}(\de\omega)$. Setting $\omega':=\varphi_\infty^{-1}(\omega)$, from the $C^0_{loc}$ convergence $\Phi_k\to\Phi_\infty$ we infer that the varifolds $\vfd_k:=\vfd_{(\omega',\Phi_k,N_k)}$ converge (a priori only after extracting a subsequence) to a stationary varifold $\tilde\vfd$ in $\R^\envdim\setminus\Phi_\infty(\de\omega')$.
		This varifold is rectifiable (see \cite[Theorem~42.4]{simon}). But, since $\mathcal{L}^2(\de\omega')=0$, Lemma \ref{pushforward} gives
		\[ \begin{split} \norm{\vfd_k}&= (\Phi_k)_*(\uno_{\omega'}\nu_k)\weakstarto(\Phi_\infty)_*(\uno_{\omega'}\nu_\infty) \\
		&=(\Phi_\infty\circ\varphi_\infty^{-1})_*\pa{\mz N_\infty\circ\varphi_\infty^{-1}\abs{\nabla(\Phi_\infty\circ\varphi_\infty^{-1})}^2\uno_\omega\mathcal{L}^2}=\norm{\vfd} \end{split} \]
		as Radon measures in $\R^\envdim$. Since $\norm{\vfd_k}\weakstarto\norm{\tilde\vfd}$ in $\R^\envdim\setminus\Phi_\infty(\de\omega')=\R^\envdim\setminus\Phi_\infty\circ\varphi_\infty^{-1}(\de\omega)$ and a rectifiable varifold is uniquely determined by the associated mass measure, we deduce that on this open set $\tilde\vfd=\vfd$. Since $\tilde\vfd$ is stationary, the theorem follows.
	\end{proof}

	Theorem \ref{blow-up} admits an analogous statement in which $B_2^2(0)$ is replaced by $\C$. Let $(x_k)\subseteq\Omega$ be a sequence of points, together with a sequence of radii $(r_k)$ such that $\lim_{k\to\infty}\frac{\dist(x_k,\de\Omega)}{r_k}=\infty$.
	Assuming $\ell_k^2:=\int_{B_{r_k}^2(x_k)}\abs{\nabla\Phi}^2\,d\mathcal{L}^2>0$ eventually, we let
	\[ \Phi_k:=\ell_k^{-1}(\Phi(x_k+r_k\cdot)-\Phi(x_k)),\quad N_k:=N(x_k+r_k\cdot) \]
	and notice that, for any $R>0$, the functions $\Phi_k,N_k$ are eventually defined on $B_R^2(0)$. Assume moreover that
	\begin{itemize}
		\item $\limsup_{k\to\infty}\frac{\int_{B_{Rr_k}^2(x_k)}\abs{\nabla\Phi}^2\,d\mathcal{L}^2}{\int_{B_{r_k}^2(x_k)}\abs{\nabla\Phi}^2\,d\mathcal{L}^2}<\infty$ for all $R>0$;
		\item $\limsup_{k\to\infty}(\Phi_k)_*\pa{\mz N_k\abs{\nabla\Phi_k}^2\uno_{B_R^2(0)}\mathcal{L}^2}(B_s^\envdim(p))\le\massbd\pi s^2$ for all $s>0$, all $p\in\R^\envdim$ and all $R>0$;
		\item $\ell_k\to 0$.
	\end{itemize}

	\begin{thm}[parametrized blow-up defined on $\bm{\C}$]\label{blow-up-c}
		Up to subsequences, there exist $\Phi_\infty\in W^{1,2}_{loc}(\C,\R^\envdim)$, $N_\infty\in L^\infty(\C,\N\setminus\set{0})$ and a
		quasiconformal homeomorphism $\varphi_\infty\in W^{1,2}_{loc}(\C,\C)$, with $\varphi_\infty(0)=0$, such that
		\[ \Phi_k\to\Phi_\infty\quad\text{in }C^0_{loc}(\C,\R^\envdim),\qquad\nabla\Phi_k\weakto\nabla\Phi_\infty\quad\text{in }L^2_{loc}(\C,\R^{\envdim\times 2}), \]
		\[ \mz N_k\abs{\nabla\Phi_k}^2\mathcal{L}^2\weakstarto N_\infty\abs{\de_1\Phi_\infty\wedge\de_2\Phi_\infty}\mathcal{L}^2\quad\text{as Radon measures}.
		\]
		Moreover, $\Phi_\infty\circ\varphi_\infty^{-1}$ is weakly conformal and $(\C,\Phi_\infty\circ\varphi_\infty^{-1},N_\infty\circ\varphi_\infty^{-1})$ is a local parametrized stationary varifold in $\R^\envdim$.
	\end{thm}

	\begin{proof}
		Let $\Phi_\infty\in W^{1,2}_{loc}(\C,\R^\envdim)$ be a local weak limit. Repeating the proof of Theorem \ref{blow-up} with $R=2^j$ in place of $2$ (for all $j\ge 1$) and using a diagonal argument, up to subsequences we get $\Phi_k\to\Phi_\infty$ in $C^0_{loc}(\C,\R^\envdim)$ and, assuming without loss of generality that
		\[ \nu_{\infty,j}:=\lim_{k\to\infty}\mz N_k\abs{\nabla\Phi_k}^2\uno_{B_{2^j}^2(0)}\mathcal{L}^2 \]
		exists, we also get that
		\[ \nu_{\infty,j}=N_\infty\abs{\de_1\Phi_\infty\wedge\de_2\Phi_\infty}\uno_{B_{2^j}^2(0)}\mathcal{L}^2 \]
		for some $N_\infty\in L^\infty(\C,\N\setminus\set{0})$ with $\norm{N_\infty}_{L^\infty}\le C''$, as well as
		\[ \abs{\nabla\Phi_\infty}^2\le 2N_\infty\abs{\de_1\Phi_\infty\wedge\de_2\Phi_\infty} \]
		a.e. Assuming also that $\lim_{k\to\infty}(\Phi_k)_*\pa{\mz N_k\abs{\nabla\Phi_k}^2\uno_{B_{2^j}^2(0)}\mathcal{L}^2}$ exists for all $j$, we can set
		\[ \nu_\infty:=\lim_{j\to\infty}\nu_{\infty,j},\quad\mu_\infty:=\lim_{j\to\infty}\lim_{k\to\infty}(\Phi_k)_*\pa{\mz N_k\abs{\nabla\Phi_k}^2\uno_{B_{2^j}^2(0)}\mathcal{L}^2} \]
		and, with the same proof as Lemma \ref{pushforward}, we have again $\mu_\infty=(\Phi_\infty)_*\nu_\infty$.
		The remainder of the proof is completely analogous to the one of Theorem \ref{blow-up}, using \cite[Theorem~4.30]{imayoshi} in order to build the quasiconformal homeomorphism $\varphi_\infty:\C\to\C$.
	\end{proof}

	\section{Regularity in the general case}\label{genregsec}

	This section is devoted to the proof of the main regularity result (see Theorem \ref{genreg} and Corollary \ref{globgenreg} below). We first show a removable singularity criterion. Its proof consists of a standard capacity argument in the target $\R^\envdim$ and could be well known to the expert community. We include it both for the reader's convenience and because it is the only place where the technical assumption \eqref{eq:massass} is used.

	\begin{lemmaen}[removable singularities]\label{removab}
		Let $(\Omega,\Phi,N)$ be a local parametrized stationary varifold in $\subman$ (possibly $\subman=\R^\envdim$). Assume that $\Phi$ satisfies
		\[ -\Delta\Phi=A(\Phi)(\nabla\Phi,\nabla\Phi) \]
		in the distributional sense on $\Omega\setminus S$, for some relatively closed $S\subseteq\Omega$ with $\mathcal{H}^1(\Phi(S))=0$. Then the equation is satisfied on the whole $\Omega$ and, as a consequence, $\Phi$ is $C^\infty$-smooth.
	\end{lemmaen}

	\begin{proof}
		Let $v\in C^\infty_c(\Omega)$.
		For any $\epsilon>0$ we can cover the compact set $K:=\Phi(S\cap\supp{v})$ by a finite union of balls $\bigcup_{i\in I}B_{r_i}^\envdim(p_i)$ with centers on this set and $\sum_i r_i<\epsilon$. Let $\rho_i\in C^\infty(\R^\envdim)$ be a function which equals $0$ on $B_{r_i}^\envdim(p_i)$, $1$ on $\R^\envdim\setminus B_{2r_i}^\envdim(p_i)$ and has $\norm{\nabla\rho_i}_{L^\infty}\le 2r_i^{-1}$.

		Since $v_\epsilon:=v\,\prod_i(\rho_i\circ\Phi)$ vanishes near $S\cap\supp{v}$, the function $v_\epsilon\in W^{1,2}\cap L^\infty(\Omega)$ is supported in a compact subset of $\Omega\setminus S$. Thus, using the hypothesis and a standard approximation argument,
		\begin{equation}\label{eq:approxpde} \int_\Omega\ang{\nabla\Phi,\nabla v_\epsilon}\,d\mathcal{L}^2=\int_\Omega A(\Phi)(\nabla\Phi,\nabla\Phi)v_\epsilon\,d\mathcal{L}^2. \end{equation}
		We claim that, as $\epsilon\to 0$, the left-hand side converges to $\int_{\Omega\setminus\Phi^{-1}(K)}\ang{\nabla\Phi,\nabla v}\,d\mathcal{L}^2$. Indeed, let us write
		\[ \nabla v_\epsilon=\bigg(\prod_i\rho_i\circ\Phi\bigg)\nabla v+v\sum_i\bigg(\prod_{j\neq i}\rho_j\circ\Phi\bigg)\nabla(\rho_i\circ\Phi). \]
		The first term converges to $\uno_{\Omega\setminus\Phi^{-1}(K)}\nabla v$ in $L^2(\Omega,\R^2)$. On the other hand, by \eqref{eq:massass},
		\[ \begin{split} &\Big|\int_\Omega\Big\langle\nabla\Phi,v\sum_i\Big(\prod_{j\neq i}\rho_j\circ\Phi\Big)\nabla(\rho_i\circ\Phi)\Big\rangle\,d\mathcal{L}^2\Big| \\
		&\le 2\norm{v}_{L^\infty}\sum_i r_i^{-1}\int_{\Phi^{-1}(B_{2r_i}^\envdim(p_i))}\abs{\nabla\Phi}^2\,d\mathcal{L}^2 \\
		&\le 4\norm{v}_{L^\infty}\sum_i r_i^{-1}\norm{\vfd_\Omega}(B_{2r_i}^\envdim(p_i))
		=4\norm{v}_{L^\infty}\sum_i O(r_i)\to 0 \end{split} \]
		as $\epsilon\to 0$. This establishes the claim. Moreover, the right-hand side of \eqref{eq:approxpde} converges to $\int_{\Omega\setminus\Phi^{-1}(K)}A(\Phi)(\nabla\Phi,\nabla\Phi)v\,d\mathcal{L}^2$. Finally, in order to establish \eqref{eq:approxpde} with $v$ in place of $v_\epsilon$, we observe that $\nabla\Phi=0$ a.e. on $\Phi^{-1}(K)$: indeed, by the area formula (see Lemma \ref{rectif}),
		\[ \int_{\Phi^{-1}(K)}\abs{\nabla\Phi}^2\,d\mathcal{L}^2=2\int_K\bigg(\sum_{y\in\Phi^{-1}(p)\cap\good}1\bigg)\,d\mathcal{H}^2(p)=0, \]
		since $\mathcal{H}^2(K)=0$. The smoothness of $\Phi$ follows from the continuity of $\Phi$ and \cite[Section~3.4]{moser}.
	\end{proof}

	\begin{definition}[admissible points]
		Given a local parametrized stationary varifold $(\Omega,\Phi,N)$, a point $x\in\Omega$ is said to be \emph{admissible} if $\Phi$ is nonconstant in any neighborhood of $x$ and
		\[ \liminf_{r\to 0}\frac{\int_{B_{2r}^2(x)}\abs{\nabla\Phi}^2\,d\mathcal{L}^2}{\int_{B_r^2(x)}\abs{\nabla\Phi}^2\,d\mathcal{L}^2}<\infty. \]
		We call $\mathcal{A}\subseteq\Omega$ the Borel set of the admissible points. \hfill\qedsymbol
	\end{definition}

	We now show that the image of the set of non-admissible points is very small.

	\begin{lemmaen}[negligibility of non-admissible points in the target]\label{noblowdimzero}
		The set $\Phi(\Omega\setminus\mathcal{A})$ has Hausdorff dimension $0$, i.e.
		\[ \mathcal{H}^s(\Phi(\Omega\setminus\mathcal{A}))=0 \]
		for any real $s>0$.
	\end{lemmaen}

	\begin{proof}
		We can assume that $\Omega$ is bounded. Fix $s>0$ and two arbitrary parameters $\delta,\epsilon>0$ and choose a real number $\alpha>4s^{-1}$. For any $x\in\Omega\setminus\mathcal{A}$ we have
		\[ \int_{B_{2^{-k-1}}^2(x)}\abs{\nabla\Phi}^2\,d\mathcal{L}^2\le 2^{-\alpha}\int_{B_{2^{-k}}^2(x)}\abs{\nabla\Phi}^2\,d\mathcal{L}^2 \]
		for all $k\ge k_0$ (for some threshold $k_0\ge 0$ depending on $x$), so
		\[ \int_{B_{2^{-k}}^2(x)}\abs{\nabla\Phi}^2\,d\mathcal{L}^2\le (2^{-k})^{\alpha}\Big(2^{k_0\alpha}\int_{B_{2^{-k_0}}^2(x)}\abs{\nabla\Phi}^2\,d\mathcal{L}^2\Big)
		=o((2^{-k})^{4/s}). \]
		Hence we can find, for all $x\in\Omega\setminus\mathcal{A}$, a radius $r_x<\mz\dist(x,\de\Omega)$ such that
		\[ \int_{B_{2r_x}^2(x)}\abs{\nabla\Phi}^2\,d\mathcal{L}^2\le\epsilon r_x^{4/s},\quad
		\diam\Phi(B_{r_x}^2(x))<\delta \]
		(by continuity of $\Phi$).
		Finally, Besicovitch covering theorem gives us a countable subcollection
		of balls $(B_{r_i}^2(x_i))$ such that $\Omega\setminus\mathcal{A}\subseteq\bigcup_i B_{r_i}^2(x_i)$ and $\sum_i\uno_{B_{r_i}^2(x_i)}$ is bounded everywhere by a universal constant. Thus, by inequality \eqref{eq:xpdiamest} (or \eqref{eq:xpdiamest2} if $\subman=\R^\envdim$),
		\[ \begin{split} \sum_i\pa{\diam\Phi(B_{r_i}^2(x_i))}^s&\le C\sum_i\Big(\int_{B_{2r_i}^2(x_i)}\abs{\nabla\Phi}^2\,d\mathcal{L}^2\Big)^{s/2}\le C\epsilon^{s/2}\sum_i r_i^2 \\
		&\le C\epsilon^{s/2}\sum_i\mathcal{L}^2(B_{r_i}^2(x_i))\le C\epsilon^{s/2}\mathcal{L}^2(\Omega). \end{split} \]
		Since $\delta$ and $\epsilon$ were arbitrary, we deduce $\mathcal{H}^s(\Phi(\Omega\setminus\mathcal{A}))=0$.
	\end{proof}

	Let $(\Omega,\Phi,N)$ be a local parametrized stationary varifold. Assume moreover that $\Omega$ is bounded, $\Phi$ extends continuously to $\bar\Omega$ with $\Phi(\Omega)\cap\Phi(\de\Omega)=\emptyset$ and $\vfd_\Omega$ is stationary in $\subman\setminus\Phi(\de\Omega)$. Recall that in this situation the upper semicontinuous function $\tilde N$ is finite (see Remark \ref{alsoloc}). We also assume that $\sup_\Omega\tilde N$ is finite.
	For any $x\in\Omega$ and any $r\le\dist(x,\de\Omega)$ we define $\ell(x,r):=\pa{\int_{B_r^2(x)}\abs{\nabla\Phi}^2\,d\mathcal{L}^2}^{1/2}$.

	We now use a compactness argument, together with Theorem \ref{blow-up} and Theorem \ref{conicalreg}, in order to prove that, under some technical assumptions, the Dirichlet energy does not decay too fast (in a uniform, quantitative way). The main underlying idea is that this happens for a holomorphic function at a zero whose order is controlled, but we have to take care of the possible distortion caused by the quasiconformal homeomorphism appearing in the blow-up.

	\begin{lemmaen}[propagation of slow energy decay at a smaller scale]\label{softlemma}
		For every $C'>0$ there exists $\epsilon=\epsilon(\Omega,\Phi,N,C')<\mz$ with the following property: whenever
		\begin{itemize}
			\item $x\in\omega\cptsub\Omega$ and $0<r<\mz\dist(x,\de\omega)$,
			\item $\ell(x,r)<\epsilon\dist(\Phi(x),\Phi(\de\omega))$,
			\item $0<\int_{B_{2r}^2(x)}\abs{\nabla\Phi}^2\,d\mathcal{L}^2< C'\int_{B_r^2(x)}\abs{\nabla\Phi}^2\,d\mathcal{L}^2$,
			\item $\frac{\norm{\vfd_\omega}(B_s^\envdim(\Phi(x)))}{\pi s^2}\in(\tilde N(x)-\epsilon,\tilde N(x)+\epsilon)$ for all $0<s<\epsilon^{-1}\ell(x,r)$,
		\end{itemize}
		there exists $r'\in\pa{\epsilon r,\frac{r}{2}}$ such that
		$\int_{B_{2r'}^2(x)}\abs{\nabla\Phi}^2\,d\mathcal{L}^2<\bar C\int_{B_{r'}^2(x)}\abs{\nabla\Phi}^2\,d\mathcal{L}^2$,
		for some $\bar C$ depending only on $\sup_\Omega\tilde N$.
	\end{lemmaen}

	\begin{proof}
		Assume by contradiction that the statement is false for all $\epsilon=2^{-k}$. We can then find a sequence of points $x_k\in\omega_k$ and radii $r_k<\mz\dist(x_k,\de\omega_k)$ such that \eqref{eq:enbd} and \eqref{eq:shrinking} are satisfied, as well as
		$\tilde N(x_k)\to\bar N\in[1,\sup_\Omega\tilde N]$ (up to subsequences),
		\begin{equation}\label{eq:denspinch} \frac{\norm{\vfd_{\omega_k}}(B_s^\envdim(\Phi(x_k)))}{\pi s^2}\in(\tilde N(x_k)-2^{-k},\tilde N(x_k)+2^{-k})\quad\text{for }0<s<2^k\ell(x_k,r_k), \end{equation}
		\begin{equation}\label{eq:badenergy} \int_{B_{2r'}^2(x_k)}\abs{\nabla\Phi}^2\,d\mathcal{L}^2\ge\bar C\int_{B_{r'}^2(x_k)}\abs{\nabla\Phi}^2\,d\mathcal{L}^2\quad\text{for }2^{-k} r_k<r'<\frac{r_k}{2} \end{equation}
		($\bar C$ will be chosen at the end of the proof).
		Moreover, using the notation introduced in Section \ref{blowsec}, the varifolds $\vfd_k:=(\ell_k^{-1}(\cdot-\Phi(x_k)))_*\vfd_{\omega_k}$ have generalized mean curvature bounded by $O(\ell_k)$ (in $L^\infty$) in
		\[ \R^\envdim\setminus\ell_k^{-1}(\Phi(\de\omega_k)-\Phi(x_k))\supseteq B_{2^k}^\envdim(0). \]
		Hence, up to subsequences, the varifolds $\vfd_k$ converge to a stationary varifold $\vfd_\infty$ in $\R^\envdim$. Moreover, by \eqref{eq:denspinch}, we have
		\[ \frac{\norm{\vfd_\infty}(B_s^\envdim(0))}{\pi s^2}=\frac{\norm{\vfd_\infty}(\bar B_s^\envdim(0))}{\pi s^2}=\bar N. \]
		The varifold $\vfd_\infty$ is rectifiable and conical, with density in $[1,\bar N]$ on $\supp{\vfd_\infty}$ (see e.g. the proofs of \cite[Corollary~42.6~and~Theorem~19.3]{simon}). In particular, since $\mu_k\le\norm{\vfd_k}$, \eqref{eq:massbd} follows, with $C'':=\bar N$. So, up to subsequences, the conclusions of Theorem \ref{blow-up} hold. We remark that $\Phi_\infty$ satisfies
		\[ \int_{B_1^2(0)}N_\infty\abs{\de_1\Phi_\infty\wedge\de_2\Phi_\infty}\,d\mathcal{L}^2=\nu_\infty(B_1^2(0))=\lim_{k\to\infty}\nu_k(B_1^2(0))=\mz \]
		and in particular it is nonconstant. The varifold $\vfd_\infty$ has also integer multiplicity by \cite[Theorem~6.4]{allard}, thus it can be expressed as a cone (with vertex $0$) over some stationary integer 1-rectifiable varifold $\mathbf{w}$ in $S^{\envdim-1}$ with density in $[1,\bar N]$ $\norm{\mathbf{w}}$-a.e.

		By the structure theorem in \cite[Section~3]{allalm}, $\mathbf{w}$ is supported in a finite union of geodesic curves. Hence, $\vfd_\infty$ is supported in a finite union of planes through the origin. Letting $\Psi:=\Phi_\infty\circ\varphi_\infty^{-1}$, by \eqref{eq:pushhomeo} and Lemma \ref{pushforward} we have
		\[ \big\|\vfd_{(\Omega_\infty,\Psi,N_\infty\circ\varphi_\infty^{-1})}\big\|=(\Phi_\infty)_*\nu_\infty=\mu_\infty\le\norm{\vfd_\infty}. \]
		So, using Proposition \ref{contandsupp}, we deduce that $\Psi(\Omega_\infty)\subseteq\supp{\norm{\vfd_\infty}}$. Hence, Theorem \ref{conicalreg} applies: we obtain that $\Psi$ takes values in a plane and is a holomorphic function (once this plane is suitably identified with $\C$).
		Furthermore, by Lemma \ref{ninfty}, we can assume $\norm{N_\infty}_\infty\le C''=\bar N$. Now, by Lemma \ref{nisusc} and Lemma \ref{ninfty},
		\[ \mz\tilde N(x_k+r_k\cdot)\abs{\nabla\Phi_k}^2\uno_{B_2^2(0)}\mathcal{L}^2=\nu_k\weakstarto N_\infty\abs{\de_1\Phi_\infty\wedge\de_2\Phi_\infty}\mathcal{L}^2, \]
		so \eqref{eq:badenergy} gives
		\[ \int_{B_{2r}^2(0)}\abs{\de_1\Phi_\infty\wedge\de_2\Phi_\infty}\,d\mathcal{L}^2\ge\frac{\bar C}{(\sup_\Omega\tilde N)^2}\int_{B_r^2(0)}\abs{\de_1\Phi_\infty\wedge\de_2\Phi_\infty}\,d\mathcal{L}^2\quad\text{for }r<\mz. \]
		Let $\bar k\ge 1$ be the biggest integer such that $\abs{\Psi(w)}=O\pa{\abs{w}^{\bar k}}$. We claim that
		\[ \bar k\le\bar N. \]
		Indeed, since $\Psi$ is nonconstant and holomorphic, we have $K_w=\set{w}$ for all $w\in\Omega_\infty$ and the function $\tilde N$ for the local parametrized stationary varifold $(\Omega_\infty,\Psi,N_\infty\circ\varphi_\infty^{-1})$ is everywhere finite (see Remark \ref{alsoloc}). We call it $N'$ in order not to confuse it with the same function for $(\Omega,\Phi,N)$.
		Since the density of the varifold $\vfd_{(\Omega_\infty,\Psi,N_\infty\circ\varphi_\infty^{-1})}$ is bounded everywhere by $\bar N$ (being this true for $\vfd_\infty$), the same argument used to prove Lemma \ref{fincpt} gives
		\[ \bar k\le\sum_{j=1}^{\bar k}N'(w_j)\le\bar N. \]
		whenever $w_1,\dots,w_{\bar k}\in\Omega_\infty$ are distinct points with the same image.
		Such points exist because $\Psi$ is a $\bar k$-to-$1$ map near $0$. This establishes our claim and we deduce that
		\[ \lim_{r\to 0}\frac{\int_{B_{2r}^2(0)}\abs{\nabla\Psi}^2\,d\mathcal{L}^2}{\int_{B_r^2(0)}\abs{\nabla\Psi}^2\,d\mathcal{L}^2}=2^{2\bar k}\le 2^{2\bar N}. \]
		As was shown in the proof of Theorem \ref{blow-up}, $\varphi_\infty$ is an $\bar N^2$-quasiconformal homeomorphism. Since $\bar N\le\sup_\Omega\tilde N$, we claim that there exists a constant $\mathcal{K}\in\N$ depending only on $\sup_\Omega\tilde N$ such that, for $r>0$ small enough, there exists $s=s(r)>0$ with
		\begin{equation}\label{eq:circdist} B_s^2(0)\subseteq\varphi_\infty(B_r^2(0))\subseteq\varphi_\infty(B_{2r}^2(0))\subseteq B_{2^{\mathcal{K}}s}^2(0). \end{equation}
		Let $s:=\min_{z\in\de B_r^2(0)}\abs{\varphi_\infty(z)}$, $s':=\max_{z\in\de B_{2r}^2(0)}\abs{\varphi_\infty(z)}$ and call $z_1$ and $z_2$ two points where the minimum and the maximum are attained, respectively. If $r$ is small enough we have $\bar B_{s'}^2(0)\subseteq\varphi_\infty(B_2^2(0))$. Letting $A:=B_{s'}^2(0)\setminus\bar B_{s}^2(0)$ and denoting by $M(\cdot)$ the module of a ring domain (see \cite[Section~I.6.1]{lehto} for the definition), by \cite[Theorem~I.7.1]{lehto} we have
		\begin{equation}\label{eq:moddistortion} \log\pa{\frac{s'}{s}}=M(A)\le\bar N^2 M(\varphi_\infty^{-1}(A)). \end{equation}
		The ring domain $\varphi_\infty^{-1}(A)$ separates $0$ and $z_1$ from $z_2$ and $\infty$ (in $\widehat\C$), so Teichm\"uller's module theorem (see \cite[Section~II.1.3]{lehto}), together with $\abs{z_2}=2\abs{z_1}$, implies that $M(\varphi_\infty^{-1}(A))$ is bounded by a constant depending only on $\sup_\Omega\tilde N$. Since $\varphi_\infty(B_r^2(0))\supseteq B_{s}^2(0)$ and $\varphi_\infty(B_{2r}^2(0))\subseteq B_{s'}^2(0)$, \eqref{eq:circdist} follows from \eqref{eq:moddistortion}.
		Thus, by the area formula,
		\[ \frac{\bar C}{(\sup_\Omega\tilde N)^2}\le\frac{\int_{B_{2r}^2(0)}\abs{\de_1\Phi_\infty\wedge\de_2\Phi_\infty}\,d\mathcal{L}^2}{\int_{B_r^2(0)}\abs{\de_1\Phi_\infty\wedge\de_2\Phi_\infty}\,d\mathcal{L}^2}\le\frac{\int_{B_{2^{\mathcal K}s(r)}^2(0)}\abs{\nabla\Psi}^2\,d\mathcal{L}^2}{\int_{B_{s(r)}^2(0)}\abs{\nabla\Psi}^2\,d\mathcal{L}^2}\to (2^{2\bar k})^{\mathcal K}\le 2^{2\mathcal{K}\bar N} \]
		as $r\to 0$. We deduce $\bar C\le(\sup_\Omega\tilde N)^2 2^{2\mathcal K\sup_\Omega\tilde N}$ and this is a contradiction once we choose $\bar C$ so large that this inequality fails.
	\end{proof}

	We need another technical result, which is again obtained by means of a compactness argument.

	\begin{lemmaen}[absence of energy jumps]\label{soft2}
		For every $\delta\in(0,1)$ there exists $\epsilon'=\epsilon'(\Omega,\Phi,N,\delta)$ with the following property: whenever
		\begin{itemize}
			\item $x\in\omega\cptsub\Omega$ and $0<r<\mz\dist(x,\de\omega)$,
			\item $\ell(x,r)<\epsilon'\dist(\Phi(x),\Phi(\de\omega))$,
			\item $0<\int_{B_{2r}^2(x)}\abs{\nabla\Phi}^2\,d\mathcal{L}^2<\bar C\int_{B_r^2(x)}\abs{\nabla\Phi}^2\,d\mathcal{L}^2$, with $\bar C$ given by Lemma \ref{softlemma},
			\item $\frac{\norm{\vfd_\omega}(B_s^\envdim(\Phi(x)))}{\pi s^2}\in(\tilde N(x)-\epsilon',\tilde N(x)+\epsilon')$ for all $0<s<(\epsilon')^{-1}\ell(x,r)$,
		\end{itemize}
		we have $\int_{B_{\delta r}^2(0)}\abs{\nabla\Phi}^2\,d\mathcal{L}^2\ge\epsilon'\int_{B_r^2(0)}\abs{\nabla\Phi}^2\,d\mathcal{L}^2$.
	\end{lemmaen}

	\begin{proof}
		Arguing by contradiction as in the proof of Lemma \ref{softlemma}, we would get a local parametrized stationary varifold $(\Omega_\infty,\Phi_\infty\circ\varphi_\infty^{-1},N_\infty\circ\varphi_\infty^{-1})$ with
		\[ \int_{B_\delta^2(0)}\abs{\nabla\Phi_\infty}^2\,d\mathcal{L}^2=0. \]
		But then $\Psi:=\Phi_\infty\circ\varphi_\infty^{-1}$ could be identified with a nonconstant holomorphic function, on the connected domain $\Omega_\infty=\varphi_\infty(B_2^2(0))$, with uncountably many zeroes. This is a contradiction.
	\end{proof}

	\begin{corollary}[slow energy decay at admissible points]\label{enwellbeh}
		If $x$ lies in the admissible set $\adm$, then there exists an arbitrarily small $r>0$ such that $\int_{B_{2r}^2(x)}\abs{\nabla\Phi}^2\,d\mathcal{L}^2<\bar C\int_{B_r^2(0)}\abs{\nabla\Phi}^2\,d\mathcal{L}^2$. Moreover,
		\[ \limsup_{r\to 0}\frac{\int_{B_{2r}^2(0)}\abs{\nabla\Phi}^2\,d\mathcal{L}^2}{\int_{B_r^2(0)}\abs{\nabla\Phi}^2\,d\mathcal{L}^2}<\infty. \]
	\end{corollary}

	\begin{proof}
		Since $x\in\adm$, we have $\liminf_{r\to 0}\frac{\int_{B_{2r}^2(0)}\abs{\nabla\Phi}^2\,d\mathcal{L}^2}{\int_{B_r^2(0)}\abs{\nabla\Phi}^2\,d\mathcal{L}^2}<C'$ for some finite $C'$. Let $K_x\subseteq\omega\cptsub\Omega$ with $\bar\omega$ disjoint from $\Phi^{-1}(\Phi(x))\setminus K_x$ and choose a radius $r$ such that
		\[ \int_{B_{2r}^2(0)}\abs{\nabla\Phi}^2\,d\mathcal{L}^2< C'\int_{B_r^2(0)}\abs{\nabla\Phi}^2\,d\mathcal{L}^2 \]
		and $r$ so small that it satisfies the other hypotheses of Lemma \ref{softlemma} with both $\epsilon=\epsilon(\Omega,\Phi,N,C')$ and $\epsilon=\epsilon(\Omega,\Phi,N,\bar C)$ (the density assumptions for $\norm{\vfd_\omega}$ are eventually satisfied by definition of $\tilde N$). Then, by Lemma \ref{softlemma}, we can find
		\[  \epsilon(\Omega,\Phi,N,C')r<r_1<\frac{r}{2} \]
		such that $\int_{B_{2r_1}^2(0)}\abs{\nabla\Phi}^2\,d\mathcal{L}^2<\bar C\int_{B_{r_1}^2(0)}\abs{\nabla\Phi}^2\,d\mathcal{L}^2$.
		This new radius $r_1$ satisfies the hypotheses of Lemma \ref{softlemma} with $\epsilon=\epsilon(\Omega,\Phi,N,\bar C)$, so there exists
		\[ \epsilon(\Omega,\Phi,N,\bar C)r_1<r_2<\frac{r_1}{2} \]
		such that $\int_{B_{2r_2}^2(0)}\abs{\nabla\Phi}^2\,d\mathcal{L}^2<\bar C\int_{B_{r_2}^2(0)}\abs{\nabla\Phi}^2\,d\mathcal{L}^2$. Again, $r_2$ satisfies the hypotheses of Lemma \ref{softlemma} with $\epsilon=\epsilon(\Omega,\Phi,N,\bar C)$, so we can find $\epsilon(\Omega,\Phi,N,\bar C)r_2<r_3<\frac{r_2}{2}$ and so on. Eventually $r_k$ satisfies the hypotheses of Lemma \ref{soft2} with $\epsilon'=\epsilon'(\Omega,\Phi,N,\epsilon(\Omega,\Phi,N,\bar C))$. Thus,
		\[ \int_{B_{r_{k+1}}^2(0)}\abs{\nabla\Phi}^2\,d\mathcal{L}^2\ge\epsilon'\int_{B_{r_k}^2(0)}\abs{\nabla\Phi}^2\,d\mathcal{L}^2. \]
		Any radius $s>0$ small enough lies in some interval $[r_{k+1},r_k]$ and $2s\le r_{k-1}$. The result follows.
	\end{proof}

	We are finally ready to show the full regularity result for parametrized stationary varifolds.

	\begin{thm}[regularity in the general case]\label{genreg}
		Let $(\Omega,\Phi,N)$ be a local parametrized stationary varifold in $\subman$. Then $\Phi$ solves $-\Delta\Phi=A(\Phi)(\nabla\Phi,\nabla\Phi)$ and, on each connected component where $\Phi$ is nonconstant, $\Phi$ is a $C^\infty$-smooth branched immersion and $N$ is a.e. constant.
	\end{thm}

	For the definition of branched immersion, see e.g. \cite[Definitions~1.2~and~1.6]{gulliver}.

	\begin{proof}
		Assume that $\Omega'\cptsub\Omega$ satisfies $\Phi(\Omega')\cap\Phi(\de\Omega')=\emptyset$ and $\sup_{\Omega'}\tilde N<\infty$.
		We show that the partial differential equation holds in $\Omega'$ (the full result will be obtained at the end of the proof). More precisely, letting
		\[ \gamma:=\frac{1}{3}\epsilon\pa{\Omega',\restr{\Phi}{\Omega'},\restr{N}{\Omega'},\bar C}<1 \]
		(where $\bar C$ is the constant given by Lemma \ref{softlemma}, depending only on $\sup_{\Omega'}\tilde N$), we will show that the equation holds on the open set $\Omega_k:=\Omega'\cap\{\gamma^{-1}\tilde N<k+1\}$, by induction on $k$. The base case $k=0$ is trivial, since $\{\tilde N<\gamma\}=\emptyset$.
		Assume that the equation holds for $k-1$.

		We call $\mathcal{C}_k$ the set of accumulation points of $\Omega_k\setminus\Omega_{k-1}$ in $\Omega_k$, i.e. its derived set in $\Omega_k$. Notice that $\mathcal{C}_k\subseteq\Omega_k\setminus\Omega_{k-1}$ is closed in $\Omega_k$. We also set $\mathcal A_k:=\mathcal C_k\cap\adm$ and $\mathcal B_k:=\mathcal C_k\setminus\mathcal A_k$.
		Notice that, by Lemma \ref{removab}, the equation holds in the open set $\Omega_k\setminus\mathcal C_k$, since here the points with $k\gamma\le\tilde N(\cdot)<(k+1)\gamma$ form a discrete set.

		\emph{Step 1.} We first show that $\mathcal A_k$ is relatively open in $\mathcal C_k$, so that $\mathcal B_k$ is closed in $\Omega_k$. Let $x_0\in\mathcal A_k$. First of all, by Remark \ref{hightilden}, we can find $K_{x_0}\subseteq\omega\cptsub\Omega'$ with $\bar\omega$ disjoint from $\Phi^{-1}(\Phi(x_0))\setminus K_{x_0}$, $\Phi(\omega)\cap\Phi(\de\omega)=\emptyset$ and $\omega\cap\Phi^{-1}(\Phi(y))=K_y$ whenever $y\in\omega$ has $k\gamma\le\tilde N(y)<(k+1)\gamma$ (thanks to the fact that $\gamma<1$).

		Let $\epsilon:=\epsilon\pa{\Omega',\restr{\Phi}{\Omega'},\restr{N}{\Omega'},\bar C}$ and assume that $x_j\to x_0$, with $x_j\in\omega\cap\mathcal C_k$. In particular, by definition of $\tilde N$, the density of $\vfd_\omega$ at $\Phi(x_j)$ coincides with $\tilde N(x_j)$.
		Using Corollary \ref{enwellbeh}, we choose a radius $0<r<\mz\dist(x_0,\de\omega)$ with
		\[ \int_{B_{2r}^2(x_0)}\abs{\nabla\Phi}^2\,d\mathcal{L}^2<\bar C\int_{B_r^2(x_0)}\abs{\nabla\Phi}^2\,d\mathcal{L}^2,\quad\ell(x_0,r)<\epsilon\dist(\Phi(x_0),\Phi(\de\omega)), \]
		\[ \frac{\norm{\vfd_\omega}(\bar B_{\epsilon^{-1}\ell(x_0,r)}^\envdim(\Phi(x_0)))}{\pi(\epsilon^{-1}\ell(x_0,r))^2}<\tilde N(x_0)+\gamma. \]
		Eventually all the assumptions of Lemma \ref{softlemma} are satisfied by $x_j\in\omega$ (with $\Omega'$ and $C':=\bar C$), provided $\ell(x_0,r)$ is small enough: eventually we have $B_{\epsilon^{-1}\ell(x_j,r)}^\envdim(\Phi(x_j))\cap\Phi(\de\omega)=\emptyset$, so by the monotonicity formula we get, for $0<s<\epsilon^{-1}\ell(x_j,r)$,
		\[ \begin{split} &\tilde N(x_j)\le e^{(\sqrt{2}\norm{A}_\infty)s}\frac{\norm{\vfd_\omega}(B_s^\envdim(\Phi(x_j)))}{\pi s^2}\le e^{(\sqrt{2}\norm{A}_\infty)\epsilon^{-1}\ell(x_j,r)}\frac{\norm{\vfd_\omega}(B_{\epsilon^{-1}\ell(x_j,r)}^\envdim(\Phi(x_j)))}{\pi(\epsilon^{-1}\ell(x_j,r))^2} \\
		&\le e^{(\sqrt{2}\norm{A}_\infty)\epsilon^{-1}\ell(x_j,r)}(\tilde N(x_0)+\gamma)\le(e^{(\sqrt{2}\norm{A}_\infty)\epsilon^{-1}\ell(x_0,r)}+o_j(1))\frac{\tilde N(x_0)+\gamma}{\tilde N(x_0)+2\gamma}(\tilde N(x_j)+\epsilon) \end{split} \]
		eventually and, since $\frac{\tilde N(x_j)}{\tilde N(x_j)-\epsilon}\ge\frac{\tilde N(x_0)-\gamma}{\tilde N(x_0)-2\gamma}$, it suffices to impose additionally that
		\[ e^{(\sqrt{2}\norm{A}_\infty)\epsilon^{-1}\ell(x_0,r)}<\min\set{\frac{\tilde N(x_0)-\gamma}{\tilde N(x_0)-2\gamma},\frac{\tilde N(x_0)+2\gamma}{\tilde N(x_0)+\gamma}}. \]
		By Lemma \ref{softlemma} applied to the parametrized varifold $\pa{\Omega',\restr{\Phi}{\Omega'},\restr{N}{\Omega'}}$ and the point $x_j\in\omega$, for $j$ big enough there exists $r'<\frac{r}{2}$ (depending on $j$) such that
		$\int_{B_{2r'}^2(x_j)}\abs{\nabla\Phi}^2\,d\mathcal{L}^2<\bar C\int_{B_{r'}^2(x_j)}\abs{\nabla\Phi}^2\,d\mathcal{L}^2$. Since $r'$ satisfies again all the hypotheses of Lemma \ref{softlemma}, we can iterate and deduce that $x_j\in\adm$. Hence, $x_j\in\adm_k$ eventually.

		\emph{Step 2.} We now claim that $\tilde N(x)$ is an integer for any admissible point $x\in\Omega'\cap\adm$. Indeed, as in the proof of Lemma \ref{softlemma}, we can apply Theorem \ref{blow-up} with $x_k:=x$ and a suitable sequence of radii $r_k\to 0$: whenever $K_x\subseteq\tilde\omega\cptsub \Omega'$ has its closure disjoint from $\Phi^{-1}(\Phi(x))\setminus K_x$, the varifolds $(\ell_k^{-1}(\cdot-\Phi(x_k)))_*\vfd_{\tilde\omega}$ converge to a stationary cone $\vfd_\infty$ having density at most $\tilde N(x)$ at $0$, so we have $\norm{\vfd_\infty}(B_s^\envdim(p))\le\tilde N(x)\pi s^2$ (see the proof of \cite[Theorem~42.4]{simon}) and thus \eqref{eq:massbd} holds with $C'':=\tilde N(x)$.
		
		We obtain a local parametrized stationary varifold $(\Omega_\infty,\Phi_\infty\circ\varphi_\infty^{-1},N_\infty\circ\varphi_\infty^{-1})$ with $\Psi:=\Phi_\infty\circ\varphi_\infty^{-1}$ nonconstant and holomorphic (again by Theorem \ref{conicalreg}, since the mass measure of this parametrized varifold is bounded by the mass measure of $\vfd_\infty$, which is an integer rectifiable stationary cone).

		Let $\omega'\cptsub B_2^2(0)$ be a smooth neighborhood of $0$ with $0\nin\Phi_\infty(\de\omega')$ and $\Phi_\infty^{-1}(0)\cap\omega'=\set{0}$.
		Using the notation of Section \ref{blowsec}, from the locally uniform convergence $\Phi_k\to\Phi_\infty$ we infer that eventually $\Phi(x)\nin\Phi(x+r_k\de\omega')$ and, for all $0<s<\dist(0,\Phi_\infty(\de\omega'))$,
		\[ \begin{split} \tilde N(x)&\le\lim_{k\to\infty}\frac{\norm{\vfd_{x+r_k\omega'}}(B_{\ell_k s}^\envdim(\Phi(x)))}{\pi(\ell_k s)^2}=\lim_{k\to\infty} \frac{(\Phi_k)_*(\uno_{\omega'}\nu_k)(B_s^\envdim(0))}{\pi s^2} \\
		&=\frac{(\Phi_\infty)_*(\uno_{\omega'}\nu_\infty)(B_s^\envdim(0))}{\pi s^2}, \end{split} \]
		by the definition of $\tilde N$, the monotonicity formula and Lemma \ref{pushforward}.
		We deduce that, with the same notation as in the proof of Lemma \ref{softlemma}, $N'(0)\ge\tilde N(x)$. We also have the converse inequality $N'(0)\le\tilde N(x)$, since the density of $\mu_\infty$ is everywhere at most $C''=\tilde N(x)$. This argument also shows that $\Psi^{-1}(0)=0$ and $\sum_{i=1}^r N'(z_i)\le\tilde N(x_0)$ whenever $z_i\in\Omega_\infty$ are distinct points in a fiber $\Psi^{-1}(p)$ (since $z_i$ contributes by $N'(z_i)$ to the density of $\mu_\infty$ at $p$).

		Finally, 
		arguing as in the proof of Theorem \ref{planarreg}, we conclude that $N'(0)$ is integer since it equals the constant density of a suitable localization of $(\Omega_\infty,\Psi,N_\infty\circ\varphi_\infty^{-1})$ (in an open subset of $\Psi(\Omega_\infty)$).
		This establishes our claim.

		\emph{Step 3.} We show by contradiction that there cannot be any $x_0\in\de\mathcal A_k\cap\mathcal A_k\cap\Omega_k$.
		Indeed, if this happens, then we can find $x_1\in\Omega_k\setminus\mathcal A_k$ with
		\[ \abs{x_1-x_0}<\mz\min\set{\dist(x_0,\mathcal B_k),\dist(x_0,\de\Omega_k)}, \]
		thanks to the fact that the latter is positive, as $\mathcal B_k$ is closed in $\Omega_k$. We infer that
		\[ r:=\dist(x_1,\mathcal{C}_k)<\mz\min\set{\dist(x_0,\mathcal B_k),\dist(x_0,\de\Omega_k)}, \]
		there exists $y_0\in\mathcal{C}_k$ with $\abs{y_0-x_1}=r$ and necessarily we have
		\[ y_0\in\mathcal A_k,\quad B_r^2(x_1)\subseteq\Omega_k\setminus\mathcal{C}_k. \]

		Let $H\subset\C$ be the unique open half-plane with $0\in\de H$ and $B_r^2(x_1)\subseteq y_0+H$. By definition of $\mathcal{C}_k$, we can find a sequence $y_j\to y_0$ with $y_j\in\Omega_k\setminus\Omega_{k-1}$ and $y_j\neq y_0$. We can assume that $\frac{y_j-y_0}{\abs{y_j-y_0}}\to\bar y$. Let $r_j:=\abs{y_j-y_0}$ and set $\ell_j^2:=\int_{B_{r_j}^2(y_0)}\abs{\nabla\Phi}^2\,d\mathcal{L}^2$, $\Phi_j:=\ell_j^{-1}(\Phi(y_0+r_j\cdot)-\Phi(y_0))$,
		$N_j:=N(y_0+r_j\cdot)$. Thanks to Corollary \ref{enwellbeh}, we can apply Theorem \ref{blow-up-c} and obtain, up to subsequences, a limiting local parametrized stationary varifold $(\C,\Phi_\infty\circ\varphi_\infty^{-1},N_\infty\circ\varphi_\infty^{-1})$.

		\begin{center}
			\begin{overpic}[width=11cm]{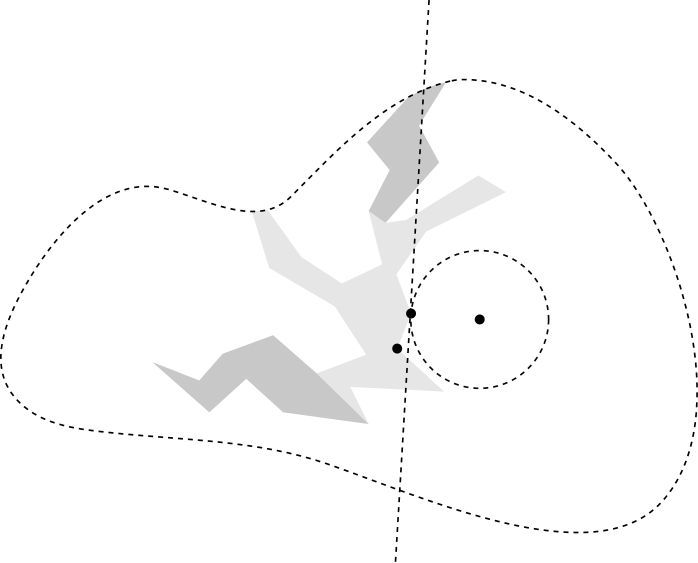}
				\put(19,47){$\Omega_k$}
				\put(51,36.5){$\mathcal{A}_k$}
				\put(55,59.5){$\mathcal{B}_k$}
				\put(36,27.5){$\mathcal{B}_k$}
				\put(64,40){$B_r^2(x_1)$}
				\put(65,75){$y_0+H$}
				\put(53,28){$x_0$}
				\put(70,32.5){$x_1$}
				\put(60.5,34){$y_0$}
			\end{overpic}
		\end{center}

		With the same notation and the same argument used in Step 2, we get $N'(0)=\tilde N(y_0)$. Actually, we also have $N'(\varphi_\infty(\bar y))\ge k\gamma$: letting $\omega''$ be a smooth neighborhood of $\bar y$ with $\Phi_\infty(\bar y)\nin\Phi_\infty(\de\omega'')$ and $\Phi_\infty^{-1}(\Phi_\infty(\bar y))\cap\omega''=\set{\bar y}$, as soon as $\Phi_j(r_j^{-1}(y_j-y_0))\nin\Phi_j(\de\omega'')$ and $r_j^{-1}(y_j-y_0)\in\omega''$ the varifold
		\[ \vfd_{(\omega'',\Phi_j,N_j)}=(\ell_j^{-1}(\cdot-\Phi(y_0)))_*\vfd_{(y_0+r_j\omega'',\Phi,N)} \]
		has density at least $\tilde N(y_j)\ge k\gamma$ at $\ell_j^{-1}(\Phi(y_j)-\Phi(y_0))=\Phi_j(r_j^{-1}(y_j-y_0))\to\Phi_\infty(\bar y)$, has infinitesimal mean curvature in $\R^\envdim\setminus\Phi_j(\de\omega'')$ and its mass measure converges to $(\Phi_\infty)_*(\uno_{\omega''}\nu_\infty)$, by Lemma \ref{pushforward}. Thus, by the monotonicity formula, $N'(\varphi_\infty(\bar y))$ (i.e. the density of this last measure at $\Phi_\infty(\bar y)$) is at least $k\gamma$.

		Since $k\gamma+1>\tilde N(y_0)$, as observed in Step 2 we must have $\Phi_\infty^{-1}(0)=\set{0}$ and $\Phi_\infty^{-1}(\Phi_\infty(\bar y))=\set{\bar y}$.
		Recall that, as in Step 2, the map $\Psi:=\Phi_\infty\circ\varphi_\infty^{-1}$ takes values in a plane and is an entire holomorphic function, up to suitable identification of this plane with $\C$.
		Since it has two values having only one preimage, by Picard's great theorem it does not have an essential singularity at $\infty$ and is thus a polynomial. Actually, Picard's great theorem can be easily avoided: by Corollary \ref{enwellbeh} the Dirichlet energy $\int_{B_R^2(0)}\abs{\nabla\Phi_\infty}^2\,d\mathcal{L}^2$ grows at most polynomially in $R$ and, by inspecting the proof of \cite[Theorem~4.30]{imayoshi} (as well as inequalities (4.21) and (4.24) in \cite{imayoshi}), we see that $\sup_{z\in B_R^2(0)}\abs{\varphi_\infty^{-1}}(z)$ also grows at most polynomially, hence the same is true for $\int_{B_R^2(0)}\abs{\nabla\Psi}^2\,d\mathcal{L}^2$ and thus (by the mean value property for harmonic functions) for $\sup_{z\in B_R^2(0)}\abs{\nabla\Psi}(z)$, i.e. $\Psi$ is a polynomial. Since $\Psi^{-1}(0)=\set{0}$, it must have the form
		\[ \Psi(z)=cz^{\bar k} \]
		for some $\bar k$ and finally the fact that $\Psi^{-1}(\Phi_\infty(\bar y))$ is a singleton gives $\bar k=1$. We deduce that $\Psi'(0)\neq 0$. Arguing as in the proof of Corollary \ref{constantn}, we get $N_\infty\circ\varphi_\infty^{-1}=N'(0)\ge k\gamma$ a.e. near $0$.

		We finally show that $\Phi_j\to\Phi_\infty$ in $W^{1,2}_{loc}(H,\R^\envdim)$. In particular, for any small ball $B\cptsub H$ close enough to $0$, this will contradict the estimate
		\[ \begin{split} &\kappa\gamma\int_B \abs{\de_1\Phi_\infty\wedge\de_2\Phi_\infty}\,d\mathcal{L}^2\le\nu_\infty(B)
		=\lim_{j\to\infty}\nu_j(B) \\
		&=\mz\lim_{j\to\infty}\int_B N_j\abs{\nabla\Phi_j}^2\,d\mathcal{L}^2
		\le\alpha\lim_{j\to\infty}\int_B\abs{\de_1\Phi_j\wedge\de_2\Phi_j}\,d\mathcal{L}^2, \end{split} \]
		where $\alpha$ is the biggest integer smaller than $\kappa\gamma$ (the last inequality comes from the fact that $N_j\in\N$ and $N_j=\tilde N(y_0+r_j\cdot)$ a.e. on $\set{\nabla\Phi_j\neq 0}$, together with the fact that eventually $y_0+r_j B\subseteq B_r^2(x_1)\subseteq\Omega_k\setminus\mathcal{C}_k$).

		Fix any $U\cptsub H$. Since eventually $-\Delta\Phi_j=\ell_j A(\Phi(y_0+r_j\cdot))(\nabla\Phi_j,\nabla\Phi_j)$ on $U$ and the right-hand side converges to $0$ in $L^1(U,\R^\envdim)$, we get
		\[ -\Delta\Phi_\infty=0 \]
		on $U$ and hence (since $U$ was arbitrary) on $H$. Fix any nonnegative $\rho\in C^\infty_c(H)$ with $\rho=1$ on $U$. Setting $\Psi_j:=\Phi_j-\Phi_\infty$, we have $-\Delta\Psi_j=\ell_j A(\Phi(y_0+r_j\cdot))(\nabla\Phi_j,\nabla\Phi_j)$ and thus
		\[ \begin{split} &\int_H\rho\abs{\nabla\Psi_j}^2\,d\mathcal{L}^2+\int_H\Psi_j\cdot\ang{\nabla\rho,\nabla\Psi_j}\,d\mathcal{L}^2=\int_H\ang{\nabla(\rho\Psi_j);\nabla\Psi_j}\,d\mathcal{L}^2 \\
		&=\ell_j\int_H\rho\Psi_j\cdot A(\Phi(y_0+r_j\cdot))(\nabla\Phi_j,\nabla\Phi_j)\,d\mathcal{L}^2. \end{split} \]
		But both the right-hand side and the second term in the left-hand side converge to $0$. Hence,
		\[ \int_U\abs{\nabla\Psi_j}^2\,d\mathcal{L}^2\le\int_H\rho\abs{\nabla\Psi_j}^2\,d\mathcal{L}^2\to 0. \]

		\emph{Step 4.} From the previous step we have $\de\adm_k\cap\Omega_k\subseteq\mathcal B_k$, hence $\adm_k$ is open. Since $\gamma<1$ and $\tilde N$ is integer-valued on $\adm_k$ (by Step 2), $\tilde N$ takes exactly a single value here. We can then apply Theorem \ref{constn} (as replacing $N$ with $\tilde N$ does not affect the stationarity) and obtain that the partial differential equation holds on $\adm_k$.

		\emph{Step 5.} From the two previous steps, it follows that $-\Delta\Phi=A(\Phi)(\nabla\Phi,\nabla\Phi)$ on the open set $\Omega_k\setminus\mathcal{B}_k$. Using Lemma \ref{noblowdimzero} and Lemma \ref{removab}, we deduce that the partial differential equation holds on the whole $\Omega_k$ (and $\Phi$ is $C^\infty$-smooth on $\Omega_k$). This completes the induction.

		\emph{Step 6.} The extra assumptions made at the beginning of the proof can be dropped by arguing as in the proof of Theorem \ref{conicalreg} and using the upper semicontinuity of $\tilde N$.
		For the fact that $\Phi$ is a branched immersion on a connected component $\Omega''$ where it is nonconstant, we refer the reader to the proofs of \cite[Theorems~1~and~2]{hartman} and \cite[Lemmas~2.1~and~2.2]{gulliver}. Finally, let $D\subset\Omega''$ denote the discrete subset where $\nabla\Phi=0$. Whenever a connected $\omega\cptsub\Omega''\setminus D$ is such that $\Phi(\omega)\cap\Phi(\de\omega)=\emptyset$ and $\restr{\Phi}{\omega}$ is an embedding, the constancy theorem (see \cite[Theorem~41.1]{simon}) implies that $\tilde N$ is constant on $\omega$. Since $\Omega''\setminus D$ is connected, Proposition \ref{nisusc} gives that $N$ is a.e. constant on $\Omega''\setminus D$, hence on $\Omega''$.
	\end{proof}

	\begin{corollary}[global regularity]\label{globgenreg}
		If $(\Sigma,\Phi,N)$ is a parametrized stationary varifold, then $\Phi$ solves $-\Delta\Phi=A(\Phi)(\nabla\Phi,\nabla\Phi)$ in local conformal coordinates, $\Phi$ is a $C^\infty$-smooth branched immersion and $N$ is a.e. constant.
	\end{corollary}

	\begin{rmk}[converse statement]\label{converse}
		We notice that the converse statement holds as well: namely, if $\Phi:\Sigma\to\subman$ is a nonconstant weakly conformal, weakly harmonic map and $N$ is a positive integer, then $(\Sigma,\Phi,N)$ is a parametrized stationary varifold. Indeed, for almost every $\omega\subseteq\Sigma$, the continuous representative of $\restr{\Phi}{\de\omega}$ coincides with the trace (by \cite[Theorem~5.7]{evans} this holds whenever $\mathcal{H}^1$-a.e. point of $\de\omega$ is a Lebesgue point for $\Phi$). Thus, for any smooth $F\in C^\infty_c(\R^\envdim\setminus\Phi(\de\omega),\R^\envdim)$, $\restr{F(\Phi)}{\omega}$ has zero trace on $\de\omega$ and \eqref{eq:locstat} follows. Notice that we did not need H\'elein's regularity result to show this assertion: on the contrary, we can immediately deduce the continuity of $\Phi$ (and hence the smoothness) from Proposition \ref{contandsupp}. \hfill\qedsymbol
	\end{rmk}

	\section{An application to the conductivity equation}\label{condsec}

	In this section we illustrate an application of Theorem \ref{genreg} to the regularity theory for the conductivity equation
	\[ -\operatorname{div}(N\nabla\Phi)=0\quad\text{on }B_1^2(0). \]
	This partial differential equation was already investigated by many authors: see e.g. \cite{astala, faraco, leonetti}. We show below that, assuming $\Phi\in W^{1,2}(B_1^2(0),\R^\envdim)$ weakly conformal and $N\in L^\infty(\N\setminus\set{0})$, $\Phi$ is necessarily harmonic and $N$ is a.e. constant, unless $\Phi$ is itself constant.

	This statement initially originated as a possible intermediate step in order to achieve Theorem \ref{genreg}, but as a matter of fact we are able to prove the former only as a consequence of the latter. It would be interesting to find an independent, purely PDE-theoretic proof.

	We can do this in the case $\envdim=2$, where the following slightly stronger result holds.

	\begin{thm}[regularity for $\bm{q=2}$]\label{planarcond}
		Assume $\Phi\in W^{1,2}(B_1^2(0),\R^2)$ is weakly conformal, $N\in L^\infty(B_1^2(0))$ is bounded below by a positive constant and
		\begin{equation}\label{eq:condstat} -\operatorname{div}(N\nabla\Phi)=0\quad\text{in }\mathcal D'(B_1^2(0),\R^2). \end{equation}
		Then $\Phi_1+i\,\Phi_2$ is holomorphic or antiholomorphic and, if $\Phi$ is nonconstant, $N$ is a.e. constant.
	\end{thm}

	\begin{proof}
		First of all, $\Phi$ is continuous (see e.g. \cite[Section~4.4]{giaquinta}). We can assume that $\Phi$ is nonconstant, so that the set $\goodrk\neq\emptyset$ of Lebesgue points $z$ for $\nabla\Phi$ with $\nabla\Phi(z)\neq 0$ is nonempty.
		Notice that $\Phi_1$ and $\Phi_2$ are both nonconstant: 
		otherwise e.g. $\nabla\Phi_1$ would be a.e. $0$ and in particular we would have $\nabla\Phi_1=0$ on $\goodrk$, contradicting the weak conformality.
		From \eqref{eq:condstat} and standard Hodge theory, we can find real functions $\Psi_k\in W^{1,2}(B_1^2(0))$ with
		\[ N\nabla\Phi_k=-\nabla^\perp\Psi_k,\quad\nabla^\perp:=(-\de_2,\de_1), \]
		for $k=1,2$. This equation can be equivalently rewritten as
		\begin{equation}\label{eq:complref} N\de_z\Phi_k=i\de_z\Psi_k\quad\text{or}\quad N \de_{\obar{z}}\Phi_k=-i\de_{\obar{z}}\Psi_k. \end{equation}
		Let $f_k:=\Phi_k+i\,\Psi_k$. We have
	\[ \de_{\obar{z}}f_k=(1-N)\de_{\obar{z}}\Phi_k,\quad\de_{{z}}f_k=(1+N)\de_{{z}}\Phi_k. \]
	We define the Beltrami coefficient $\mu_k$ on $\C$ as
	\[ \mu_k:=\frac{(1-N)\de_{\obar{z}}\Phi_k}{(1+N)\de_{{z}}\Phi_k}\uno_{B_1^2(0)\cap\goodrk}. \]
	Notice that, by weak conformality, $\de_z\Phi_k\neq 0$ on $\goodrk$. Our hypotheses on $N$ clearly imply
	\[
	\|\mu_k\|_{L^\infty}<1
	\]
	and $f_k$ satisfies the Beltrami equation
	\[
	\de_{\obar{z}}f_k=\mu_k\de_z f_k
	\]
	on $B_1^2(0)$. Let $\varphi_k$ be the normal solution of
	\[
	\de_{\obar{z}}\varphi_k=\mu_k\de_z\varphi_k
	\]
	(see \cite[Theorem~4.24]{imayoshi}). As already pointed out in the proof of Theorem \ref{blow-up}, $\varphi_k,\varphi_k^{-1}\in W^{1,2}_{loc}(\C,\C)$ are homeomorphisms of $\C$ mapping negligible sets to negligible sets and
	\[ \de_{\bar w}\varphi_k^{-1}=-(\mu_k\circ\varphi_k^{-1})\bar{\de_w\varphi_k^{-1}}. \]
	By the chain rule (see \cite[Lemma~III.6.4]{lehto}), $h_k:=f_k\circ\varphi_k^{-1}$ (defined on $\varphi_k(B_1^2(0))$) satisfies
	\[ \de_{\bar w}h_k=(\de_z f_k\circ\varphi_k^{-1})\de_{\bar w}\varphi_k^{-1}+(\de_{\bar z}f_k\circ\varphi_k^{-1})\bar{\de_w\varphi_k^{-1}}=0 \]
	and is thus a nonconstant holomorphic function. Pick now any $z_0\in B_1^2(0)$ such that the points $\varphi_1(z_0)$ and $w_0:=\varphi_2(z_0)$ satisfy $h_1'(\varphi_1(z_0))\neq 0$ and $h_2'(w_0)\neq 0$: by holomorphicity of $h_1$ and $h_2$, this holds true for all $z_0$ outside a discrete, relatively closed subset $D\subset B_1^2(0)$.

	By the Cauchy--Riemann equations, the harmonic map $\Phi_2\circ\varphi_2^{-1}=\Re h_2$ has nonzero differential at $w_0$. By the inverse function theorem, there exists a local chart $\psi$ centered at $w_0$, with some ball $B_\delta^2(0)$ as its image, such that
	\begin{equation}\label{eq:localchart} \Phi_2\circ\varphi_2^{-1}\circ\psi^{-1}(y)-\Phi_2(z_0)=y_2\quad\text{on }B_\delta^2(0). \end{equation}

	Since $\Phi$ is weakly conformal, we have $(\de_z\Phi_1)^2+(\de_z\Phi_2)^2=0$ a.e., hence there exists a measurable function $\varepsilon\in L^\infty(B_1^2(0),\set{-1,1})$ such that
	\begin{equation}
	\label{eq:segno}
	\de_z\Phi_1(z)=i\varepsilon(z)\de_z\Phi_2(z)\quad\text{a.e. on }B_1^2(0).
	\end{equation}
	Combining \eqref{eq:complref} and \eqref{eq:segno} we obtain
	\[ \nabla\Psi_1=\varepsilon N\nabla\Phi_2 \]
	a.e. Using \eqref{eq:localchart} and the chain rule again, we get
	\[ \de_1(\Psi_1\circ\varphi_2^{-1}\circ\psi^{-1})=0,\quad \de_2(\Psi_1\circ\varphi_2^{-1}\circ\psi^{-1})=(\varepsilon N)\circ\varphi_2^{-1}\circ\psi^{-1} \]
	a.e. and, since $\de_{12}^2(\Psi_1\circ\varphi_2^{-1}\circ\psi^{-1})=0$ distributionally, we deduce that
	\[ (\epsilon N)\circ\varphi_2^{-1}\circ\psi^{-1}(y)=\de_2(\Psi_1\circ\varphi_2^{-1}\circ\psi^{-1})(y)=g(y_2) \]
	a.e. on $B_\delta^2(0)$, for a suitable $g\in L^\infty((-\delta,\delta))$, as is immediately verified e.g. by mollification.

	Let $G$ be a Lipschitz primitive of $g$ on $(-\delta,\delta)$. We have
	\[ \nabla(\Psi_1\circ\varphi_2^{-1}\circ\psi^{-1})=\nabla(G(y_2))\quad\text{on }B_\delta^2(0), \]
	so up to subtracting a constant from $G$ we obtain
	\[ \Psi_1\circ\varphi_2^{-1}\circ\psi^{-1}(y)=G(y_2) \]
	and finally, using \eqref{eq:localchart},
	\begin{equation}\label{eq:psifctphi} \Psi_1=G\circ(\Phi_2-\Phi_2(z_0))\quad\text{on }\varphi_2^{-1}\circ\psi^{-1}(B_\delta^2(0))\ni z_0. \end{equation}
	Since $f_1=\Phi_1+i\Psi_1$ is injective in a neighborhood of $z_0$ (being $h_1'(\varphi_1(z_0))\neq 0$), from \eqref{eq:psifctphi} we deduce that $\Phi$ is injective on some neighborhood $B_\eta^2(z_0)$.

	We claim that, as a consequence, $\det(\nabla\Phi)$ has a constant sign on $B_\eta^2(z_0)\cap\goodrk$: indeed, the induced map
	\[ \Phi_*:H_1(\de B_r^2(z))\to H_1(\C\setminus\set{\Phi(z)}) \]
	is clearly independent of $z\in B_\eta^2(z_0)$ and $0<r<\dist(z,\de B_\eta^2(z_0))$, once the two groups are canonically identified with $\Z$. But, for any $z\in B_\eta^2(z_0)\cap\goodrk$, applying Lemma \ref{slicingleb} to $\Phi-\Phi(z)-\ang{\nabla\Phi(z),\cdot-z}$ we can find a small radius $r$ such that $\restr{\Phi}{\de B_r^2(z)}$ is homotopic to $\Phi(z)+\ang{\nabla\Phi(z),\cdot-z}$ (as a map $\de B_r^2(z)\to\C\setminus\set{\Phi(z)}$). Thus the above map $\Phi_*$ coincides with the multiplication by $\sgn\det(\nabla\Phi(z))$ and our claim follows.

	From this fact and the weak conformality assumption, on $B_\eta^2(z_0)$ either $(\Phi_1,\Phi_2)$ or $(\Phi_1,-\Phi_2)$ satisfy the Cauchy--Riemann equations.
	In particular, $\Phi$ is real analytic on the connected set $B_1^2(0)\setminus D$. Since locally we have either $\de_{\bar z}(\Phi_1+i\Phi_2)=0$ or $\de_{z}(\Phi_1+i\Phi_2)=0$, by analyticity $\Phi_1+i\Phi_2$ is globally holomorphic or antiholomorphic on $B_1^2(0)\setminus D$, hence also on $B_1^2(0)$. Finally, \eqref{eq:condstat} gives
	\[ \de_1\Phi_k\de_1 N+\de_2\Phi_k\de_2 N=0\quad\text{in }\mathcal{D}'(B_1^2(0)) \]
	for $k=1,2$ (as $\Delta\Phi_k=0$). Since $\nabla\Phi_1$ and $\nabla\Phi_2$ are smooth and linearly independent outside a closed discrete set, we infer $\nabla N=0$ here and thus, by connectedness, $N$ is a.e. constant.
	\end{proof}

	We now prove the result for arbitrary $\envdim$.

	\begin{thm}[regularity for $\bm{q>2}$]\label{condthm}
		Assume $\Phi\in W^{1,2}(B_1^2(0),\R^\envdim)$ is weakly conformal, $N\in L^\infty(B_1^2(0),\N\setminus\set{0})$ and
		\[ -\operatorname{div}(N\nabla\Phi)=0\quad\text{in }\mathcal D'(B_1^2(0),\R^\envdim). \]
		Then $\Delta\Phi=0$ and, if $\Phi$ is nonconstant, $N$ is a.e. constant.
	\end{thm}

	\begin{proof}
		As in the previous proof, we notice that $\Phi$ is continuous. We can assume that $\Phi$ is not constant. The triple $(B_1^2(0),\Phi,N)$ satisfies Definition \ref{locpardef} (with $\subman=\R^\envdim$), except possibly for the technical condition \eqref{eq:massass}: indeed, for any $\omega\cptsub\B_1^2(0)$ and any $F\in C^\infty_c(\R^\envdim\setminus\Phi(\de\omega),\R^\envdim)$, $F(\Phi)\uno_\omega$ lies in $W^{1,2}_0(B_1^2(0),\R^\envdim)$ and thus $\int_\omega N\ang{\nabla (F(\Phi));\nabla\Phi}\,d\mathcal{L}^2=0$.

		Assume, without loss of generality, that $\Phi_1$ is not constant and let $\Psi_1$, $f_1$, $\varphi_1$ and $h_1$ be the functions constructed as in the preceding proof. Let $D\subset\varphi_1(B_1^2(0))$ be the discrete set of points where $h_1'=0$, or equivalently (by the Cauchy--Riemann equations) where $\nabla(\Phi_1\circ\varphi_1^{-1})=0$. From the chain rule and the fact that $\varphi_1$ and $\varphi_1^{-1}$ map negligible sets to negligible sets (see \cite[Lemma~III.6.4]{lehto} and \cite[Lemma~4.12]{imayoshi}) we deduce that, for a.e. $x\in B_1^2(0)$,
		$\nabla\varphi_1(x)$ is invertible, $\nabla(\Phi_1\circ\varphi_1^{-1})(\varphi_1(x))\neq 0$ and $\nabla\Phi_1(x)$ is given by the composition of these differentials. Hence, $\nabla\Phi\neq 0$ a.e. and thus has full rank a.e. (by weak conformality).

		Let $S\subseteq B_1^2(0)$ denote the complement of the biggest open subset where $\Delta\Phi=0$.
		We remark that, given $x\in B_1^2(0)$, if there exists a neighborhood $\omega\cptsub B_1^2(0)$ with $\Phi(x)\cap\Phi(\de\omega)=\emptyset$ then $x\nin S$: indeed, $\vfd_\omega$ is stationary in $\R^\envdim\setminus\Phi(\de\omega)$, so by the monotonicity formula $\vfd_\omega$ satisfies \eqref{eq:massass} for $p\in B_\rho^\envdim(\Phi(x))$ and $s<\rho$, where $\rho:=\mz\dist(\Phi(x),\Phi(\de\omega))$. Thus, replacing $\omega$ with $\omega\cap\Phi^{-1}(B_{\rho/2}^2(\Phi(x)))$, the triple $(\omega,\Phi,N)$ is a local parametrized stationary varifold and Theorem \ref{genreg} gives $\Delta\Phi=0$ near $x$. In particular, arguing as in the proof of Theorem \ref{conicalreg}, we infer that $x\nin S$ for any Lebesgue point $x$ for $\nabla\Phi$ with $\nabla\Phi(x)\neq 0$. Being $\nabla\Phi\neq 0$ a.e., we get that $\mathcal{L}^2(S)=0$ and that $\Phi$ is nonconstant on any ball outside $S$.

		Moreover, $N$ is a.e. constant on every connected component $U$ of $B_1^2(0)\setminus S$, since on $U$ we have $\de_1\Phi_k\de_1 N+\de_2\Phi_k\de_2 N=0$ for all $k=1,\dots,\envdim$ and the (classical) differential $\nabla\Phi$ has full rank except for a discrete set (being $\Phi$ a nonconstant harmonic function on $U$), which does not disconnect $U$.

		It suffices to show that $S\subseteq\varphi_1^{-1}(D)$, since then any point of $S$ is a removable singularity and thus $S=\emptyset$. Assume by contradiction that there exists a point $x_0\in S\setminus\varphi_1^{-1}(D)$.
		Let $\psi:V\to(-1,1)^2$ be a local chart centered at $\varphi_1(x_0)$ and such that
		\begin{equation}\label{eq:localchart2} \Phi_1\circ\varphi_1^{-1}\circ\psi^{-1}(y)=\Phi_1(x_0)+y_2\quad\text{on }(-1,1)^2. \end{equation}
		Let $\Xi:=\Phi\circ\varphi_1^{-1}\circ\psi^{-1}$. We claim that the negligible set $S':=\psi(V\cap\varphi_1(S))$, relatively closed in $(-1,1)^2$, has the following property: if $y=(y_1,y_2)\in S'$, then $\Xi^{-1}(\Xi(y))$ contains either $\pabra{-1,y_1}\times\set{y_2}$ or $\brapa{y_1,1}\times\set{y_2}$. If this were not true, we could find $a\in(-1,y_1)$ and $b\in(y_1,1)$ with $\Xi(a,y_2),\Xi(b,y_2)\neq\Xi(y)$. Given any $\epsilon<1-\abs{y_2}$, by \eqref{eq:localchart2} we would have
		$\Xi(y)\nin\Xi(\de([a,b]\times[y_2-\epsilon,y_2+\epsilon]))$. But, as remarked earlier, this would imply $y\nin S'$.

		This horizontal segment, i.e. either $\pabra{-1,y_1}\times\set{y_2}$ or $\brapa{y_1,1}\times\set{y_2}$, has to be contained in $S'$ (being $\Xi$ locally injective at points outside $S'$, with at most countably many exceptions). As a consequence we have $\mathcal{L}^1(\set{t:(t,y_2)\in S'})>0$ and thus, by Fubini's theorem and $\mathcal{L}^2(S')=0$, we infer
		\begin{equation}\label{eq:linee} ((-1,1)\times\set{t})\cap S'=\emptyset\quad\text{for a.e. }t\in(-1,1). \end{equation}

		Pick now any $\mu>0$ with $((-1,1)\times\set{\mu})\cap S'=\emptyset$.
		Recall that $(0,0)\in S'$ and let $\lambda:=\max\set{t<\mu:(0,t)\in S'}\ge 0$. From \eqref{eq:linee} it follows that the open set $((-1,1)\times(\lambda,\mu))\setminus S'$ is connected. We infer that $N$ is a.e. constant on $\varphi_1^{-1}\circ\psi^{-1}((-1,1)\times(\lambda,\mu))$, hence $\Delta\Phi=0$ on this open set. However, this contradicts the weak conformality assumption: let
		\[ \omega:=\varphi_1^{-1}\circ\psi^{-1}\pa{\pa{-\mz,\mz}\times\pa{\lambda,\frac{\lambda+\mu}{2}}}, \]
		on which $\Delta\Phi=0$, and take any homeomorphism $\upsilon:\bar\omega\to\bar B_1^2(0)$ biholomorphic on $\omega$ (using Riemann's mapping theorem and \cite[Theorem~I.3.1]{garnett}).
		We have $\Phi_1\circ\upsilon^{-1}>\lambda$ on $B_1^2(0)$, as well as $\Phi_1\circ\upsilon^{-1}=\lambda$ and $\Phi\circ\upsilon^{-1}$ constant on some open arc $A$ in $\de B_1^2(0)$. Since $\Delta(\Phi\circ\upsilon^{-1})=0$ on $B_1^2(0)$, by standard regularity theory $\Phi\circ\upsilon^{-1}$ is smooth in a neighborhood of $A$ in $\bar B_1^2(0)$. But, by Hopf's maximum principle, the radial derivative of $\Phi_1\circ\upsilon^{-1}$ is nonzero on $A$. So, by weak conformality of $\Phi\circ\upsilon^{-1}$, the tangential derivative of $\Phi\circ\upsilon^{-1}$ does not vanish on $A$. This contradicts the fact that $\Phi\circ\upsilon^{-1}$ is constant on $A$.
	\end{proof}

	\appendix
	\section*{Appendix}
	\renewcommand{\thesection}{A}
	\setcounter{definition}{0}
	\setcounter{equation}{0}

	This appendix collects some useful general facts about Sobolev functions on the plane. Most of them are standard results, with the possible exception of Lemmas \ref{oneinfty} and \ref{oneinftyaux}; however, we preferred to include the proofs since a precise reference for these statements is not immediately available in the literature.

	\begin{lemmaen}[essential image of a Sobolev map]\label{essim}
		Let $\Omega\subseteq\C$ be an open connected set.
		Let $\Psi\in W^{1,1}_{loc}(\Omega,\R^\envdim)$ and assume that any point in a nonempty measurable set $E\subseteq\Omega$ is a Lebesgue point for $\Psi$, as well as $\nabla\Psi=0$ a.e. on $\Omega\setminus E$. Then the essential image of $\Psi$ equals $\bar{\Psi(E)}$. 
	\end{lemmaen}

	We recall that the essential image of $\Psi$ is the (closed) set of values $p\in\R^\envdim$ such that $\mathcal{L}^2(\Psi^{-1}(B_s^\envdim(p)))>0$ for all $s>0$, or equivalently it is the support of the positive measure $\Psi_*(\uno_\Omega\mathcal{L}^2)$.

	\begin{proof}
		Fix $x\in E$, $s>0$ and let $p:=\Psi(x)$. By Chebyshev's inequality
		\[ \mathcal{L}^2(B_r^2(x)\setminus\Psi^{-1}(B_s^\envdim(p)))
		\le s^{-1}\int_{B_r^2(x)}\abs{\Psi-\Psi(x)}\,d\mathcal{L}^2=o(r^2), \]
		so $p$ belongs to the essential image. We deduce that $\bar{\Psi(E)}$ is included in the essential image. Conversely, assume that $B_s^\envdim(p)$ is disjoint from $\bar{\Psi(E)}$. We can find a function $\rho\in C^\infty(\R^\envdim,\R^\envdim)$ with $\rho=\operatorname{id}$ on $\R^\envdim\setminus B_s^\envdim(p)$ and $\rho=\operatorname{id}+e_1$ on $B_{s/2}^\envdim(p)$.
		By the chain rule,
		\[ \nabla(\rho\circ\Psi)=(\nabla\rho\circ\Psi)\nabla\Psi=\nabla\Psi \]
		a.e. on $\Omega$, since $\nabla\Psi=0$ a.e. on $\set{\nabla\rho\circ\Psi\neq\operatorname{id}}$.
		Thus, $\rho\circ\Psi-\Psi$ is constant a.e. But this function vanishes on $E$, hence $0$ belongs to its essential image (by the same argument used above). Thus $\rho\circ\Psi=\Psi$ a.e. and we infer that
		$\Psi^{-1}(B_{s/2}^\envdim(p))\subseteq\set{\rho\circ\Psi\neq\Psi}$ is negligible. This shows that $p$ does not belong to the essential image.
	\end{proof}

	\begin{lemmaen}[rectifiability of the image]\label{rectif}
		Let $\Omega\subseteq\C$ be open. If $\Psi\in W^{1,1}_{loc}(\Omega,\R^\envdim)$ and $\good$ denotes the set of Lebesgue points for both $\Psi$ and $\nabla\Psi$, then there exist Lebesgue measurable sets $E_1,E_2,\dots$ such that $\good=\bigcup_i E_i$ and $\restr{\Psi}{E_i}$ is Lipschitz.
	\end{lemmaen}

	\begin{proof}
		We set $F_j:=\big\{x\in\good\cap\Omega_{1/j}:\media_{B_r^2(x)}\abs{\Psi(y)-\Psi(x)}\,dy\le jr\text{ for }0<r\le j^{-1}\big\}$ (where $\Omega_\delta:=\{x\in\Omega:\dist(x,\de\Omega)>\delta\}$), which is a Lebesgue measurable set. By \cite[Theorem~6.1]{evans} (and its proof), for all $x\in\good$
		\[ \begin{split} \media_{B_r^2(x)}\abs{\Psi(y)-\Psi(x)}\,dy&\le r\abs{\nabla\Psi(x)}+\media_{B_r^2(x)}\abs{\Psi(y)-\Psi(x)-\ang{\nabla\Psi(x),y-x}}\,dy \\
		&=r\abs{\nabla\Psi(x)}+o(r). \end{split} \]
		We infer that $\good=\bigcup_{j\ge 1} F_j$. Assume now that $x,x'\in F_j$ are distinct and $\abs{x-x'}\le\frac{1}{2j}$. Let $r:=\abs{x-x'}$ and notice that $B_r^2(x)\subseteq B_{2r}^2(x')$. We can estimate
		\[ \begin{split} \abs{\Psi(x)-\Psi(x')}&\le\media_{B_r^2(x)}\abs{\Psi(y)-\Psi(x)}\,dy
		+\media_{B_r^2(x)}\abs{\Psi(y)-\Psi(x')}\,dy \\
		&\le\media_{B_r^2(x)}\abs{\Psi(y)-\Psi(x)}\,dy+4\media_{B_{2r}^2(x')}\abs{\Psi(y)-\Psi(x)}\,dy \\
		&\le jr+8jr=9j\abs{x-x'} \end{split} \]
		(since $2r\le j^{-1}$). The result follows once we split each $F_j$ into countably many subsets of diameter at most $\frac{1}{2j}$.
	\end{proof}

	\begin{lemmaen}[small diameter for suitable slices]\label{slicing}
		Assume $\Psi\in W^{1,2}(B_{(1+\tau)r}^2(0),\R^\envdim)$, with $r,\tau>0$. Then there is a measurable subset $E\subseteq(r,(1+\tau)r)$ of positive measure such that, for all $r'\in E$, $\restr{\Psi}{\de B_{r'}^2(0)}$ has an absolutely continuous representative whose image satisfies
		\[ \diam\Psi(\de B_{r'}^2(0))\le\bigg(\frac{2\pi(1+\tau)}{\tau}\int_{B_{(1+\tau)r}^2(0)}\abs{\nabla\Psi}^2\,d\mathcal{L}^2\bigg)^{1/2}. \]
	\end{lemmaen}

	\begin{proof}
		For a.e. $r'\in (r,(1+\tau)r)$ the function $\theta\mapsto\Psi(r'\cos\theta,r'\sin\theta)$ has an absolutely continuous representative, whose weak derivative is $r'\ang{\nabla\Psi(r'\cos\theta,r'\sin\theta),(-\sin\theta,\cos\theta)}$ (see \cite[Theorem~4.21]{evans}). Moreover,
		\[ \media_r^{r+\tau r}\int_{\de B_t^2(0)}\abs{\nabla\Psi}^2\,d\mathcal{H}^1\,dt\le\frac{1}{\tau r}\int_{B_{(1+\tau)r}^2(0)}\abs{\nabla\Psi}^2\,d\mathcal{L}^2. \]
		Hence, for a set of radii $r'$ of positive measure, the inner integral is less or equal than the right-hand side, giving
		\[ \bigg(\int_{\de B_{r'}^2(0)}\abs{\nabla\Psi}\,d\mathcal{H}^1\bigg)^2\le 2\pi r'\int_{\de B_{r'}^2(0)}\abs{\nabla\Psi}^2\,d\mathcal{H}^1\le\frac{2\pi(1+\tau)}{\tau}\int_{B_{(1+\tau)r}^2(0)}\abs{\nabla\Psi}^2\,d\mathcal{L}^2. \qedhere \]
	\end{proof}

	\begin{lemmaen}[vanishing first order on suitable slices]\label{slicingleb}
		Assume that $0$ is a Lebesgue point for $\Psi\in W^{1,2}(B_R^2(0),\R^\envdim)$ and $\nabla\Phi$, as well as $\Psi(0)=\nabla\Psi(0)=0$. Then, for any $\tau,\epsilon>0$ and any $0<r\le\bar r$, we can select $r'\in(r,(1+\tau)r)$ such that $\restr{\Psi}{\de B_{r'}^2(0)}$ has an absolutely continuous representative with
		\[ \max_{y\in S^1}\abs{\Psi(r'y)}\le\epsilon r, \]
		for a suitable $\obar r=\obar r(\Psi,\tau,\epsilon)$.
	\end{lemmaen}

	\begin{proof}
		We can assume $\epsilon<1$. By \cite[Theorem~6.1]{evans} and its proof we have
		\begin{equation}\label{eq:o4est} \int_{B_r^2(0)}|\Psi|^2\,d\mathcal{L}^2=o(r^4) \end{equation}
		(since $1^*=2$) and, as a consequence,
		\[ \begin{split} &\media_r^{r+\tau r}\int_{\de B_t^2(0)}((2\pi t)^{-1}|\Psi|^2+r|\nabla\Psi|)\,d\mathcal{H}^1\,dt \\
		&\le \frac{1}{2\pi\tau r^2}\int_{B_{(1+\tau)r}^2(0)}|\Psi|^2\,d\mathcal{L}^2+\frac{1}{\tau}\int_{B_{(1+\tau)r}^2(0)}|\nabla\Psi|\,d\mathcal{L}^2=o(r^2). \end{split} \]
		Hence, if $r$ is small enough, there exists some $r'\in(r,r+\tau r)$ such that
		\[ (2\pi r')^{-1}\norm{\Psi}_{L^2(\de B_{r'}^2(0))}^2+r\norm{\nabla\Psi}_{L^1(\de B_{r'}^2(0))}\le\frac{\epsilon^2 r^2}{4} \]
		and $\restr{\Psi}{\de B_{r'}^2(0)}$ is absolutely continuous (a.e.), with derivative given by the chain rule.
		The elementary inequality
		\[ \norm{\Psi}_{L^\infty(\de B_{r'}^2(0))}\le\bigg|\media_{\de B_{r'}^2(0)}\Psi\,d\mathcal{H}^1\bigg|+\int_{\de B_{r'}^2(0)}\abs{\nabla\Psi}\,d\mathcal{H}^1\le\frac{\norm{\Psi}_{L^2(\de B_{r'}^2(0))}}{\sqrt{2\pi r'}}+\norm{\nabla\Psi}_{L^1(\de B_{r'}^2(0))} \]
		gives the desired result.
	\end{proof}

	\begin{lemmaen}[strong subconvergence on a.e. slice]\label{equibded}
		If $\Psi_k\weakto\Psi_\infty$ in $W^{1,2}(B_R^2(0),\R^\envdim)$, then for a.e. $r'\in(0,R)$ there exists a subsequence $(k_i)$ such that
		$\restr{\Psi_\infty}{\de B_{r'}^2(0)}$ and $\restr{\Psi_{k_i}}{\de B_{r'}^2(0)}$ have absolutely continuous representatives and
		\[ \restr{\Psi_{k_i}}{\de B_{r'}^2(0)}\to\restr{\Psi_\infty}{\de B_{r'}^2(0)}\quad\text{in }L^\infty. \]
	\end{lemmaen}

	\begin{proof}
		We observe that $\restr{\Psi_\infty}{\de B_{r'}^2(0)}$ coincides with the trace of $\Psi_\infty$ on $\de B_{r'}^2(0)$ for a.e. $r'\in(0,R)$: actually, this happens whenever $\mathcal{H}^1$-a.e. point on $\de B_{r'}^2(0)$ is a Lebesgue point of $\Psi_\infty$ (see \cite[Theorem~5.7]{evans}); the same holds for $\Psi_k$. Now, by Fatou's lemma,
		\[ \begin{split} &\int_0^R\liminf_{k\to\infty}\int_{\de B_{r'}^2(0)}(\abs{\Psi_k}^2+\abs{\nabla\Psi_k}^2)\,d\mathcal{H}^1\,dr' \\
		&\le\liminf_{k\to\infty}\int_0^R\int_{\de B_{r'}^2(0)}(\abs{\Psi_k}^2+\abs{\nabla\Psi_k}^2)\,d\mathcal{H}^1\,dr' \\
		&=\liminf_{k\to\infty}\norm{\Psi_k}_{W^{1,2}}^2<\infty. \end{split} \]
		Hence, for a.e. $r'\in(0,R)$,
		\[ \liminf_{k\to\infty}\int_{\de B_{r'}^2(0)}(\abs{\Psi_k}^2+\abs{\nabla\Psi_k}^2)\,d\mathcal{H}^1<\infty, \]
		which means that there exists a subsequence $(k_i)$ such that
		\[ \sup_i\int_{\de B_{r'}^2(0)}(\abs{\Psi_{k_i}}^2+\abs{\nabla\Psi_{k_i}}^2)\,d\mathcal{H}^1<\infty. \]
		We can also assume that $\restr{\Psi_\infty}{\de B_{r'}^2(0)}$ equals the trace, that it has a $W^{1,2}$ representative with weak derivative given by the chain rule and that the analogous statements hold also for $\Psi_{k_i}$.

		In particular, the sequence $\pa{\restr{\Psi_{k_i}}{\de B_{r'}^2(0)}}$ is bounded in $W^{1,2}(\de B_{r'}^2(0),\R^\envdim)$: by the compact embedding into $C^0(\de B_{r'}^2(0),\R^\envdim)$ we deduce that (up to subsequences) $\restr{\Psi_{k_i}}{\de B_{r'}^2(0)}\to f$ in $L^\infty$, for some $f\in C^0(\de B_{r'}^2(0),\R^\envdim)$. By the weak continuity of the trace, we have $\restr{\Psi_{k_i}}{\de B_{r'}^2(0)}\weakto\restr{\Psi_\infty}{\de B_{r'}^2(0)}$ in $L^2(\de B_{r'}^2(0),\R^\envdim)$, hence $f=\restr{\Psi_\infty}{\de B_{r'}^2(0)}$ $\mathcal{H}^1$-a.e.
	\end{proof}

	\begin{lemmaen}[almost planar conformal matrices]\label{cptconfineq}
		Let $\mathcal{C}\subseteq\R^{\envdim\times 2}$ denote the set of matrices $M$ with $\sum_{i=1}^\envdim M_{ij}M_{ik}=\frac{|M|^2}{2}\delta_{jk}$, where $|M|$ is the Hilbert--Schmidt norm of $M$ ($\mathcal{C}$ can be identified with the set of linear weakly conformal maps $\R^2\to\R^\envdim$). For any $\tau>0$ there exists $C=C(\tau,\envdim)$ such that
		\[ |M_{11}M_{21}+M_{12}M_{22}|+|M_{11}^2+M_{12}^2-J(M)|+|M_{21}^2+M_{22}^2-J(M)|\le\tau|M|^2+C\sum_{i=3}^\envdim\sum_{j=1}^2 M_{ij}^2 \]
		for all $M\in\mathcal{C}$, where $J(M):=\abs{M_{11}M_{22}-M_{12}M_{21}}$.
	\end{lemmaen}

	\begin{proof}
		Assume by contradiction that, for a sequence $(M^k)\subseteq\mathcal{C}$, we have
		\[ \begin{split} &|M_{11}^k M_{21}^k+M_{12}^k M_{22}^k|+|(M_{11}^k)^2+(M_{12}^k)^2-J(M^k)|+|(M_{21}^k)^2+(M_{22}^k)^2-J(M^k)| \\
		&>\tau|M^k|^2+k\sum_{i=3}^\envdim\sum_{j=1}^2 (M_{ij}^k)^2. \end{split} \]
		By homogeneity we can assume $|M^k|=1$ for all $k$. As a consequence, $\sum_{i=3}^\envdim\sum_{j=1}^2 (M_{ij}^k)^2\to 0$. So the sequence has a limit point $M^\infty\in\R^{\envdim\times 2}$ satisfying
		\begin{equation}\label{eq:plconf} M^\infty\in\mathcal{C},\qquad M_{ij}^\infty=0\quad\text{for }i=3,\dots,\envdim\text{ and }j=1,2, \end{equation}
		\[ |M_{11}^\infty M_{21}^\infty+M_{12}^\infty M_{22}^\infty|+|(M_{11}^\infty)^2+(M_{12}^\infty)^2-J(M^\infty)|+|(M_{21}^\infty)^2+(M_{22}^\infty)^2-J(M^\infty)|\ge\tau. \]
		But conditions \eqref{eq:plconf} force $\begin{pmatrix}M_{11}^\infty & M_{12}^\infty \\ M_{21}^\infty & M_{22}^\infty\end{pmatrix}=\begin{pmatrix}a & \mp b \\ b & \pm a\end{pmatrix}$ for some $a,b\in\R$,
		so the left-hand side of the last inequality vanishes, yielding the desired contradiction.
	\end{proof}

	\begin{lemmaen}[$\bm{L^{1,\infty}}$-variance estimate]\label{oneinfty}
		For any $\epsilon>0$ there exists a $C=C(\epsilon)$ such that
		\[ \begin{split} &\big\|f-(f)_{B_{r/2}^2(0)}\big\|_{L^{1,\infty}(B_{r/2}^2(0))}\le \epsilon\big\|f-(f)_{B_{3r/4}^2(0)}\big\|_{L^1(B_{3r/4}^2(0))} \\
		&+C\sup\Big\{\Big|\int_{B_r^2(0)} f\nabla\phi\,d\mathcal{L}^2\Big|\,;\ \phi\in C^\infty_c(B_r^2(0)),\norm{\nabla\phi}_{L^\infty}\le 1\Big\}, \end{split} \]
		for all $f\in L^1(B_r^2(0))$, where $(f)_{B_s^2(0)}:=\media_{B_s^2(0)}f\,d\mathcal{L}^2$.
	\end{lemmaen}

	\begin{proof} By scaling-invariance and homogeneity, we can assume that $r=1$ and that the supremum in the right-hand side equals $1$. Thanks to Hahn--Banach theorem applied to the functional
		\[ (\de_1\phi,\de_2\phi)\mapsto\int_{B_1^2(0)} f\de_i\phi\,d\mathcal{L}^2,\quad\text{on }\set{(\de_1\phi,\de_2\phi)\mid\phi\in C_c^\infty(B_1^2(0))}\subseteq C_c^0(B_1^2(0))^2, \]
		we can find finite Borel measures $g_{ij}\in\mathcal{M}(B_1^2(0))$ (for $i,j=1,2$) with
		\begin{align}\label{eq:sistmis} \norm{g_{ij}}(B_1^2(0))\le 1,\qquad\de_i f=\sum_j\de_j g_{ij}\quad\text{in }\mathcal{D}'(B_1^2(0)). \end{align}
		By convolution with a mollifier $\rho_\delta$, we can assume $f,g_{ij}\in C^\infty(B_{3/4}^2(0))$ and that \eqref{eq:sistmis} holds in the ball $B_{3/4}^2(0)$: once we show that \eqref{eq:sistmis}, together with the requirement that $\big|\int_{B_{3/4}^2(0)} f\nabla\phi\,d\mathcal{L}^2\big|\le 1$ whenever $\phi\in C^\infty_c(B_{3/4}^2(0))$ and $\norm{\nabla\phi}_{L^\infty}\le 1$, imply
		\begin{align}\label{eq:diffthesis} \big\|f-(f)_{B_{1/2}^2(0)}\big\|_{L^{1,\infty}(B_{1/2}^2(0))}\le \epsilon\big\|f-(f)_{B_{3/4}^2(0)}\big\|_{L^1(B_{3/4}^2(0))}+C \end{align}
		for the mollified $f$, the result follows for the original function by letting $\delta\to 0$ (notice that the above requirement holds if $\delta$ is small enough).
		We choose a cut-off function $\chi\in C^\infty_c(B_{3/4}^2(0))$ such that $\chi=1$ on a neighborhood of $\obar B_{1/2}^2(0)$. Notice that
		\[ \de_i(\chi f)=\sum_j\de_j(\chi g_{ij})-\sum_j\de_j\chi g_{ij}+\de_i\chi f, \]
		so we get the following distributional identity in $\R^2$:
		\[ \begin{split} (1-\Delta)(\chi f)&=(1-\Delta)\chi f-\sum_{i,j}\de_{ij}^2(\chi g_{ij})+\sum_{i,j}\de_{ij}^2\chi g_{ij}+\sum_{i,j}\de_j\chi\de_i g_{ij}-\sum_i\de_i\chi\de_i f \\
		&=(1-\Delta)\chi f-\sum_{i,j}\de_{ij}^2(\chi g_{ij})+\sum_{i,j}\de_{ij}^2\chi g_{ij}+\sum_{i,j}\de_j\chi\de_i g_{ij}-\sum_{i,j}\de_i\chi\de_j g_{ij} \\
		&=(1-\Delta)\chi f-\sum_{i,j}\de_{ij}^2(\chi g_{ij})+\sum_{i,j}\de_{ij}^2\chi g_{ij}+\sum_{i,j}\de_i(\de_j\chi g_{ij})-\sum_{i,j}\de_j(\de_i\chi g_{ij}). \end{split} \]
		Applying $(1-\Delta)^{-1}$ to both sides and observing that the operators
		$(1-\Delta)^{-1}$, $(1-\Delta)^{-1}\nabla$ and $(1-\Delta)^{-1}\nabla^2$ correspond to H\"ormander--Mikhlin Fourier multipliers, we obtain
		\[ \chi f=(1-\Delta)^{-1}((1-\Delta)\chi f)+r, \]
		where the Schwartz function $r$ satisfies $\norm{r}_{L^{1,\infty}(\R^2)}\le C'$ for some absolute constant $C'$.
		The desired estimate \eqref{eq:diffthesis} now follows from Lemma \ref{oneinftyaux} below.
	\end{proof}

	\begin{lemmaen}[$\bm{L^1}$-estimate for a commutator]\label{oneinftyaux}
		For any $\epsilon>0$ there exists a $C=C(\epsilon)$ such that
		\[ \begin{split} &\big\|(1-\Delta)^{-1}((1-\Delta)\chi f)-(f)_{B_{1/2}^2(0)}\big\|_{L^1(B_{1/2}^2(0))}\le \epsilon\big\|f-(f)_{B_{3/4}^2(0)}\big\|_{L^1(B_{3/4}^2(0))} \\
		&+C\sup\Big\{\Big|\int_{B_{3/4}^2(0)} f\nabla\phi\,d\mathcal{L}^2\Big|\,;\  \phi\in C^\infty_c(B_{3/4}^2(0)),\norm{\nabla\phi}_{L^\infty}\le 1\Big\}, \end{split} \]
		for all $f\in C^\infty(B_{3/4}^2(0))\cap L^1(B_{3/4}^2(0))$.
	\end{lemmaen}

	\begin{proof}
		We can assume $(f)_{B_{3/4}^2(0)}=0$ and $\norm{f}_{L^1(B_{3/4}^2(0))}=1$.
		If the statement is false, we can find a sequence $(f_n)$ of such functions with
		\begin{equation}\label{eq:nooneinftyaux} \begin{split} &\big\|(1-\Delta)^{-1}((1-\Delta)\chi f_n)-(f_n)_{B_{1/2}^2(0)}\big\|_{L^1(B_{1/2}^2(0))} \\
		&\ge \epsilon+n\sup\Big\{\Big|\int_{B_{3/4}^2(0)} f_n\nabla\phi\,d\mathcal{L}^2\Big|\,;\ \phi\in C^\infty_c(B_{3/4}^2(0)),\norm{\nabla\phi}_{L^\infty}\le 1\Big\}. \end{split} \end{equation}
		We remark that, for any $h\in C^\infty_c(B_{3/4}^2(0))$,
		\begin{equation}\label{eq:kernbessel} (1-\Delta)^{-1}h=K*h\quad\text{on }B_{1/2}^2(0),\qquad K:=1_{B_2^2(0)}\mathcal{F}^{-1}((1+4\pi^2|\xi|^2)^{-1})\in L^1(\R^2). \end{equation}
		In particular, thanks to \eqref{eq:kernbessel}, the left-hand side of \eqref{eq:nooneinftyaux} is bounded by a constant and we deduce $\int_{B_{3/4}^2(0)} f_n\nabla\phi\,d\mathcal{L}^2\to 0$ for all $\phi\in C^\infty_c(B_{3/4}^2(0))$. There exists a subsequence $(f_{n_k})$ such that $f_{n_k}\mathcal{L}^2\weakstarto\mu$, for some $\mu\in\mathcal{M}(B_{3/4}^2(0))$. From the identity $\int_{B_{3/4}^2(0)}\nabla\phi\,d\mu=0$ for all $\phi\in C^\infty_c(B_{3/4}^2(0))$ we infer $\mu=c\uno_{B_{3/4}^2(0)}\mathcal{L}^2$ (with $c$ a real constant).
		Hence, letting
		\[ h_k:=(1-\Delta)\chi f_{n_k},\quad h_\infty:=c(1-\Delta)\chi, \]
		we have $h_k\mathcal{L}^2\weakstarto h_\infty\mathcal{L}^2$ in $\mathcal{M}(\R^2)$ and, since $\mu(\de B_{1/2}^2(0))=0$, we deduce that $(f_{n_k})_{B_{1/2}^2(0)}\to c$.
		We also have
		\begin{equation}\label{eq:convcpt} \psi*h_k\to\psi*h_\infty\quad\text{in }C^0(\R^2), \end{equation}
		for any $\psi\in C^0_c(B_2^2(0))$: if this were not true, up to further subsequences we could find $\delta>0$ and $(x_k)\subseteq B_3^2(0)$ such that
		\begin{equation}\label{eq:convnoncpt} \abs{\psi*h_k(x_k)-\psi*h_\infty(x_k)}\ge\delta \end{equation}
		and we could also assume that $x_k\to x_\infty\in\bar B_3^2(0)$; but $\psi(x_k-\cdot)\to\psi(x_\infty-\cdot)$ uniformly, so
		\[ \psi*h_k(x_k)=\int_{\R^2}\psi(x_k-\cdot)\,d(h_k\mathcal{L}^2)
		\to\int_{\R^2}\psi(x_\infty-\cdot)\,d(h_\infty\mathcal{L}^2)=\psi*h_\infty(x_\infty), \]
		while trivially $\psi*h_\infty(x_k)\to\psi*h_\infty(x_\infty)$, contradicting \eqref{eq:convnoncpt} for $k$ large enough.

		Moreover, $K*h_k\to K*h_\infty$ in $L^1(\R^2)$, as is readily seen from \eqref{eq:convcpt} by approximating $K$ with functions $\psi\in C^0_c(B_2^2(0))$ in the $L^1$-topology.
		Using \eqref{eq:kernbessel}, it follows that
		\[ 1_{B_{1/2}^2(0)}(1-\Delta)^{-1}h_k\to 1_{B_{1/2}^2(0)}(1-\Delta)^{-1}h_\infty= 1_{B_{1/2}^2(0)}(1-\Delta)^{-1}(c(1-\Delta)\chi)=c1_{B_{1/2}^2(0)}  \]
		in $L^1(\R^2)$. However, for $k$ large enough, this contradicts the inequality
		\[ \big\|(1-\Delta)^{-1}((1-\Delta)\chi f_{n_k})-(f_{n_k})_{B_{1/2}^2(0)}\big\|_{L^1(B_{1/2}^2(0))}\ge \epsilon. \qedhere \]
	\end{proof}

	\nocite{*}
	\bibliographystyle{plainnat}

\end{document}